\documentclass[a4paper,11pt]{article}
\usepackage{graphicx}
\usepackage{amssymb}
\usepackage[french, english]{babel}
\usepackage[utf8]{inputenc}
\usepackage{amsmath,amsfonts,amssymb,amsthm}
\usepackage[T1]{fontenc}
\usepackage{fullpage}
\usepackage{hyperref}
\usepackage{mathrsfs}
\usepackage{relsize}
\usepackage{tikz}
\usepackage{bbm}
\usepackage{appendix}
\usepackage{pifont}
\usetikzlibrary{patterns}

\definecolor{darkgreen}{rgb}{0.0, 0.5, 0.0}

\newcommand{\cmark}{\ding{51}}
\newcommand{\xmark}{\ding{55}}

\newcommand{\N}{\mathbb N}
\newcommand{\Z}{\mathbb Z}

\newcommand{\R}{\mathbb R}

\newcommand{\E}{\mathbb E}

\renewcommand{\P}{\mathbb P}

\newcommand{\cau}{\mathcal{C}}
\newcommand{\hcau}{\mathcal{H}}
\newcommand{\bcau}{\mathcal{B}}
\newcommand{\sli}{\mathcal{S}}
\newcommand{\uu}{\mathcal{U}}
\newcommand{\expl}{\mathcal{E}}
\newcommand{\eff}{\mathrm{eff}}
\newcommand{\back}{\mathbf{B}}
\newcommand{\f}{\mathcal{F}}
\newcommand{\bdd}{\mathrm{bdd}}
\newcommand{\str}{\mathrm{str}}

\newcommand{\m}{\mathcal{M}}
\newcommand{\bdy}{\widehat{\partial} \mathbf{T}}
\newcommand{\bd}{\widehat{\partial}}

\theoremstyle{definition}

\newtheorem{thm}{Theorem}

\newtheorem{defn}{Definition}
\newtheorem{rem}[defn]{Remark}

\newtheorem{prop}[defn]{Proposition}

\newtheorem{lem}[defn]{Lemma}

\tikzstyle{every node}=[circle, draw, fill=black!50, inner sep=0pt, minimum width=4pt]
\tikzstyle{small}=[circle, draw, fill=black!50, inner sep=0pt, minimum width=3pt]
\tikzstyle{white}=[circle, draw, fill=black!0, inner sep=0pt, minimum width=4pt]
\tikzstyle{bigwhite}=[circle, draw, fill=black!0, inner sep=0pt, minimum width=10pt]
\tikzstyle{dual}=[circle, draw=blue, fill=black!0, inner sep=0pt, minimum width=4pt]
\tikzstyle{fat}=[circle, draw, fill=red!50, inner sep=0pt, minimum width=8pt]
\tikzstyle{fat_bis}=[circle, draw, fill=blue!50, inner sep=0pt, minimum width=8pt]
\tikzstyle{fat_ter}=[circle, draw, fill=green!50, inner sep=0pt, minimum width=8pt]
\tikzstyle{rouge}=[circle, draw, fill=red, inner sep=0pt, minimum width=7pt]
\tikzstyle{bleu}=[circle, draw, fill=blue, inner sep=0pt, minimum width=7pt]
\tikzstyle{moyenrouge}=[circle, draw, fill=red, inner sep=0pt, minimum width=6pt]
\tikzstyle{petitrouge}=[circle, draw, fill=red, inner sep=0pt, minimum width=4pt]
\tikzstyle{petitbleu}=[circle, draw, fill=blue, inner sep=0pt, minimum width=4pt]
\tikzstyle{texte}=[draw=none, fill=none]

\title{\bf{Supercritical causal maps : geodesics and simple random walk}}
\author{Thomas \bsc{Budzinski} \footnote{ENS Paris and Université Paris-Saclay, \url{thomas.budzinski@ens.fr}}}

\begin{document}

\maketitle

\begin{abstract}
We study the random planar maps obtained from supercritical Galton--Watson trees by adding the horizontal connections between successive vertices at each level. These are the hyperbolic analog of the maps studied by Curien, Hutchcroft and Nachmias in \cite{CHN17}, and a natural model of random hyperbolic geometry. We first establish metric hyperbolicity properties of these maps: we show that they admit bi-infinite geodesics and satisfy a weak version of Gromov-hyperbolicity. We also study the simple random walk on these maps: we identify their Poisson boundary and, in the case where the underlying tree has no leaf, we prove that the random walk has positive speed. Some of the methods used here are robust, and allow us to obtain more general results about planar maps containing a supercritical Galton--Watson tree.
\end{abstract}

\begin{figure}[!h]
\begin{center}
\includegraphics[scale=0.151]{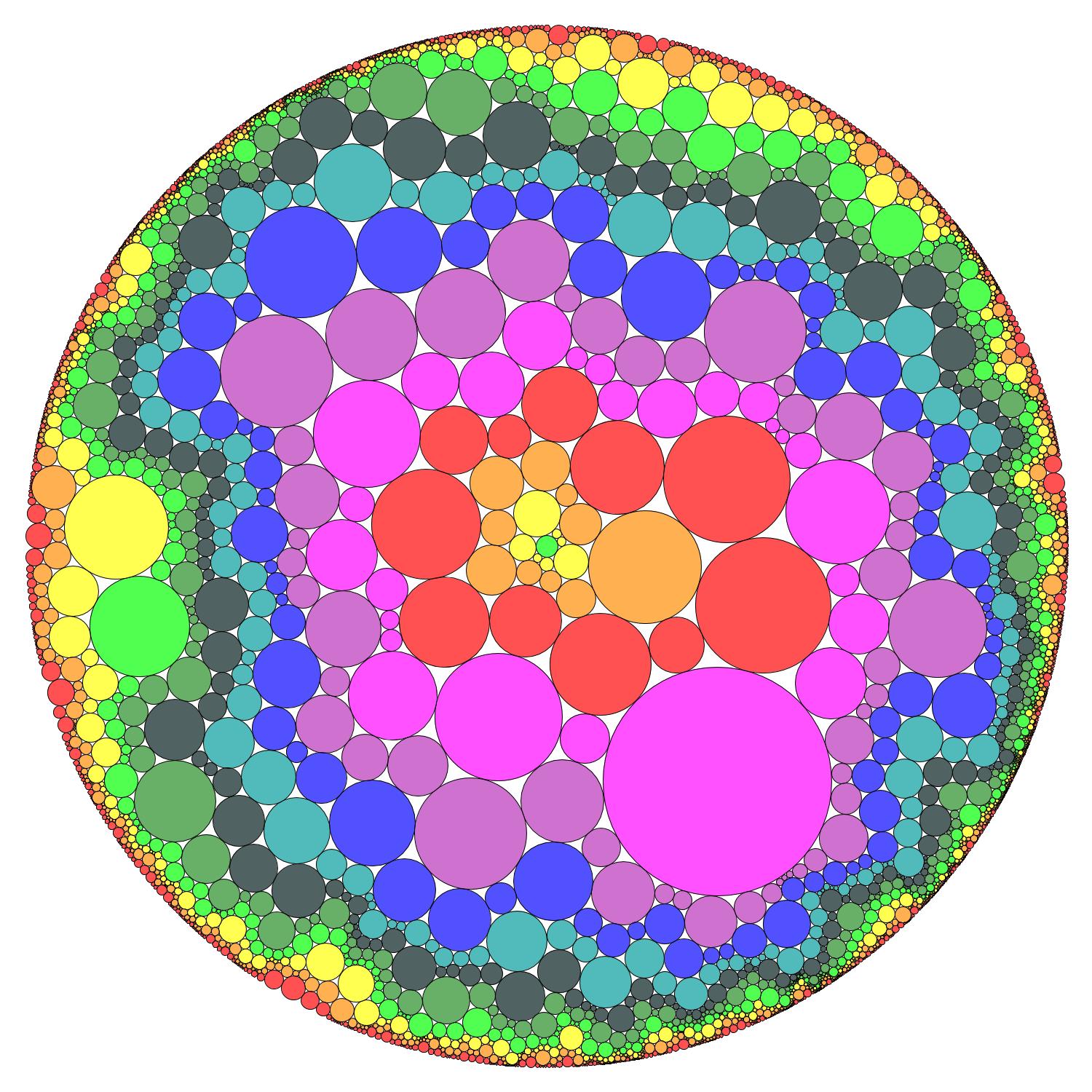}
\end{center}
\caption{The circle packing of a causal triangulation constructed from a Galton--Watson tree with geometric offspring distribution of mean $3/2$. This was made with the help of the software CirclePack by Ken Stephenson.}
\end{figure}

\section*{Introduction}

\paragraph{Causal maps and random hyperbolic geometry.}
\emph{Causal triangulations} were introduced by theoretical physicists Ambj\o rn and Loll \cite{AL98}, and have been the object of a lot of numerical investigations. However, their rigorous study is quite recent \cite{DJW10, CHN17}. They are a discrete model of Lorentzian quantum gravity with one time and one space dimension where, in contrast with uniform random planar maps, time and space play asymmetric roles.

Here is the definition of the model. For any (finite or infinite) plane tree $t$, we denote by $\cau(t)$ the planar map obtained from $t$ by adding at each level the horizontal connections between consecutive vertices, as on Figure \ref{fig_def_causal} (this includes an horizontal connection between the leftmost and rigtmost vertices at each level). Our goal here is to study the graph $\cau(T)$, where $T$ is a supercritical Galton--Watson tree conditioned to survive.

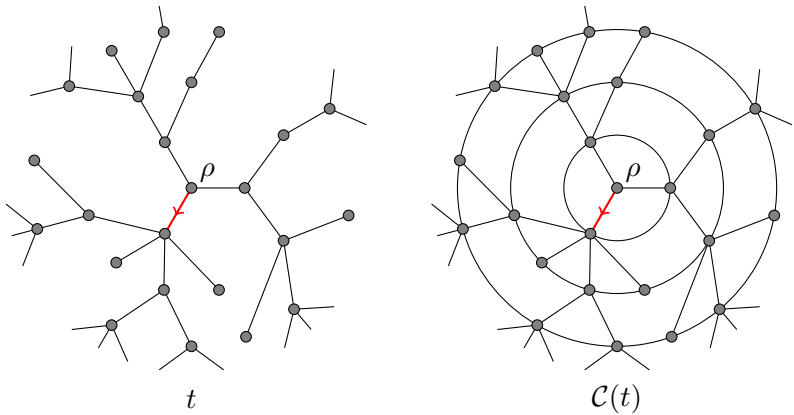
\begin{figure}[h]
\begin{center}
\begin{tikzpicture}[scale=0.7]

\draw(0,0)--(0:1);
\draw(0,0)--(120:1);
\draw[thick, red, ->](0,0)--(240:0.6);
\draw[thick, red](240:0.5)--(240:1);
\draw(0:1)--(30:2);
\draw(0:1)--(330:2);
\draw(120:1)--(90:2);
\draw(120:1)--(120:2);
\draw(240:1)--(195:2);
\draw(240:1)--(225:2);
\draw(240:1)--(255:2);
\draw(240:1)--(285:2);
\draw(30:2)--(30:3);
\draw(330:2)--(350:3);
\draw(330:2)--(310:3);
\draw(330:2)--(290:3);
\draw(90:2)--(80:3);
\draw(120:2)--(100:3);
\draw(120:2)--(120:3);
\draw(120:2)--(140:3);
\draw(195:2)--(170:3);
\draw(195:2)--(195:3);
\draw(255:2)--(240:3);
\draw(255:2)--(270:3);
\draw(310:3)--(300:3.5);
\draw(310:3)--(310:3.5);
\draw(310:3)--(320:3.5);
\draw(30:3)--(20:3.5);
\draw(30:3)--(40:3.5);
\draw(100:3)--(100:3.5);
\draw(140:3)--(130:3.5);
\draw(140:3)--(150:3.5);
\draw(195:3)--(185:3.5);
\draw(195:3)--(195:3.5);
\draw(195:3)--(205:3.5);
\draw(240:3)--(230:3.5);
\draw(240:3)--(240:3.5);
\draw(240:3)--(250:3.5);
\draw(270:3)--(260:3.5);
\draw(270:3)--(280:3.5);

\draw(0,0) node{};
\draw(0:1) node{};
\draw(120:1) node{};
\draw(240:1) node{};
\draw(330:2) node{};
\draw(30:2) node{};
\draw(90:2) node{};
\draw(120:2) node{};
\draw(195:2) node{};
\draw(225:2) node{};
\draw(255:2) node{};
\draw(285:2) node{};
\draw(290:3) node{};
\draw(310:3) node{};
\draw(350:3) node{};
\draw(30:3) node{};
\draw(80:3) node{};
\draw(100:3) node{};
\draw(120:3) node{};
\draw(140:3) node{};
\draw(170:3) node{};
\draw(195:3) node{};
\draw(240:3) node{};
\draw(270:3) node{};

\draw(0,-4) node[texte]{$t$};
\draw(0.3,0.3) node[texte]{$\rho$};

\begin{scope}[shift={(8,0)}]
\draw(0,0)--(0:1);
\draw(0,0)--(120:1);
\draw[thick, red, ->](0,0)--(240:0.6);
\draw[thick, red](240:0.5)--(240:1);
\draw(0:1)--(30:2);
\draw(0:1)--(330:2);
\draw(120:1)--(90:2);
\draw(120:1)--(120:2);
\draw(240:1)--(195:2);
\draw(240:1)--(225:2);
\draw(240:1)--(255:2);
\draw(240:1)--(285:2);
\draw(30:2)--(30:3);
\draw(330:2)--(350:3);
\draw(330:2)--(310:3);
\draw(330:2)--(290:3);
\draw(90:2)--(80:3);
\draw(120:2)--(100:3);
\draw(120:2)--(120:3);
\draw(120:2)--(140:3);
\draw(195:2)--(170:3);
\draw(195:2)--(195:3);
\draw(255:2)--(240:3);
\draw(255:2)--(270:3);
\draw(310:3)--(300:3.5);
\draw(310:3)--(310:3.5);
\draw(310:3)--(320:3.5);
\draw(30:3)--(20:3.5);
\draw(30:3)--(40:3.5);
\draw(100:3)--(100:3.5);
\draw(140:3)--(130:3.5);
\draw(140:3)--(150:3.5);
\draw(195:3)--(185:3.5);
\draw(195:3)--(195:3.5);
\draw(195:3)--(205:3.5);
\draw(240:3)--(230:3.5);
\draw(240:3)--(240:3.5);
\draw(240:3)--(250:3.5);
\draw(270:3)--(260:3.5);
\draw(270:3)--(280:3.5);

\draw(0,0) circle (1cm);
\draw(0,0) circle (2cm);
\draw(0,0) circle (3cm);

\draw(0,0) node{};
\draw(0:1) node{};
\draw(120:1) node{};
\draw(240:1) node{};
\draw(330:2) node{};
\draw(30:2) node{};
\draw(90:2) node{};
\draw(120:2) node{};
\draw(195:2) node{};
\draw(225:2) node{};
\draw(255:2) node{};
\draw(285:2) node{};
\draw(290:3) node{};
\draw(310:3) node{};
\draw(350:3) node{};
\draw(30:3) node{};
\draw(80:3) node{};
\draw(100:3) node{};
\draw(120:3) node{};
\draw(140:3) node{};
\draw(170:3) node{};
\draw(195:3) node{};
\draw(240:3) node{};
\draw(270:3) node{};

\draw(0,-4) node[texte]{$\cau(t)$};
\draw(0.3,0.3) node[texte]{$\rho$};
\end{scope}

\end{tikzpicture}
\end{center}
\caption{An infinite plane tree $t$ and the associated causal map $\mathcal{C}(t)$. The edge in red joins the root vertex to its leftmost child.}\label{fig_def_causal}
\end{figure}

This defines a new model of random "hyperbolic" graph. Several other such models have been investigated so far, such as supercritical Galton--Watson trees \cite{LPP95}, Poisson--Voronoi tesselations of the hyperbolic plane \cite{BPP14}, or the Planar Stochastic Hyperbolic Infinite Triangulations (PSHIT) of \cite{CurPSHIT}. Many notions appearing in the study of these models are adapted from the study of Cayley graphs of nonamenable groups, and an important idea is to find more general versions of the useful properties of these Cayley graphs. Let us mention two such tools.

\begin{itemize}
\item[$\bullet$]
For example, the three aforementioned models are all \emph{stationary}, which means their distribution is invariant under rerooting along the simple random walk\footnote{This is not exactly true for Galton--Watson trees, but it is true for the closely related \emph{augmented Galton--Watson trees.}}. This property generalizes the transitivity of Cayley graphs, and is a key tool to prove positive speed for the simple random walk on supercritical Galton--Watson trees \cite{LPP95, Aid11} or on the PSHIT \cite{CurPSHIT}. More generally, in the context of stationary random graphs, general relations are known between the exponential growth rate, the speed of the random walk, and its asymptotic entropy, which is itself related to the Poisson boundary and the Liouville property. See \cite[Proposition 3.6]{BCstationary}, which adapts classical results about Cayley graphs \cite{KV83}. On the other hand, supercritical causal maps are not stationary, and it seems hard to find a stationary environment for the simple random walk\footnote{See for example \cite{LW18} for the particular case where the tree is the complete binary tree: the existence and uniqueness of a stationary environment are proved, but it is very difficult to say anything explicit about the distribution of this environment.}. In absence of stationarity, we will be forced to use other properties of our graphs such as the independence properties given by the structure of Galton--Watson trees.
\item[$\bullet$]
Another important property in the study of random hyperbolic graphs is anchored expansion, which is a weaker version of nonamenability, and may be thought of as a natural generalization of nonamenability to random graphs. It is known to imply positive speed and heat kernel decay bounds of the form $\exp(-n^{1/3})$ for bounded-degree graphs \cite{Vir00}. This property also played an important role in the study of non-bounded-degree graphs such as Poisson-Voronoi tesselations of the hyperbolic plane \cite{BPP14}, and the half-planar versions of the PSHIT \cite{ANR14}. However, we have not been able to establish this property for causal maps, and need once again to use other methods.
\end{itemize}

\paragraph{Supercritical causal maps.}
In all that follows, we fix an offspring distribution $\mu$ with $\sum_{i=0}^{\infty} i \mu(i)>1$. Note that we do not require the mean number of children to be finite. We denote by $T$ a Galton--Watson tree with offspring distribution $\mu$ conditioned to survive.
The goal of this work is to study the maps $\cau(T)$. We will study both large-scale metric properties of $\cau(T)$, and the simple random walk on this map. All the results that we will prove show that $\cau(T)$ has a hyperbolic flavour, which is also true for the tree $T$.

\paragraph{Metric hyperbolicity properties.}
The first goal of this work is to establish two metric hyperbolicity properties of $\cau(T)$. We recall that a graph $G$ is hyperbolic in the sense of Gromov if there is a constant $k \geq 0$ such that all the triangles are $k$-thin in the following sense. Let $x$, $y$ and $z$ be three vertices of $G$ and $\gamma_{xy}$, $\gamma_{yz}$, $\gamma_{zx}$ be geodesics from $x$ to $y$, from $y$ to $z$ and from $z$ to $x$. Then for any vertex $v$ on $\gamma_{xy}$, the graph distance between $v$ and $\gamma_{yz} \cup \gamma_{zx}$ is at most $k$. However, such a strong, uniform statement usually cannot hold for random graphs. For example, if $\mu(1)>0$, then $\cau(T)$ contains arbitrarily large portions of the square lattice, which is not hyperbolic. Therefore, we suggest a weaker, "anchored" definition\footnote{The most natural definition would be to require that any geodesic triangle surrounding the root is $k$-thin, but this is still too strong (consider the triangle formed by root vertex and two vertices $x,y$ in a large portion of square lattice).}.

\begin{defn}
Let $M$ be a rooted planar map. We say that $M$ is \emph{weakly anchored hyperbolic} if there is a constant $k \geq 0$ such that the following holds. Let $x$, $y$ and $z$ be three vertices of $M$ and $\gamma_{xy}$ (resp. $\gamma_{yz}$, $\gamma_{zx}$) be a geodesic from $x$ to $y$ (resp. $y$ to $z$, $z$ to $x$). Assume the triangle formed by $\gamma_{xy}$, $\gamma_{yz}$ and $\gamma_{zx}$ surrounds the root vertex $\rho$. Then
\[ d_M \left( \rho, \gamma_{xy} \cup \gamma_{yz} \cup \gamma_{zx} \right) \leq k.\]
\end{defn}

\begin{thm}[Metric hyperbolicity of $\cau(T)$]\label{thm_1_metric}
Let $T$ be a supercritical Galton--Watson tree conditioned to survive, and let $\cau(T)$ be the associated causal map.
\begin{enumerate}
\item
The map $\cau(T)$ is a.s.~weakly anchored hyperbolic.
\item
The map $\cau(T)$ a.s.~admits bi-infinite geodesics, i.e. paths $\left( \gamma(i) \right)_{i \in \Z}$ such that for any $i$ and $j$, the graph distance between $\gamma(i)$ and $\gamma(j)$ is exactly $|i-j|$.
\end{enumerate}
\end{thm}

These two results are very robust and hold in a much more general setting that includes the PSHIT. In particular, the second point for the PSHIT answers a question of Benjamini and Tessera \cite{BT17}. More general results are discussed in the end of this introduction.

\paragraph{Poisson boundary.}
The second goal of this work is to study the simple random walk on $\cau(T)$ and to identify its Poisson boundary. First note that $\cau(T)$ contains as a subgraph the supercritical Galton--Watson tree $T$, which is transient, so $\cau(T)$ is transient as well.
We recall the general definition of the Poisson boundary. Let $G$ be an infinite, locally finite graph,  and let $G \cup \partial G$ be a compactification of $G$, i.e. a compact metric space in which $G$ is dense. Let also $(X_n)$ be the simple random walk on $G$ started from $\rho$. We say that $\partial G$ is a \emph{realization of the Poisson boundary} of $G$ if the following two properties hold:
\begin{itemize}
\item
$(X_n)$ converges a.s.~to a point $X_{\infty} \in \partial G$,
\item
every bounded harmonic function $h$ on $G$ can be written in the form
\[h(x)=\mathbb{E}_x \left[ g \left( X_{\infty} \right) \right],\]
where $g$ is a bounded measurable function from $\partial G$ to $\R$.
\end{itemize}

We denote by $\partial T$ the space of infinite rays of $T$. If $\gamma, \gamma' \in \partial T$, we write $\gamma \sim \gamma'$ if $\gamma=\gamma'$ or if $\gamma$ and $\gamma'$ are two "consecutive" rays in the sense that there is no ray between them. Then $\sim$ is a.s.~an equivalence relation for which countably many equivalence classes have cardinal $2$ and all the others have cardinal $1$. We write $\widehat{\partial} T=\partial T / \sim$. There is a natural way to equip $\cau(T) \cup \widehat{\partial} T$ with a topology that makes it a compact space. We refer to Section \ref{causal_subsec_construct_poisson} for the construction of this topology, but we mention right now that $\widehat{\partial} T$ is homeomorphic to the circle, whereas $\partial T$ is homeomorphic to a Cantor set. The space $\cau(T) \cup \widehat{\partial} T$ can be seen as a compactification of the infinite graph $\cau(T)$. We show that this is a realization of its Poisson boundary.

\begin{thm}[Poisson boundary of $\cau(T)$]\label{thm_2_Poisson}
Almost surely:
\begin{enumerate}
\item
the limit $\lim (X_n)=X_{\infty}$ exists and its distribution has full support and no atoms in $\widehat{\partial} T$,
\item
$\widehat{\partial} T$ is a realization of the Poisson boundary of $\cau(T)$.
\end{enumerate}
\end{thm}

Note that, by a result of Hutchcroft and Peres \cite{HP15}, the second point will follow from the first one.

\paragraph{Positive speed.}
A natural and strong property shared by many models of hyperbolic graphs is the positive speed of the simple random walk. See for example \cite{LPP95} for supercritical Galton--Watson trees, and \cite{CurPSHIT, ANR14} for the PSHIT or their half-planar analogs. The third goal of this work is to prove that the simple random walk on $\cau(T)$ has a.s.~positive speed. Unfortunately, we have only been able to prove it in the case where $\mu(0)=0$, i.e. when the tree $T$ has no leaf. We recall that $(X_n)$ is the simple random walk on $\cau(T)$, and denote by $d_{\cau(T)}$ the graph distance on $\cau(T)$.
\begin{thm}[Positive speed on $\cau(T)$]\label{thm_3_positive}
If $\mu(0)=0$ and $\mu(1)<1$, then there is $v_{\mu}>0$ such that
\[\frac{d_{\mathcal{C}(T)} (\rho, X_n)}{n} \xrightarrow[n \to +\infty]{\mathrm{a.s.}} v_{\mu}.\]
\end{thm}
However, we expect to still have positive speed if $\mu(0)=0$. As mentioned above, this result is not obvious because of the lack of stationarity (for stationary graphs, the results of \cite{BCstationary} show that positive speed is equivalent to being non-Liouville under some mild assumptions).

\paragraph{The critical case.}
We note that similar properties have been studied in the critical case in \cite{CHN17}. The results of \cite{CHN17} show that the geometric properties of causal maps are closer to those of uniform random maps than to those of the trees from which they were built. This contrasts sharply with the supercritical case, where the properties of the causal map are very close to those of the associated tree. More precisely, in the finite variance case, the distance between vertices at some fixed height $r$ is $o(r)$, but $r^{1-o(1)}$. Moreover, the exponents describing the behaviour of the simple random walk are the same as for the square lattice, and different from the exponents we would obtain in a tree.

\paragraph{Robustness of the results and applications to other models.}
Another motivation to study causal maps is that many other models of random planar maps can be obtained by adding connections (and, in some cases, vertices) to a random tree. For example, the UIPT \cite{AS03} or its hyperbolic variants the PSHIT \cite{CurPSHIT} can be constructed from a reverse Galton--Watson tree or forest via the Krikun decomposition \cite{Kri04,  CLGmodif, B18}. Among all the maps that can be obtained from a tree $t$ in such a way that the branches of the tree remain geodesics, the causal map is the one with the "closest" connections, which makes it a useful toy model. The causal map may even provide general bounds for any map obtained from a fixed tree (we will see such applications in this paper, see also \cite{C18+} for applications to uniform planar maps via the Krikun decomposition).

Here, the causal maps $\cau(T)$ fit in a more general framework. We define a \emph{strip} as an infinite, one-ended planar map $s$ with exactly one infinite face, such that the infinite face has a simple boundary $\partial s$, and equipped with a root vertex on the boundary on the infinite face. If $t$ is an infinite tree with no leaf and $(s_i)_{i \in \N}$ is a sequence of strips, let $\m \left( t, (s_i) \right)$ be the map obtained by filling the (infinite) faces of $t$ with the strips $s_i$ (see Section \ref{causal_sec_setting} for a more careful construction). Some of our results can be generalized to random maps of the form $\m \left( \mathbf{T}, (s_i) \right)$, where $\mathbf{T}$ is a supercritical Galton--Watson tree with no leaf, and the $s_i$ are strips (which may depend on $\mathbf{T})$.

By the backbone decomposition for supercritical Galton--Watson trees (that we recall in Section \ref{causal_sec_setting}), the maps of the form $\cau(T)$, where $T$ is a supercritical Galton--Watson tree with leaves, are a particular case of this construction. The results of \cite{B18} prove that the PSHIT $\mathbb{T}_{\lambda}$ can also be obtained by this construction: the tree $\mathbf{T}$ is then the tree of infinite leftmost geodesics of $\mathbb{T}_{\lambda}$ and has geometric offspring distribution. 

We will show that Theorem \ref{thm_1_metric} is very robust and applies to this general context, see Theorem \ref{thm_1_bis}. A particular case of interest are the PSHIT. In particular, point 2 of Theorem \ref{thm_1_metric} for the PSHIT answers a question of Benjamini and Tessera \cite{BT17}.

As for causal maps, any map of the form $\m \left( \mathbf{T}, (s_i) \right)$ contains the transient graph $\mathbf{T}$, so it is transient itself.
Most of our proof of Theorem \ref{thm_2_Poisson} can also be adapted to the general setting where the strips $s_i$ are i.i.d. and independent of $\mathbf{T}$. However, Theorem \ref{thm_2_Poisson} cannot be true if the strips $s_i$ are too large (for example if themselves have a non-trivial Poisson boundary). On the other hand, we can still show that the Poisson boundary is non-trivial. See Theorem \ref{thm_2_bis} for a precise statement, and Figure \ref{summary_conjectures} for a summary of the results proved in this paper and the results left to prove.

As we will see later, Theorem \ref{thm_2_bis} is not strictly speaking more general than Theorem \ref{thm_2_Poisson}, since the strips used to construct $\cau(T)$ from the backbone of $T$ are not completely independent.
On the other hand, once again, the PSHIT satisfy these assumptions (up to a root transformation, since the strip containing the root has a slightly different distribution). However, it was already known that the PSHIT are non-Liouville (see \cite{CurPSHIT}, or \cite{AHNR15} for another identification of the Poisson boundary via circle packings). We also prove in \cite{B18}, by a specific argument based on the peeling process, that $\widehat{\partial} \mathbf{T}$ is indeed a realization of the Poisson boundary in the case of the PSHIT.

\begin{figure}
\begin{center}
\begin{tabular}{|l|c|c|c|}
\hline
& non-Liouville & $\widehat{\partial} \mathbf{T}$ is the Poisson boundary & positive speed\\
\hline
$T$ & \color{darkgreen}{\cmark} & \color{darkgreen}{\cmark} & \color{darkgreen}{\cmark} \\
\hline
$\cau(T)$ (if $\mu(0)=0$) & \color{darkgreen}{\cmark} & \color{darkgreen}{\cmark} & \color{darkgreen}{\cmark} \\
\hline
$\cau(T)$ (if $\mu(0)>0$) & \color{darkgreen}{\cmark} & \color{darkgreen}{\cmark} & \color{blue}{\textbf{?}} \\
\hline
PSHIT & \color{darkgreen}{\cmark} & \color{darkgreen}{\cmark} & \color{darkgreen}{\cmark} \\
\hline
$\underset{\mbox{bounded-degree}}{\mbox{$\m$ with $(S_i)$ i.i.d., recurrent,}}$ & \color{darkgreen}{\cmark} & \color{darkgreen}{\cmark} & \color{red}{\xmark} \\
\hline
$\m$ with $(S_i)$ i.i.d., recurrent & \color{darkgreen}{\cmark} & \color{blue}{\textbf{?}} & \color{red}{\xmark} \\
\hline
$\m$ with $(S_i)$ i.i.d. & \color{darkgreen}{\cmark} & \color{red}{\xmark} & \color{red}{\xmark} \\
\hline
general $\m$ & \color{red}{\xmark} & \color{red}{\xmark} & \color{red}{\xmark} \\
\hline
\end{tabular}
\end{center}
\caption{The symbol {\color{darkgreen}{\cmark}} means that the property is proved in an earlier work or in this one. The symbol {\color{blue}{\textbf{?}}} indicates properties that we believe to be true but did not prove in this paper, and the symbol {\color{red}{\xmark}} means the property is false in general. See Section \ref{causal_sec_counter} for a quick description of some counterexamples.}\label{summary_conjectures}
\end{figure}

\paragraph{Structure of the paper.}
The paper is structured as follows. In Section \ref{causal_sec_setting}, we fix some definitions and notations that will be used in all the rest of this work, and recall the backbone decomposition of supercritical Galton--Watson trees. In Section \ref{causal_sec_metric}, we investigate metric properties and establish Theorem \ref{thm_1_bis}, of which Theorem \ref{thm_1_metric} is a particular case. Section \ref{causal_sec_poisson} is devoted to the study of the Poisson boundary and to the proof of Theorems \ref{thm_2_Poisson} and \ref{thm_2_bis}. In Section \ref{causal_sec_speed}, we prove Theorem \ref{thm_3_positive} about positive speed. Finally, in Section \ref{causal_sec_counter}, we discuss some counterexamples related to Figure \ref{summary_conjectures}, and state a few conjectures.

\paragraph{Acknowledgments:} I thank Nicolas Curien for his comments on earlier versions of this work, Arvind Singh for explanations about renewal theory, and Itai Benjamini for providing the reference \cite{Ben14}. I am grateful to the anonymous referee for pointing out that the third item of Theorem \ref{thm_2_bis} could also be proved. I acknowledge the support of ANR Liouville (ANR-15-CE40-0013), ANR GRAAL (ANR-14-CE25-0014) and ERC GeoBrown (740943).

\tableofcontents

\section{General framework and the backbone decomposition}\label{causal_sec_setting}

The goal of this first section is to give definitions and notations, and to make a few useful remarks that will be needed in all the paper. All our constructions will be based on infinite, locally finite plane trees. We insist that the plane tree structure is important to define the associated causal map. We will use normal letters to denote general infinite trees, and bold letters like $\mathbf{t}$ for trees with no leaf. All the trees will be rooted at a vertex $\rho$. If $v$ is a vertex of a tree $t$, we denote by $h(v)$ its distance to the root, which we will sometimes call its \emph{height}. 
A \emph{ray} in a tree $t$ is an infinite sequence $\left( \gamma(i) \right)_{i \geq 0}$ of vertices such that $\gamma(0)=\rho$, and $\gamma(i+1)$ is a child of $\gamma(i)$ for every $i \geq 0$.
If $t$ is an infinite tree, the \emph{backbone} of $t$ is the union of its rays, i.e. the set of the vertices of $t$ that have infinitely many descendants. We will denote it by $\back(t)$, and we note that $\back(t)$ is always an infinite tree with no leaf.

We recall that if $t$ is an infinite plane tree, then $\cau(t)$ is the map obtained from $t$ by adding horizontal edges at every height between consecutive vertices.
We also define the \emph{causal slice} $\sli(t)$ associated to $t$, which will be used a lot in all that follows. Let $\gamma_{\ell}$ (resp. $\gamma_r$) be the leftmost (resp. rightmost) infinite ray of $\back(t)$. Then $\sli(t)$ is the map obtained from $t$ by deleting all the vertices on the left of $\gamma_{\ell}$ and on the right of $\gamma_r$, and by adding the same horizontal edges as for $\cau(t)$ between the remaining vertices, except the edge between $\gamma_{\ell}$ and $\gamma_r$ at each level (cf. Figure \ref{fig_def_slice}). The union of $\gamma_{\ell}$ and $\gamma_r$ is the \emph{boundary} of $\sli$, and is written $\partial \sli$.

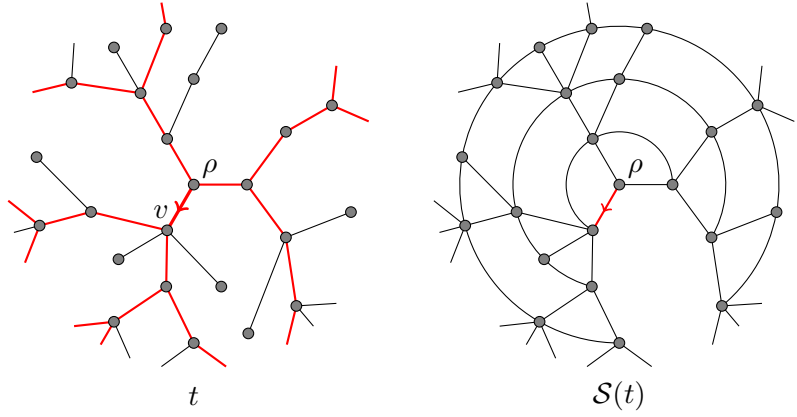
\begin{figure}
\begin{center}
\begin{tikzpicture}[scale=0.7]

\draw[red, thick](0,0)--(0:1);
\draw[red, thick](0,0)--(120:1);
\draw[very thick, red, ->](0,0)--(240:0.6);
\draw[very thick, red](240:0.5)--(240:1);
\draw[red, thick](0:1)--(30:2);
\draw[red, thick](0:1)--(330:2);
\draw(120:1)--(90:2);
\draw[red, thick](120:1)--(120:2);
\draw[red, thick](240:1)--(195:2);
\draw(240:1)--(225:2);
\draw[red, thick](240:1)--(255:2);
\draw(240:1)--(285:2);
\draw[red, thick](30:2)--(30:3);
\draw(330:2)--(350:3);
\draw[red, thick](330:2)--(310:3);
\draw(330:2)--(290:3);
\draw(90:2)--(80:3);
\draw[red, thick](120:2)--(100:3);
\draw(120:2)--(120:3);
\draw[red, thick](120:2)--(140:3);
\draw(195:2)--(170:3);
\draw[red, thick](195:2)--(195:3);
\draw[red, thick](255:2)--(240:3);
\draw[red, thick](255:2)--(270:3);
\draw[red, thick](310:3)--(300:3.5);
\draw(310:3)--(310:3.5);
\draw(310:3)--(320:3.5);
\draw[red, thick](30:3)--(20:3.5);
\draw[red, thick](30:3)--(40:3.5);
\draw[red, thick](100:3)--(100:3.5);
\draw(140:3)--(130:3.5);
\draw[red, thick](140:3)--(150:3.5);
\draw[red, thick](195:3)--(185:3.5);
\draw(195:3)--(195:3.5);
\draw[red, thick](195:3)--(205:3.5);
\draw[red, thick](240:3)--(230:3.5);
\draw[red, thick](240:3)--(240:3.5);
\draw(240:3)--(250:3.5);
\draw(270:3)--(260:3.5);
\draw[red, thick](270:3)--(280:3.5);

\draw(0,0) node{};
\draw(0:1) node{};
\draw(120:1) node{};
\draw(240:1) node{};
\draw(330:2) node{};
\draw(30:2) node{};
\draw(90:2) node{};
\draw(120:2) node{};
\draw(195:2) node{};
\draw(225:2) node{};
\draw(255:2) node{};
\draw(285:2) node{};
\draw(290:3) node{};
\draw(310:3) node{};
\draw(350:3) node{};
\draw(30:3) node{};
\draw(80:3) node{};
\draw(100:3) node{};
\draw(120:3) node{};
\draw(140:3) node{};
\draw(170:3) node{};
\draw(195:3) node{};
\draw(240:3) node{};
\draw(270:3) node{};

\draw(0,-4) node[texte]{$t$};
\draw(0.3,0.3) node[texte]{$\rho$};
\draw(220:0.8) node[texte]{$v$};

\begin{scope}[shift={(8,0)}]
\draw(0,0)--(0:1);
\draw(0,0)--(120:1);
\draw[thick, red, ->](0,0)--(240:0.6);
\draw[thick, red](240:0.5)--(240:1);
\draw(0:1)--(30:2);
\draw(0:1)--(330:2);
\draw(120:1)--(90:2);
\draw(120:1)--(120:2);
\draw(240:1)--(195:2);
\draw(240:1)--(225:2);
\draw(240:1)--(255:2);
\draw(330:2)--(310:3);
\draw(330:2)--(350:3);
\draw(30:2)--(30:3);
\draw(90:2)--(80:3);
\draw(120:2)--(100:3);
\draw(120:2)--(120:3);
\draw(120:2)--(140:3);
\draw(195:2)--(170:3);
\draw(195:2)--(195:3);
\draw(255:2)--(240:3);
\draw(255:2)--(270:3);
\draw(310:3)--(300:3.5);
\draw(310:3)--(310:3.5);
\draw(310:3)--(320:3.5);
\draw(30:3)--(20:3.5);
\draw(30:3)--(40:3.5);
\draw(100:3)--(100:3.5);
\draw(140:3)--(130:3.5);
\draw(140:3)--(150:3.5);
\draw(195:3)--(185:3.5);
\draw(195:3)--(195:3.5);
\draw(195:3)--(205:3.5);
\draw(240:3)--(230:3.5);
\draw(240:3)--(240:3.5);
\draw(240:3)--(250:3.5);
\draw(270:3)--(260:3.5);
\draw(270:3)--(280:3.5);

\draw(0:1) arc (0:240:1);
\draw(-30:2) arc (-30:255:2);
\draw(-50:3) arc (-50:270:3);

\draw(0,0) node{};
\draw(0:1) node{};
\draw(120:1) node{};
\draw(240:1) node{};
\draw(330:2) node{};
\draw(30:2) node{};
\draw(90:2) node{};
\draw(120:2) node{};
\draw(195:2) node{};
\draw(225:2) node{};
\draw(255:2) node{};
\draw(310:3) node{};
\draw(350:3) node{};
\draw(30:3) node{};
\draw(80:3) node{};
\draw(100:3) node{};
\draw(120:3) node{};
\draw(140:3) node{};
\draw(170:3) node{};
\draw(195:3) node{};
\draw(240:3) node{};
\draw(270:3) node{};

\draw(0,-4) node[texte]{$\sli(t)$};
\draw(0.3,0.3) node[texte]{$\rho$};
\end{scope}

\end{tikzpicture}
\end{center}
\caption{The same infinite tree $t$ as on Figure \ref{fig_def_causal} and the associated causal slice $\sli(t)$. Note that two vertices have been deleted. On the left part, the backbone of $t$ is in red. We have $c(v)=4$, $c_{\mathbf{B}}(v)=2$ and $A_v=\{2,4\}$.}\label{fig_def_slice}
\end{figure}

In all this work, $\mu$ will denote a supercritical offspring distribution, i.e. a distribution satisfying $\sum_{i \geq 0} i \mu(i)>1$, and $T$ will be a Galton--Watson tree with offspring distribution $\mu$ conditioned to survive. For every $n \geq 0$, we will denote by $Z_n$ the number of vertices of $T$ at height $n$. We will write $\cau$ for $\cau(T)$ and $\sli$ for $\sli(T)$, unless stated otherwise.

If $v$ is a vertex of $\back(T)$, we will denote by $T[v]$ the tree of descendants of $v$ in $T$, and by $\sli[v]$ the causal slice associated to $T[v]$. An important consequence of the backbone decomposition stated below is that for each $v \in T$, conditionally on $v  \in \mathbf{B}(T)$, the slice $\sli[v]$ has the same distribution as $\sli$. Moreover, these slices are independent for base points that are not ancestors of each other.

If $G$ is a graph rooted at a vertex $\rho$, we will denote by $d_G$ its graph distance and by $B_r(G)$ (resp. $\partial B_r(G)$) the set of vertices of $G$ at distance at most $r$ (resp. exactly $r$) from $\rho$. Note that the vertices of $B_r(T)$ and of $B_r(\cau)$ are the same.
For any vertex $v$ of $T$, we also denote by $c_T(v)$ (or by $c(v)$ when there is no ambiguity) the number of children of $v$ in $T$. Note that the degree of $v$ in $\cau$ is equal to $c_T(v)+3$ if $v \ne \rho$, and to $c_T(v)$ if $v=\rho$.
For every graph $G$ and every vertex $v$ of $G$, we will denote by $P_{G,v}$ the distribution of the simple random walk on $G$ started from $v$.

We now recall the backbone decomposition for supercritical Galton--Watson trees conditioned to survive, as it appears e.g. in \cite{L92}. Let $f$ be the generating function of $\mu$, i.e. $f(x)=\sum_{i \geq 0} \mu(i) x^i$. Let also $q$ be the extinction probability of a Galton--Watson tree with offspring distribution $\mu$, i.e. the smallest fixed point of $f$ in $[0,1]$. We define $\boldsymbol{f}$ and $\widetilde{f}$ by
\begin{equation}\label{several_offspring}
\boldsymbol{f}(s)=\frac{f(q+(1-q)s)-q}{1-q} \quad \mbox{ and } \quad \widetilde{f}(s)=\frac{f(qs)}{q}
\end{equation}
for $s \in [0,1]$. Then $\boldsymbol{f}$ is the generating function of a supercritical offspring distribution $\boldsymbol{\mu}$ with $\boldsymbol{\mu}(0)=0$, and $\widetilde{f}$ is the generating function of a subcritical offspring distribution $\widetilde{\mu}$. For every vertex $x$ of $\back(T)$, we denote by $A_x$ the set of indices $1 \leq i \leq c(x)$ such that the $i$-th child of $x$ is in the backbone, and we write $c_{\back}(x)=|A_x|$. Finally, let $\back'(T)$ be the set of vertices $y$ of $T$ such that the parent of $y$ is in $\back(T)$, but $y$ is not. The following result characterizes entirely the distribution of $T$.

\begin{thm}\label{backbone_decomposition}
\begin{enumerate}
\item
The tree $\back(T)$ is a Galton--Watson tree with offspring distribution $\boldsymbol{\mu}$.
\item
Conditionally on $\back(T)$, the variables $c(x)-c_{\back}(x)$ are independent, with distribution characterized by
\[ \E \left[ s^{c(x)-c_{\back}(x)} \big| \back(T) \right]=\frac{f^{(c_{\back}(x))}(qs)}{f^{(c_{\back}(x))}(q)}\]
for every $s \in [0,1]$, where $f^{(k)}$ stands for the $k$-th derivative of $f$.
\item
Conditionally on $\back(T)$ and the variables $c(x)$ for $x \in \back(T)$, the sets $A_x$ for $x \in \back(T)$ are independent and, for every $x$, the set $A_x$ is uniformly distributed among all the subsets of $\{1,2, \dots, c(x)\}$ with $c_{\back}(x)$ elements.
\item
Conditionally on everything above, the trees $T[y]$ for $y \in \back'(T)$ are independent Galton--Watson trees with offspring distribution $\widetilde{\mu}$.
\end{enumerate}
\end{thm}

In particular, this decomposition implies that, for every $h \geq 1$, conditionally on $B_h(T)$ and on the set $\partial B_h(T) \cap \back(T)$, the trees $T[x]$ for $x \in \partial B_h(T) \cap \back(T)$ are i.i.d. copies of $T$. Therefore, the slices $\sli[x]$ for $x \in \partial B_h(T) \cup \back(T)$ are i.i.d. copies of $\sli$. This "self-similarity" property of $\sli$ will be used a lot later on.

We end this section by adapting these notions to the more general setting of strips glued in the faces of a tree with no leaf. We recall that a \emph{strip} is an infinite, one-ended planar map with an infinite, simple boundary, such that all the faces except the outer face have finite degree. A strip is also rooted at a root vertex on its boundary. Let $\mathbf{t}$ be an infinite plane tree with no leaf. We draw $\mathbf{t}$ in the plane in such a way that its edges do not intersect (except at a common endpoint), and every compact subset of the plane intersects finitely many vertices and edges. Then $\mathbf{t}$ separates the plane into a countable family $(f_i)_{i \geq 0}$ of faces, where $f_0$ is the face delimited by the leftmost and the rightmost rays of $\mathbf{t}$, and the other faces are enumerated in a deterministic fashion. For every index $i \geq 0$, we denote by $\rho_i$ the lowest vertex of $\mathbf{t}$ adjacent to $f_i$, and by $h_i$ its height. Note that this vertex is always unique. On the other hand, for every vertex $v$ of $\mathbf{t}$, there are exactly $c_{\mathbf{t}}(v)-1$ faces $f_i$ such that $\rho_i=v$.

Let $(s_i)_{i \geq 0}$ be a family of random strips. We denote by $\m \left( \mathbf{t}, (s_i)_{i \geq 0} \right)$ the infinite planar map obtained by gluing $s_i$ in the face $f_i$ for every $i \geq 0$, in such a way that the root vertex of $s_i$ coincides with $\rho_i$ for every $i$. We also denote by $\sli \left( \mathbf{t}, (s_i)_{i \geq 0} \right)$ the map obtained by gluing $s_i$ in the face $f_i$ for every $i > 0$ (this is a map with an infinite boundary analog to the slice $\sli$).
If $v$ is a vertex of $\mathbf{t}$, we also define the "slice of descendants" of $v$ as the map enclosed between the leftmost and the rightmost rays of $\mathbf{t}$ started from $v$. We denote it by $\sli \left( \mathbf{t}, (s_i) \right) [v]$.

We note that causal maps are a particular case of this construction. This is trivial for supercritical Galton--Watson trees with no leaf. Thanks to Theorem \ref{backbone_decomposition}, this can be extended to the case $\mu(0)>0$, with $\back(T)$ playing the role of $\mathbf{t}$. This time, however, the strips are random, but they are not independent. Indeed, if $v$ is a vertex of $\back(T)$ and $w$ one of its children in $\back(T)$, the children of $v$ on the left of $w$ and the children on the right of $w$ belong to different strips. However, by points 2 and 3 of Theorem \ref{backbone_decomposition}, the numbers of children on the left and on the right are not independent, except in some very particular cases (for example if $\mu$ is geometric).

In what follows, we will study maps of the form $\m \left( \mathbf{T}, (s_i)_{i \geq 0} \right)$, where $\mathbf{T}$ is a supercritical Galton--Watson tree with no leaf.  We notice right now that if the strips $s_i$ are random and i.i.d., then the slice $\sli \left( \mathbf{T}, (s_i) \right)$ has the same self-similarity property as the causal slices of the form $\sli(T)$. Let $h>0$. We condition on $B_h(\mathbf{T})$ and on all the strips $s_i$ such that $h_i \leq h-1$. Then the trees $\mathbf{T}[v]$ for $v \in \partial B_h(\mathbf{T})$ are independent copies of $\mathbf{T}$, so the slices $\sli \left( \mathbf{T}, (s_i) \right)[v]$ for $v \in \partial B_h(\mathbf{T})$ are i.i.d. copies of $\sli \left( \mathbf{T}, (s_i) \right)$. This will be useful in Section \ref{causal_sec_poisson}.

\section{Metric hyperbolicity properties}\label{causal_sec_metric}

The goal of this section is to prove the following result, of which Theorem \ref{thm_1_metric} is a particular case. Note that we make no assumption about the strips $s_i$ below. In particular, they may be deterministic or random, and may depend on the tree $\mathbf{T}$.

\begin{thm}\label{thm_1_bis}
Let $\mathbf{T}$ be a supercritical Galton--Watson tree with no leaf, and let $(s_i)$ be a sequence of strips. Then:
\begin{enumerate}
\item
the map $\m \left( \mathbf{T}, (s_i) \right)$ is a.s.~weakly anchored hyperbolic,
\item
the map $\m \left( \mathbf{T}, (s_i) \right)$ a.s.~admits bi-infinite geodesics.
\end{enumerate}
\end{thm}

In all this section, we will only deal with the general case of $\m \left( \mathbf{T}, (s_i) \right)$ where $\mathbf{T}$ is a supercritical Galton--Watson tree with no leaf and the $s_i$ are strips. We will write $\m$ for $\m \left( \mathbf{T}, (s_i) \right)$ and $\sli$ for $\sli \left( \mathbf{T}, (s_i) \right)$. Our main tool will be the forthcoming Proposition \ref{geodesics_contains_bdd_point}, which roughly shows that $\sli$ is hard to cross horizontally at large heights.

\subsection{A hyperbolicity result about slices}

We call $\gamma_{\ell}$ and $\gamma_r$ the left and right boundaries of $\sli$, and $\rho$ its root (note that $\gamma_{\ell}$ and $\gamma_r$ may have an initial segment in common near $\rho$). Both points of Theorem \ref{thm_1_bis} will be consequences of the following hyperbolicity result about $\sli$.

\begin{prop}\label{geodesics_contains_bdd_point}
There is a (random) $K \geq 0$ such that any geodesic in $\sli$ from a point on $\gamma_{\ell}$ to a point on $\gamma_{r}$ contains a point at distance at most $K$ from $\rho$.
\end{prop}


We first give a very short proof of this proposition in the particular case of a causal slice of the form $\sli(T)$. Let $i,j>0$ and let $\gamma$ be a geodesic in $\sli(T)$ from $\gamma_{\ell}(i)$ to $\gamma_r(j)$. Let $v_0$ be the lowest point of $\gamma$, and let $h_0$ be the height of $v_0$. By the structure of $\sli(T)$, each step of $\gamma$ is either horizontal or vertical. Since the height varies by at most $1$ at each vertical step, we need at least $(i-h_0)+(j-h_0)$ vertical steps. Moreover, for every $h \geq 0$, let $Z_h^{\back}$ be the number of vertices of $\back(T)$ at height $h$. Then $\gamma$ needs to cross all the trees $T[x]$ for $x \in \back(T)$ at height $h_0$, so $\gamma$ contains at least $Z_{h_0}^{\back}$ horizontal steps. On the other hand, $\gamma$ is a geodesic so it is shorter than the "obvious" path following $\gamma_{\ell}$ from $\gamma_{\ell}(i)$ to $\rho$ and then $\gamma_r$ until $\gamma_r(j)$. Therefore, we have
\[i+j-2h_0+Z^{\back}_{h_0} \geq |\gamma| \geq i+j,\]
so $Z^{\back}_{h_0} \leq 2h_0$. However, $Z^{\back}$ has a.s.~exponential growth, so this inequality only holds for finitely many values of $h_0$, so $h_0$ is bounded independently of $i$ and $j$.

To generalize this proof, there are two obstacles: first, the branches of the tree are no longer geodesics in $\sli$ in the general case, so a single step may change the height by more than one. Second, an edge of $\sli$ can play the role both of a vertical and a horizontal step if it crosses a strip and joins two vertices of $\mathbf{T}$ at different heights. However, we can still cross at most one strip per step in this way.

In order to prove the general Proposition \ref{geodesics_contains_bdd_point}, we first state a lemma showing roughly that if a path in $\mathbf{T}$ with nondecreasing height stays at height $h$ during a time subexponential in $h$, it cannot cross $\sli$ from left to right.

More precisely, we fix a sequence of positive integers $(u_i)_{i \geq 0}$ and a height $k \geq 0$. Let $x \in \partial B_k(\mathbf{T})$. We define by induction two sequences $(y_i)_{i \geq k}$ and $(z_i)_{i \geq k}$ of vertices of $\mathbf{T}$ with $y_i, z_i \in \partial B_i (\mathbf{T})$ as follows (see Figure \ref{example_yz} for an example):
\begin{itemize}
\item[(i)]
$y_k=x$,
\item[(ii)]
for every $i \geq k$, if there are at least $u_i$ vertices on the right of $y_i$ on $\partial B_i(\mathbf{T})$ ($y_i$ excluded), then $z_i$ is the $u_i$-th such vertex,
\item[(iii)]
if there are less than $u_i$ vertices of $\partial B_i(\mathbf{T})$ on the right of $y_i$, the sequences $(y_i)$ and $(z_i)$ are killed at time $i$,
\item[(iv)]
for every $i \geq k$, if $z_i \notin \gamma_r$, the vertex $y_{i+1}$ is the rightmost child of $z_i$ in $\mathbf{T}$. If $z_i \in \gamma_r$, both sequences are killed.
\end{itemize}
We call $(y_i)_{i \geq k}$ and $(z_i)_{i \geq k}$ the \emph{sequences escaping from $x$ on the right}. We say that $x$ is \emph{$u$-far from $\gamma_r$} if the sequences $(y_i)$ and $(z_i)$ survive, that is, $y_i$ and $z_i$ are well-defined and do not hit $\gamma_r$ for all $i \geq k$.

Note that being $u$-far from $\gamma_r$ is a monotonic property: if we shift the point $x$ to the left, then the points $y_i$ and $z_i$ are also shifted to the left. Hence, if
a point $x \in \partial B_k(\mathbf{T})$ is $u$-far from $\gamma_r$ and $x' \in \partial B_k(\mathbf{T})$ lies on the left of $x$, then $x'$ is also $u$-far from $\gamma_r$. We can similarly define the sequences escaping on the left, and a vertex $u$-far from $\gamma_{\ell}$.

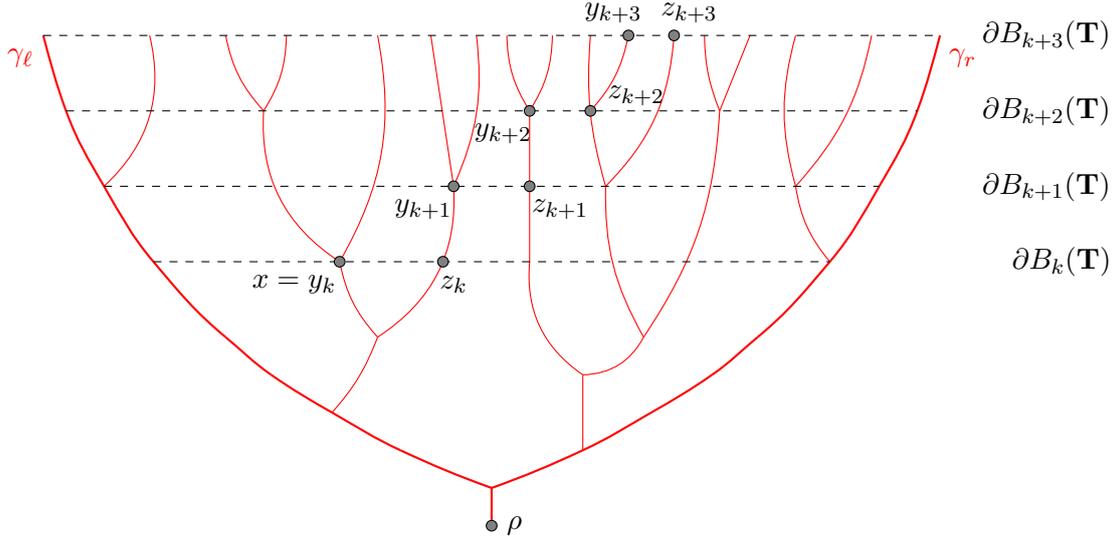
\begin{figure}
\begin{center}
\begin{tikzpicture}
\draw[thick, red] (0,-0.5)--(0,0);
\draw[thick, red] (0,0) to[out=20,in=205] (1.2,0.5);
\draw[thick, red] (1.2,0.5) to[out=25, in=210] (2.1,1);
\draw[thick, red] (2.1,1) to[out=30, in=215] (2.9,1.5);
\draw[thick, red] (2.9,1.5) to[out=35, in=220] (3.5,2);
\draw[thick, red] (3.5,2) to[out=40, in=230] (4.45,3);
\draw[thick, red] (4.45,3) to[out=50, in=240] (5.1,4);
\draw[thick, red] (5.1,4) to[out=60, in=250] (5.6,5);
\draw[thick, red] (5.6,5) to[out=70, in=255] (5.9,6);
\draw[red] (6.2,5.7) node[texte]{$\gamma_r$};

\draw[thick, red] (0,0) to[out=160,in=335] (-1.2,0.5);
\draw[thick, red] (-1.2,0.5) to[out=155, in=330] (-2.1,1);
\draw[thick, red] (-2.1,1) to[out=150, in=325] (-2.9,1.5);
\draw[thick, red] (-2.9,1.5) to[out=145, in=320] (-3.5,2);
\draw[thick, red] (-3.5,2) to[out=140, in=310] (-4.45,3);
\draw[thick, red] (-4.45,3) to[out=130, in=300] (-5.1,4);
\draw[thick, red] (-5.1,4) to[out=120, in=290] (-5.6,5);
\draw[thick, red] (-5.6,5) to[out=110, in=285] (-5.9,6);
\draw[red] (-6.2,5.7) node[texte]{$\gamma_{\ell}$};

\draw[red] (1.2,0.5) -- (1.2,1.5);
\draw[red] (-2.1,1) to[bend right=10] (-1.5,2);
\draw[red] (1.2,1.5) to[bend right] (2,2);
\draw[red] (1.2,1.5) to[bend left] (0.5,3);
\draw[red] (-1.5,2) to[bend left=15] (-2,3);
\draw[red] (-1.5,2) to[bend right] (-0.5,4);
\draw[red] (2,2) to[bend right=15] (3,5);
\draw[red] (2,2) to[bend left=15] (1.5,4);
\draw[red] (4.45,3) to[bend left=15] (4,4);
\draw[red] (4,4) to[bend left=15] (4,6);
\draw[red] (4,4) to[bend right=15] (5,6);
\draw[red] (-2,3) to[bend left] (-3,5);
\draw[red] (-2,3) to[bend right=20] (-1.5,6);
\draw[red] (-5.1,4) to[bend right] (-4.5,6);
\draw[red] (-3,5) to[bend left=15] (-3.5,6);
\draw[red] (-3,5) to[bend right=15] (-2.7,6);
\draw[red] (0.5,3) -- (0.5,5);
\draw[red] (0.5,5) to[bend left=15] (0.2,6);
\draw[red] (0.5,5) to[bend right=15] (0.8,6);
\draw[red] (1.3,5) to[bend right=15] (1.8,6);
\draw[red] (3,5) to[bend left=0] (3.4,6);
\draw[red] (3,5) to[bend left=10] (2.8,6);
\draw[red] (1.5,4) to[bend left=10] (1.3,6);
\draw[red] (1.5,4) to[bend right=20] (2.4,6);
\draw[red] (-0.5,4) to[bend left=0] (-0.8,6);
\draw[red] (-0.5,4) to[bend right=15] (-0.2,6);

\draw[dashed] (-4.45,3)--(4.45,3);
\draw[dashed] (-5.1,4)--(5.1,4);
\draw[dashed] (-5.6,5)--(5.6,5);
\draw[dashed] (-5.9,6)--(5.9,6);

\draw(7.5,3) node[texte]{$\partial B_k(\mathbf{T})$};
\draw(7.3,4) node[texte]{$\partial B_{k+1}(\mathbf{T})$};
\draw(7.3,5) node[texte]{$\partial B_{k+2}(\mathbf{T})$};
\draw(7.3,6) node[texte]{$\partial B_{k+3}(\mathbf{T})$};

\draw(0,-0.5) node{};
\draw(0.3,-0.5) node[texte]{$\rho$};
\draw(-2,3) node{};
\draw(-2.6,2.7) node[texte]{$x=y_k$};
\draw(-0.64,3) node{};
\draw(-0.5,2.7) node[texte]{$z_k$};
\draw(-0.5,4) node{};
\draw(-0.9,3.7) node[texte]{$y_{k+1}$};
\draw(0.5,4) node{};
\draw(0.9,3.7) node[texte]{$z_{k+1}$};
\draw(0.5,5) node{};
\draw(0.15,4.7) node[texte]{$y_{k+2}$};
\draw(1.3,5) node{};
\draw(1.9,5.2) node[texte]{$z_{k+2}$};
\draw(1.8,6) node{};
\draw(1.6,6.3) node[texte]{$y_{k+3}$};
\draw(2.4,6) node{};
\draw(2.6,6.3) node[texte]{$z_{k+3}$};

\end{tikzpicture}
\end{center}
\caption{The sequences $(y_i)_{i \geq k}$ and $(z_i)_{i \geq k}$. Here we have taken $u_i=1$ for every $i$. The tree $\mathbf{T}$ is in red.} \label{example_yz}
\end{figure}

\begin{lem} \label{LfarfromR}
Assume $u$ is subexponential, i.e. $u_i=o(c^i)$ for every $c>1$. Then there is a (random) $K$ such that for any $k \geq K$, the vertex $\gamma_{\ell}(k)$ is $u$-far from $\gamma_r$ and the vertex $\gamma_r(k)$ is $u$-far from $\gamma_{\ell}$.
\end{lem}

\begin{proof}
The idea is to reduce the proof to the study of a supercritical Galton--Watson process where $u_i$ individuals are killed at generation $i$. It is enough to show that $\gamma_{\ell}(k)$ is $u$-far from $\gamma_r$ for $k$ large enough. Note that if $\gamma_{\ell}(k)$ is $u$-far from $\gamma_r$ and $(y_i)_{i \geq k}$ is its sequence escaping on the right, then $(y_i)_{i \geq k+1}$ is the sequence escaping on the right of $y_{k+1}$, so $y_{k+1}$ is also $u$-far from $\gamma_r$ and, by monotonicity, so is $\gamma_{\ell}(k+1)$. Therefore, it is enough to show that there is $k \geq 0$ such that $\gamma_{\ell}(k)$ is $u$-far from $\gamma_r$.
 
Let $k \geq 0$ and $x=\gamma_{\ell}(k)$. Let $(y_i)_{i \geq k}$ and $(z_i)_{i \geq k}$ be the sequences escaping from $x$ on the right. We also denote by $Z_i^k$ the number of vertices of $\partial B_i(\mathbf{T})$ lying (strictly) on the right of $z_i$. We first remark that the evolution of the process $Z^k$ can be described explicitly. We have $Z_k^k=Z_k-1$. Moreover, we recall that $\mu$ is the offspring distribution of $\mathbf{T}$. Conditionally on $(Z_k^k, Z_{k+1}^k, \dots, Z_i^k)$, the variable $Z_{i+1}^k$ has the same distribution as
\[ \left( \left( \sum_{j=1}^{Z_i^k} X_{i,j} \right)-u_{i+1} \right)^+ ,\]
where the $X_{i,j}$ are i.i.d. with distribution $\mu$. Moreover, the process $Z^k$ is killed when it hits $0$. To prove our lemma, it is enough to show that $\P \left( \forall i \geq k, Z_i^k > 0 \right)$ goes to $1$ as $k$ goes to $+\infty$. Since the process $Z$ describing the number of individuals at each generation is a supercritical Galton--Watson process, there is a constant $c>1$ such that
\[\P \left( \partial B_k(\mathbf{T})>c^k \right) \xrightarrow[k \to +\infty]{} 1.\]
Therefore, by a monotonicity argument, it is enough to prove that
\[ \P \left( \forall i \geq k, Z_i^k>0 \big| Z_k^k=\lfloor c^k \rfloor  \right) \xrightarrow[k \to +\infty]{} 1.\]
To prove this, we show that $Z^k$ dominates a supercritical Galton--Watson process. Let $\delta>0$ be such that $(1-\delta) \sum_{i \geq 0} i \mu(i)>1$. Let also $Z^*$ be the Markov chain defined by
\[ Z^*_k= \lfloor c^k \rfloor \quad \mbox{and} \quad Z^*_{i+1}= \left( \sum_{j=1}^{Z^*_i} X_{i,j} \right) -N_{i+1},\]
where the $X_{i,j}$ are i.i.d. with distribution $\mu$ and, conditionally on $(X_{i,j})_{i,j \geq 0}$ and $(N_j)_{1 \leq j \leq i}$, the variable $N_{i+1}$ has binomial distribution with parameters $\delta$ and $\sum_{j=1}^{Z^*_i} X_{i,j}$.
In other words, $Z^*$ is a Galton--Watson process in which at every generation, right after reproduction, every individual is killed with probability $\delta$. By our choice of $\delta$, the process $Z^*$ is a supercritical Galton--Watson process and, on survival, grows exponentially.
By an easy large deviation argument, we have
\[ \P \left( \forall i \geq k, N_i \geq u_i | Z^*_k= \lfloor c^k \rfloor \right) \xrightarrow[k \to +\infty]{}1.\]
If this occurs, then $Z^*_i \leq Z_i^k$ for every $i \geq k$ (by an easy induction on $i$), so
\[\P \left( \forall i \geq k, Z_i^k \geq Z^*_i >0  | Z_k^k= \lfloor c^k \rfloor \right) \xrightarrow[k \to +\infty]{}1\]
and the lemma follows.
\end{proof}

The proof of Proposition \ref{geodesics_contains_bdd_point} given Lemma \ref{LfarfromR} only relies on deterministic considerations.

\begin{proof}[Proof of Proposition \ref{geodesics_contains_bdd_point}]
We note that for any $a \in \mathbf{T}$, the slice $\sli[a]$ is of the form $\sli \left( \mathbf{T}[a], (s'_i)\right)$ where $\mathbf{T}[a]$ is a supercritical Galton--Watson tree with no leaf, so we can apply Lemma \ref{LfarfromR} to $\sli[a]$.

Now let $a,a'$ be two vertices of $\mathbf{T}$, neither of which is an ancestor of the other, as on Figure \ref{figure_farfromLorR}. Then $\sli[a]$ and $\sli[a']$ are disjoint. Without loss of generality, we may assume that $\sli[a]$ lies on the right of $\sli[a']$. Let $K$ (resp. $K'$) be given by the conclusion of Lemma \ref{LfarfromR} for $\sli[a]$ (resp. $\sli[a']$) and $u_i=2 \left( i+ \max ( d(\rho,a), d(\rho,a')) \right)+1$. We take $K''=\max \left( d(\rho,a)+K+1, d(\rho,a')+K'+1 \right)$. We consider a geodesic $\gamma$ from a vertex $\gamma_{\ell}(m)$ to a vertex $\gamma_r(n)$ in $\sli$. Let
$k$ be the minimal height of $\gamma \cap \mathbf{T}$. We assume $k>K''$, and we will get a contradiction.

Let $b$ be the leftmost vertex of $\sli[a] \cap \mathbf{T}$ that lies at height $k$, and let $b'$ be the rightmost vertex of $\sli[a'] \cap \mathbf{T}$ that lies at height $k$ (cf. Figure \ref{figure_farfromLorR}). By Lemma \ref{LfarfromR} and our choice of $K''$, the point $b$ is $\widetilde{u}$-far from $\gamma_r$ in $\sli$ and $b'$ is $\widetilde{u}$-far from $\gamma_{\ell}$ in $\sli$, where $\widetilde{u}_i=2i+1$ (the change of the sequence $u$ to the sequence $\widetilde{u}$ is due to the fact that the distances to the root are not the same in $\sli[a]$ and $\sli$, so the sequence needs to be shifted). But any vertex of $\partial B_k(\mathbf{T})$ lies either on the left of $b$ or on the right of $b'$, so it is either $\widetilde{u}$-far from $\gamma_{\ell}$ or from $\gamma_r$ (see Figure \ref{figure_farfromLorR}). In particular, let $x$ be the first point of $\gamma$ lying on $\mathbf{T}$ at height $k$. We may assume that $x$ is $\widetilde{u}$-far from $\gamma_r$, the other case can be treated in the same way.

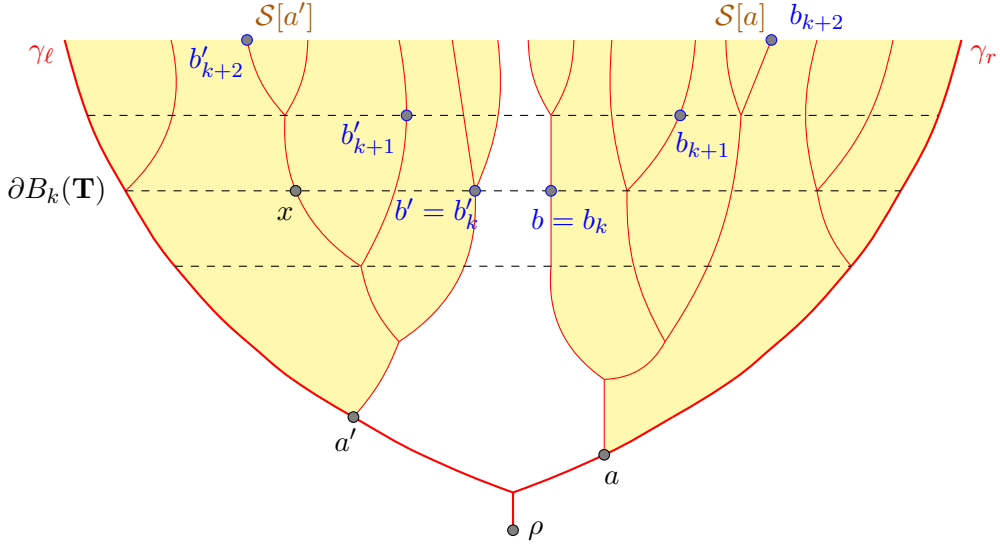
\begin{figure}
\begin{center}
\begin{tikzpicture}
\fill[yellow!40] (-5.9,6) to[in=110, out=285] (-5.6,5) to[in=120, out=290] (-5.1,4) to[in=130, out=300] (-4.45,3) to[in=140, out=310] (-3.5,2) to[in=145, out=320] (-2.9,1.5) to[in=150, out=325] (-2.1,1) to[bend right=10] (-1.5,2) to[bend right] (-0.5,4) to[bend right=15] (-0.2,6);
\fill[yellow!40] (5.9,6) to[in=70, out=255] (5.6,5) to[in=60, out=250] (5.1,4) to[in=50, out=240] (4.45,3) to[in=40, out=230] (3.5,2) to[in=35, out=220] (2.9,1.5) to[in=30, out=215] (2.1,1) to[in=25, out=210] (1.2,0.5) -- (1.2,1.5) to[bend left] (0.5,3) -- (0.5,5)to[bend left=15] (0.2,6);
\draw[color={rgb:red,100;green,50;blue,0}] (-3,6.3) node[texte]{$\sli[a']$};
\draw[color={rgb:red,100;green,50;blue,0}] (3,6.3) node[texte]{$\sli[a]$};

\draw[thick, red] (0,0) to[out=20,in=205] (1.2,0.5);
\draw[thick, red] (1.2,0.5) to[out=25, in=210] (2.1,1);
\draw[thick, red] (2.1,1) to[out=30, in=215] (2.9,1.5);
\draw[thick, red] (2.9,1.5) to[out=35, in=220] (3.5,2);
\draw[thick, red] (3.5,2) to[out=40, in=230] (4.45,3);
\draw[thick, red] (4.45,3) to[out=50, in=240] (5.1,4);
\draw[thick, red] (5.1,4) to[out=60, in=250] (5.6,5);
\draw[thick, red] (5.6,5) to[out=70, in=255] (5.9,6);
\draw[red] (6.2,5.8) node[texte]{$\gamma_{r}$};

\draw[thick, red] (0,-0.5)--(0,0);
\draw[thick, red] (0,0) to[out=160,in=335] (-1.2,0.5);
\draw[thick, red] (-1.2,0.5) to[out=155, in=330] (-2.1,1);
\draw[thick, red] (-2.1,1) to[out=150, in=325] (-2.9,1.5);
\draw[thick, red] (-2.9,1.5) to[out=145, in=320] (-3.5,2);
\draw[thick, red] (-3.5,2) to[out=140, in=310] (-4.45,3);
\draw[thick, red] (-4.45,3) to[out=130, in=300] (-5.1,4);
\draw[thick, red] (-5.1,4) to[out=120, in=290] (-5.6,5);
\draw[thick, red] (-5.6,5) to[out=110, in=285] (-5.9,6);
\draw[red] (-6.2,5.8) node[texte]{$\gamma_{\ell}$};

\draw[red] (1.2,0.5) -- (1.2,1.5);
\draw[red] (-2.1,1) to[bend right=10] (-1.5,2);
\draw[red] (1.2,1.5) to[bend right] (2,2);
\draw[red] (1.2,1.5) to[bend left] (0.5,3);
\draw[red] (-1.5,2) to[bend left=15] (-2,3);
\draw[red] (-1.5,2) to[bend right] (-0.5,4);
\draw[red] (2,2) to[bend right=15] (3,5);
\draw[red] (2,2) to[bend left=15] (1.5,4);
\draw[red] (4.45,3) to[bend left=15] (4,4);
\draw[red] (4,4) to[bend left=15] (4,6);
\draw[red] (4,4) to[bend right=15] (5,6);
\draw[red] (-2,3) to[bend left] (-3,5);
\draw[red] (-2,3) to[bend right=20] (-1.5,6);
\draw[red] (-5.1,4) to[bend right] (-4.5,6);
\draw[red] (-3,5) to[bend left=15] (-3.5,6);
\draw[red] (-3,5) to[bend right=15] (-2.7,6);
\draw[red] (0.5,3) -- (0.5,5);
\draw[red] (0.5,5) to[bend left=15] (0.2,6);
\draw[red] (0.5,5) to[bend right=15] (0.8,6);
\draw[red] (3,5) to[bend left=0] (3.4,6);
\draw[red] (3,5) to[bend left=10] (2.8,6);
\draw[red] (1.5,4) to[bend left=10] (1.3,6);
\draw[red] (1.5,4) to[bend right=20] (2.4,6);
\draw[red] (-0.5,4) to[bend left=0] (-0.8,6);
\draw[red] (-0.5,4) to[bend right=15] (-0.2,6);

\draw[dashed] (-4.45,3)--(4.45,3);
\draw[dashed] (-5.1,4)--(5.1,4);
\draw[dashed] (-5.6,5)--(5.6,5);

\draw(-6,4)node[texte]{$\partial B_k(\mathbf{T})$};

\draw(-2.1,1) node{};
\draw(-2.2,0.7) node[texte]{$a'$};
\draw(1.2,0.5) node{};
\draw(1.3,0.2) node[texte]{$a$};
\draw(0,-0.5) node{};
\draw(0.3,-0.5) node[texte]{$\rho$};
\draw(-2.86,4) node{};
\draw(-3,3.7) node[texte]{$x$};
\draw[blue](-0.5,4) node{};
\draw[blue](-1,3.7) node[texte]{$b'=b'_k$};
\draw[blue](-1.4,5) node{};
\draw[blue](-1.9,4.7) node[texte]{$b'_{k+1}$};
\draw[blue](-3.5,6) node{};
\draw[blue](-3.9,5.7) node[texte]{$b'_{k+2}$};
\draw[blue](0.5,4) node{};
\draw[blue](0.75,3.6) node[texte]{$b=b_k$};
\draw[blue](2.2,5) node{};
\draw[blue](2.5,4.6) node[texte]{$b_{k+1}$};
\draw[blue](3.4,6) node{};
\draw[blue](4,6.3) node[texte]{$b_{k+2}$};
\end{tikzpicture}
\end{center}
\caption{Proof of Proposition \ref{geodesics_contains_bdd_point}: the point $b$ is $u$-far from $\gamma_r$ and $b'$ is $u$-far from $\gamma_{\ell}$, so any point on $\partial B_k(\mathbf{T})$ is $u$-far from either $\gamma_r$ or $\gamma_{\ell}$. Here we have taken $u_i=1$.} \label{figure_farfromLorR}
\end{figure}

Recall that $n$ is the height of the endpoint of $\gamma$. For every $k \leq i \leq n$, let
\[j_i=\max \left\{ j \in [\![0,|\gamma| ]\!] \big| \gamma(j) \in \mathbf{T} \mbox{ and } h \left( \gamma(j) \right) \leq i \right\},\]
and let $x_i=\gamma(j_i)$. Note that we have $h(x_i) \leq i$, but since the height can increase by more than $1$ in one step, the inequality may be strict.

Let also $(y_i)_{i \geq k}$ and $(z_i)_{i \geq k}$ be the sequences escaping from $x$ on the right in $\sli$ (for $\widetilde{u}_i=2i+1$). By our assumption that $x$ is $\widetilde{u}$-far from $\gamma_r$, these sequences are well-defined and do not hit $\gamma_r$. Moreover, for every $k \leq i \leq n$, the vertices $x_i$ and $z_{h(x_i)}$ are both in $\mathbf{T}$ and at the same height $h(x_i)$.
We claim that for every $i \geq k$, the vertex $x_i$ lies strictly on the left of the vertex $z_{h(x_i)}$. This is enough to prove the proposition, since then $x_n$ cannot lie on $\gamma_r$.

We show this claim by induction on $i$, and we start with the case $i=k$. Let $j_*$ be the index such that $x=\gamma(j_*)$. The vertices $x=\gamma(j_*)$ and $x_k=\gamma(j_k)$ both lie on $\partial B_k(\mathbf{T})$, so the distance in $\sli$ between them is at most $2k$. Hence, since $\gamma$ is a geodesic, we have $|j_k-j_*| \leq 2k$. Now, we consider the slices $\sli[v]$ for $v \in \partial B_k( \mathbf{T})$. These slices are disjoint and, by definition of $k$, the path $\left( \gamma(j) \right)_{j_* \leq j \leq j_k}$ does not cross $\mathbf{T}$ below height $k$, so it cannot intersect more than $2k<\widetilde{u}_k$ of these slices, which implies that $x_k$ lies on the left on $z_k$.

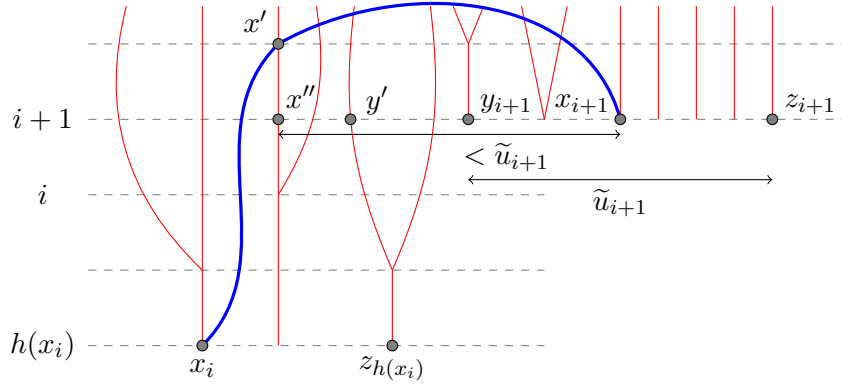
\begin{figure}
\begin{center}
\begin{tikzpicture}
\draw[dashed, color=black!50] (0,0)--(6,0);
\draw[dashed, color=black!50] (0,1)--(6,1);
\draw[dashed, color=black!50] (0,2)--(6,2);
\draw[dashed, color=black!50] (0,3)--(10,3);
\draw[dashed, color=black!50] (0,4)--(10,4);

\draw[red] (1.5,0)--(1.5,4.5);
\draw[red] (1.5,1) to[bend left] (0.5,4.5);
\draw[red] (2.5,0)--(2.5,4.5);
\draw[red] (2.5,2) to[bend right=20] (3,4.5);
\draw[red] (4,0)--(4,1);
\draw[red] (4,1) to[bend left=15] (3.5,4.5);
\draw[red] (4,1) to[bend right=15] (4.5,4.5);
\draw[red] (5,3)--(5,4);
\draw[red] (5,4)--(4.8,4.5);
\draw[red] (5,4)--(5.2,4.5);
\draw[red] (6,3)--(5.7,4.5);
\draw[red] (6,3)--(6.3,4.5);
\draw[red] (7,3)--(7,4.5);
\draw[red] (7.5,3)--(7.5,4.5);
\draw[red] (8,3)--(8,4.5);
\draw[red] (8.5,3)--(8.5,4.5);
\draw[red] (9,3)--(9,4.5);
\draw[blue, very thick] (1.5,0) to[out=45,in=225] (2.5,4);
\draw[blue, very thick] (2.5,4) to[out=30,in=105] (7,3);

\draw(1.5,0) node{};
\draw(1.5,-0.3) node[texte]{$x_i$};
\draw(4,0) node{};
\draw(4,-0.3) node[texte]{$z_{h(x_i)}$};
\draw(2.5,4) node{};
\draw(2.2,4.3) node[texte]{$x'$};
\draw(2.5,3) node{};
\draw(2.8,3.3) node[texte]{$x''$};
\draw(3.45,3) node{};
\draw(3.8,3.25) node[texte]{$y'$};
\draw(5,3) node{};
\draw(5.5,3.2) node[texte]{$y_{i+1}$};
\draw(7,3) node{};
\draw(6.5,3.2) node[texte]{$x_{i+1}$};
\draw(9,3) node{};
\draw(9.5,3.2) node[texte]{$z_{i+1}$};

\draw[<->] (2.5,2.8)--(7,2.8);
\draw(5.5,2.5)node[texte]{$< \widetilde{u}_{i+1}$};
\draw[<->] (5,2.2)--(9,2.2);
\draw(7,1.9)node[texte]{$\widetilde{u}_{i+1}$};

\draw(-0.6,0) node[texte]{$h(x_i)$};
\draw(-0.6,2) node[texte]{$i$};
\draw(-0.6,3) node[texte]{$i+1$};
\end{tikzpicture}
\end{center}
\caption{The induction step in the proof of Proposition \ref{geodesics_contains_bdd_point}. The path $\gamma$ between $x_i$ and $x_{i+1}$ is in blue. Branches of $\mathbf{T}$ are in red. The dashed lines are not edges of $\sli$, but indicate the height. We see that $x''$ is on the left of $y'$ which is on the left of $y_{i+1}$, so $x_{i+1}$ is on the left of $z_{i+1}$.}\label{fig_end_proof_geodesics}
\end{figure}
We now move on to the induction step. We assume $x_i$ lies strictly on the left of $z_{h(x_i)}$. We recall that $h(x_{i+1}) \leq i+1$, and split the proof in two cases.
\begin{itemize}
\item
If $h(x_{i+1})<i+1$, then $x_{i+1}$ is the last point of $\gamma$ at height at most $i+1$, so it is also the last point of $\gamma$ at height at most $i$, so $x_{i+1}=x_i$. In particular, it is strictly on the left of $z_{h(x_{i+1})}=z_{h(x_i)}$. 
\item
If $h(x_{i+1})=i+1$, we need to introduce some more notation that is summed up on Figure \ref{fig_end_proof_geodesics}. We denote by $x'$ the first point of $\gamma$ after $x_i$ that belongs to $\mathbf{T}$ (note that $h(x') \geq i+1$ by definition of $x_i$), and by $x''$ the ancestor of $x'$ at height $i+1$. Let also $y'$ be the leftmost descendant of $z_{h(x_i)}$ at height $i+1$. Note that by construction of the sequences $(y_i)$ and $(z_i)$, the vertex $y'$ is on the left of $y_{i+1}$. We know that between $x_i$ and $x'$, the path $\gamma$ does not cross $\mathbf{T}$, so $x_i$ and $x'$ must be adjacent to the same strip. Since $x_i$ is strictly on the left of $z_{h(x_i)}$, it implies that $x'$ lies on the left of any descendant of $z_{h(x_i)}$, so $x''$ is on the left of $y'$, and therefore on the left of $y_{i+1}$.
Moreover, by the definition of $x_i$, all the vertices of $\gamma$ between $x_i$ and $x_{i+1}$ that belong to $\mathbf{T}$ have height at least $i+1$. By the same argument as before, the length of the part of $\gamma$ between $x_i$ and $x_{i+1}$ is at most $h(x_i)+h(x_{i+1}) \leq 2(i+1)<\widetilde{u}_{i+1}$. Hence, this part cannot cross $\widetilde{u}_{i+1}$ of the slices $\sli[v]$ with $v \in \partial B_{i+1}(\mathbf{T})$, so the distance between $x''$ and $x_{i+1}$ along $\partial B_{i+1}(\mathbf{T})$ is less than $\widetilde{u}_{i+1}$. Since $x''$ is on the left of $y_{i+1}$ and $z_{i+1}$ is at distance $\widetilde{u}_{i+1}$ on the right of $y_{i+1}$, it follows (see again Figure \ref{fig_end_proof_geodesics}) that $x_{i+1}$ is strictly on the left of $z_{i+1}$, which concludes the induction and the proof of the proposition.
\end{itemize}
\end{proof}

\subsection{Weak anchored hyperbolicity and bi-infinite geodesics}

We can now deduce Theorem \ref{thm_1_bis} from Proposition \ref{geodesics_contains_bdd_point}.

\begin{proof}[Proof of point 1 of Theorem \ref{thm_1_bis}]
Let $(a_i)_{1 \leq i \leq 4}$ be four points of $\mathbf{T}$, neither of which is an ancestor of another. The slices $\sli[a_i]$ are disjoint and satisfy the assumptions of Proposition \ref{geodesics_contains_bdd_point}. For every $1 \leq i \leq 4$, let $K_i$ be given by Proposition \ref{geodesics_contains_bdd_point} for $\sli[a_i]$. Now consider three vertices $x$, $y$, $z$ of $\m$ and three geodesics $\gamma_{xy}$ (resp. $\gamma_{yz}$, $\gamma_{zx}$) from $x$ to $y$ (resp. $y$ to $z$, $z$ to $x$) that surround $\rho$. There is an index $1 \leq i \leq 4$ such that $\sli[a_i]$ contains none of the points $x$, $y$ and $z$. Assume it is $\sli[a_1]$. Since the triangle formed by $\gamma_{xy}$, $\gamma_{yz}$ and $\gamma_{zx}$ surrounds $\rho$, one of these three geodesics must either intersect the path in $\mathbf{T}$ from $\rho$ to $a_1$, or cross the slice $\sli[a_1]$, as on Figure \ref{fig_proof_weak_gromov}. We assume this geodesic is $\gamma_{xy}$. In the first case, $\gamma_{xy}$ contains a point at distance at most $h(a_1)$ from $\rho$. In the second case, assume $\gamma_{xy}$ crosses $\sli[a_1]$ from left to right. Let $\gamma_{\ell}$ and $\gamma_r$ be respectively the left and right boundaries of $\sli[a_1]$. Let $v$ be the last point of $\gamma_{xy}$ that lies on $\gamma_{\ell}$ and let $w$ be the first point of $\gamma_{xy}$ after $v$ that lies on $\gamma_r$. Then the portion of $\gamma_{xy}$ between $v$ and $w$ is a geodesic in $\m$ so it is also a geodesic in $\sli[a_1]$ that crosses $\sli[a_1]$. Hence, it contains a point $z$ such that $d(z,a_1) \leq K_1$. This concludes the proof by taking
\[K=\max_{1 \leq i \leq 4} \left( K_i+h(a_i) \right).\]
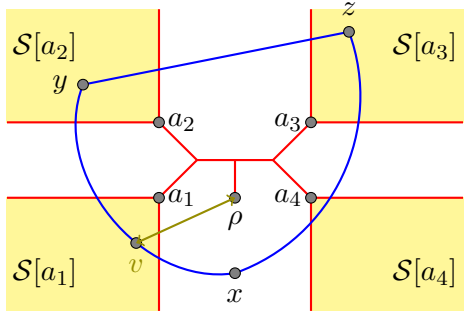
\begin{figure}
\begin{center}
\begin{tikzpicture}
\fill[yellow!50](-1,0)--(-3,0)--(-3,-1.5)--(-1,-1.5)--(-1,0);
\fill[yellow!50](-1,1)--(-3,1)--(-3,2.5)--(-1,2.5)--(-1,1);
\fill[yellow!50](1,0)--(3,0)--(3,-1.5)--(1,-1.5)--(1,0);
\fill[yellow!50](1,1)--(3,1)--(3,2.5)--(1,2.5)--(1,1);

\draw[thick, red](0,0)--(0,0.5);
\draw[thick, red](-0.5,0.5)--(0.5,0.5);
\draw[thick, red](-0.5,0.5)--(-1,0);
\draw[thick, red](-0.5,0.5)--(-1,1);
\draw[thick, red](0.5,0.5)--(1,0);
\draw[thick, red](0.5,0.5)--(1,1);

\draw[thick, red](-1,0)--(-3,0);
\draw[thick, red](-1,0)--(-1,-1.5);
\draw[thick, red](-1,1)--(-3,1);
\draw[thick, red](-1,1)--(-1,2.5);
\draw[thick, red](1,0)--(3,0);
\draw[thick, red](1,0)--(1,-1.5);
\draw[thick, red](1,1)--(3,1);
\draw[thick, red](1,1)--(1,2.5);

\draw(0,0)node{};
\draw(-1,0)node{};
\draw(-1,1)node{};
\draw(1,0)node{};
\draw(1,1)node{};
\draw(0,-0.3)node[texte]{$\rho$};
\draw(-0.7,0)node[texte]{$a_1$};
\draw(-0.7,1)node[texte]{$a_2$};
\draw(0.7,1)node[texte]{$a_3$};
\draw(0.7,0)node[texte]{$a_4$};

\draw(-2.5,-1)node[texte]{$\sli[a_1]$};
\draw(-2.5,2)node[texte]{$\sli[a_2]$};
\draw(2.5,2)node[texte]{$\sli[a_3]$};
\draw(2.5,-1)node[texte]{$\sli[a_4]$};

\draw[thick, blue](0,-1) to[bend left=60] (-2,1.5);
\draw[thick, blue](-2,1.5)--(1.5,2.2);
\draw[thick, blue](1.5,2.2) to[bend left=45] (0,-1);

\draw(0,-1)node{};
\draw(-2,1.5)node{};
\draw(1.5,2.2)node{};

\draw(0,-1.3)node[texte]{$x$};
\draw(-2.3,1.5)node[texte]{$y$};
\draw(1.5,2.5)node[texte]{$z$};

\draw(-1.3,-0.6)node{};
\draw[olive](-1.3,-0.9)node[texte]{$v$};
\draw[olive, thick, <->](-1.3,-0.6)--(0,0);
\end{tikzpicture}
\end{center}
\caption{Illustration of the proof of point 1 of Theorem \ref{thm_1_bis}. Here $\sli[a_1]$ contains none of the vertices $x$, $y$ and $z$, so it is crossed by the geodesic from $x$ to $y$, and contains a point $v$ at bounded distance from $\rho$.}\label{fig_proof_weak_gromov}
\end{figure}
\end{proof}

\begin{proof}[Proof of point 2 of Theorem \ref{thm_1_bis}]
Let $a_1,a_2 \in \mathbf{T}$, neither of which is an ancestor of the other, so that $\sli[a_1]$ and $\sli[a_2]$ are disjoint. Let $\gamma_{\ell}$ and $\gamma_r$ be the left and right boundaries of $\sli[a_1]$. The idea of our construction is the following: we first "approximate" the paths $\gamma_{\ell}$ and $\gamma_r$ by two infinite geodesics $\widetilde{\gamma}_{\ell}$ and $\widetilde{\gamma}_r$, and we then try to connect $\widetilde{\gamma}_{\ell}$ to $\widetilde{\gamma}_r$ in the shortest possible way. Note that the first step is needed because $\gamma_r$ and $\gamma_{\ell}$ may not be geodesics in the general case. Before making this construction explicit, we need to reinforce slightly Proposition \ref{geodesics_contains_bdd_point}.

By Proposition \ref{geodesics_contains_bdd_point}, we know that any geodesic $\gamma$ in $\sli[a_1]$ between a point of $\gamma_{\ell}$ and a point of $\gamma_r$ contains a vertex at bounded distance from $\rho$. We claim that this is also the case if we consider geodesics in $\m$ instead of $\sli[a_1]$. Indeed, let $K_1$ (resp. $K_2$ ) be given by Proposition \ref{geodesics_contains_bdd_point} for $\sli[a_1]$ (resp. $\sli[a_2]$).
Let $i,j \geq 0$ and let $\gamma$ be a geodesic from $\gamma_{\ell}(i)$ to $\gamma_r(j)$ in $\m$. We are in one of the three following cases (cf. Figure \ref{3_cases_gamma}) :
\begin{enumerate}
\item[(i)]
$\gamma$ intersects the path in $\mathbf{T}$ from $\rho$ to $a_1$ or from $\rho$ to $a_2$,
\item[(ii)]
$\gamma$ crosses $\sli[a_1]$,
\item[(iii)]
$\gamma$ crosses $\sli[a_2]$.
\end{enumerate}
In all three cases, $\gamma$ contains a point at distance at most $K$ from $\rho$, where
\begin{equation}\label{K_geodesic_in_m}
K=\max\left( h(a_1)+K_1, h(a_2)+K_2 \right).
\end{equation}

\begin{figure}
\begin{center}
\begin{tikzpicture}
\fill[yellow!50] (2,2) to[bend right] (4,-1)--(4,3) to[bend right] (2,2);
\fill[yellow!50] (-1,2) to[bend right] (-3,3)--(-3,0) to[bend right] (-1,2);

\draw[red, thick] (0,0)--(0,1);
\draw[red, thick] (0,1)--(2,2);
\draw[red, thick] (0,1)--(-1,2);
\draw[red, thick] (2,2)to[bend right](4,-1);
\draw[red, thick] (2,2)to[bend left](4,3);
\draw[red, thick] (-1,2)to[bend left](-3,0);
\draw[red, thick] (-1,2)to[bend right](-3,3);
\draw[blue, thick, dashed] (3.5, 2.98)--(3,-0.32);
\draw[blue, thick, dashed] (3.5, 2.98) to[out=180,in=90] (0.5,1.5);
\draw[blue, thick, dashed] (0.5,1.5) to[out=270,in=180] (3,-0.32);
\draw[blue, thick, dashed] (3.5, 2.98) to[out=150,in=90] (-2,1);
\draw[blue, thick, dashed] (-2,1) to[out=270,in=210] (3,-0.32);

\draw (0,0)node{};
\draw (1,1.5)node{};
\draw (-0.5,1.5)node{};
\draw (3.5,2.98)node{};
\draw (3,-0.32)node{};

\draw[red] (0,-0.3) node[texte]{$\rho$};
\draw[red] (1,1.2) node[texte]{$a_1$};
\draw[red] (-0.75,1.25) node[texte]{$a_2$};
\draw (4.8,1) node[texte]{$\sli[a_1]$};
\draw (-4,1.5) node[texte]{$\sli[a_2]$};
\draw[red] (3.7,3.3) node[texte]{$\gamma_{\ell}(i)$};
\draw[red] (3,-0.8) node[texte]{$\gamma_r(j)$};
\draw[blue] (4.2,2) node[texte]{Case (ii)};
\draw[blue] (0,2.3) node[texte]{Case (i)};
\draw[blue] (-1,3.5) node[texte]{Case (iii)};
\end{tikzpicture}
\end{center}
\caption{Reinforcement of Proposition \ref{geodesics_contains_bdd_point}. In blue, the geodesic $\gamma$. It must intersect a geodesic from $\rho$ to $a_1$ or $a_2$, or cross $\sli[a_1]$ or $\sli[a_2]$.}\label{3_cases_gamma}
\end{figure}
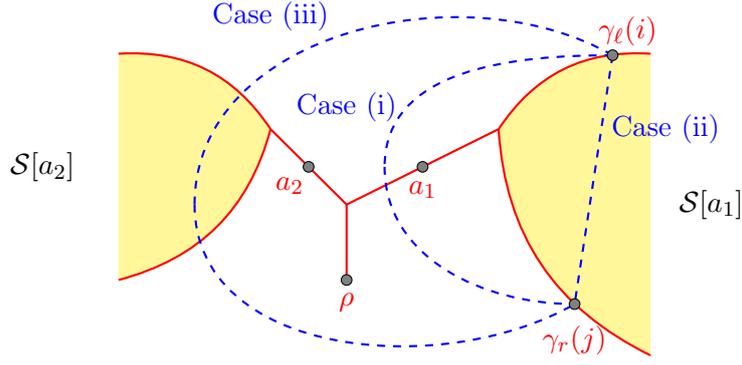

We can now build our infinite geodesics $\widetilde{\gamma}_{\ell}$ and $\widetilde{\gamma}_r$. For every $n$, let $\gamma_{\ell}^n$ be a geodesic from $\rho$ to $\gamma_{\ell}(n)$. By an easy compactness argument, there is an infinite geodesic $\widetilde{\gamma}_{\ell}$ such that, for every $i \geq 0$, there are infinitely many $n$ such that $\widetilde{\gamma}_{\ell}(i)=\gamma_{\ell}^n(i)$. We build an infinite geodesic $\widetilde{\gamma}_r$ from $\gamma_r$ in a similar way. 

For every $i,j \geq 0$, we define
\[a_{i,j}=i+j-d_{\m} \left( \widetilde{\gamma}_{\ell}(i), \widetilde{\gamma}_r(j) \right).\]
This quantity measures "by how much" the concatenation of $\widetilde{\gamma}_{\ell}$ and $\widetilde{\gamma}_r$ between $\widetilde{\gamma}_{\ell}(i)$ and $\widetilde{\gamma}_r(j)$ is not a geodesic. We note that for any $i,j \geq 0$, we have
\[d_{\m} \left( \widetilde{\gamma}_{\ell}(i+1), \widetilde{\gamma}_r(j) \right) \leq d_{\m} \left( \widetilde{\gamma}_{\ell}(i+1), \widetilde{\gamma}_{\ell}(i) \right)+d_{\m} \left( \widetilde{\gamma}_{\ell}(i), \widetilde{\gamma}_r(j) \right)=1+d_{\m} \left( \widetilde{\gamma}_{\ell}(i), \widetilde{\gamma}_r(j) \right),\]
so $a_{i+1,j} \geq a_{i,j}$, so $a_{i,j}$ is nondecreasing in $i$. Similarly, it is nondecreasing in $j$. We claim the following.

\begin{lem}\label{aij_bounded}
Almost surely, $(a_{i,j})_{i,j \geq 0}$ is bounded.
\end{lem}

\begin{proof}[Proof of Lemma \ref{aij_bounded}]
Let $i, j\geq 0$. By the definition of $\widetilde{\gamma}_{\ell}$ and $\widetilde{\gamma}_r$, there are two indices $m$ and $n$ such that $\widetilde{\gamma}_{\ell}(i)$ lies on a geodesic from $\rho$ to $\gamma_{\ell}(m)$ and $\widetilde{\gamma}_r(j)$ lies on a geodesic from $\rho$ to $\gamma_{r}(n)$. Therefore, we have
\begin{eqnarray*}
d_{\m} \left( \gamma_{\ell}(m), \gamma_r(n) \right) & \leq & d_{\m} \left( \gamma_{\ell}(m), \widetilde{\gamma}_{\ell}(i) \right) + d_{\m} \left( \widetilde{\gamma}_{\ell}(i), \widetilde{\gamma}_r(j) \right) + d_{\m} \left( \widetilde{\gamma}_r(j), \gamma_r(n) \right) \\
&=& d_{\m} \left( \rho, \gamma_{\ell}(m) \right) -i + d_{\m} \left( \widetilde{\gamma}_{\ell}(i), \widetilde{\gamma}_r(j) \right) + d_{\m} \left( \rho, \gamma_r(n) \right) -j\\
&=& d_{\m} \left( \rho, \gamma_{\ell}(m) \right) + d_{\m} \left( \rho, \gamma_r(n) \right) -a_{i,j}.
\end{eqnarray*}
On the other hand, we know that for any $m,n \geq 0$, any geodesic from $\gamma_{\ell}(m)$ to $\gamma_r(n)$ contains a vertex $v_0$ at distance at most $K$ from $\rho$, where $K$ is given by \eqref{K_geodesic_in_m}. Therefore, we have
\begin{eqnarray*}
d_{\m} \left( \gamma_{\ell}(m), \gamma_r(n) \right) &=& d_{\m} \left( \gamma_{\ell}(m), v_0 \right) + d_{\m} \left( v_0, \gamma_r(n) \right)\\
&\geq & d_{\m} \left( \gamma_{\ell}(m), \rho \right) + d_{\m} \left( \gamma_r(n), \rho \right)- 2 d_{\m}(\rho,v_0)\\
& \geq & d_{\m} \left( \gamma_{\ell}(m), \rho \right) + d_{\m} \left( \gamma_r(n), \rho \right)-2K.
\end{eqnarray*}
By combining the last two equations, we obtain $a_{i,j} \leq 2K$ for every $i,j \geq 0$.
\end{proof}

The construction of a bi-infinite geodesic is now easy. Let $i_0, j_0$ be two indices such that $a_{i_0,j_0}=\sup \{a_{i,j} | i,j \geq 0\}$, and let $d=d_{\m} \left( \widetilde{\gamma}_{\ell}(i_0), \widetilde{\gamma}_r(j_0) \right)$. Let also $\widehat{\gamma}$ be a geodesic from $\widetilde{\gamma}_{\ell}(i_0)$ to $\widetilde{\gamma}_r(j_0)$ in $\m$. We define a bi-infinite path $\gamma$ as follows :
\[
\gamma(i)=\begin{cases}
\widetilde{\gamma}_{\ell}(i_0-i) & \mbox{if  $i \leq 0$}, \\
\widehat{\gamma}(i) &\mbox{if $0 \leq i \leq d$},\\
\widetilde{\gamma}_r(i-d+j_0) & \mbox{if $i \geq d$}.
\end{cases}
\]
We finally check that this is indeed a bi-infinite geodesic. Let $i \geq i_0$ and $j \geq j_0$. We have $a_{i,j}=a_{i_0,j_0}$, so
\[d_{\m} \left( \widetilde{\gamma}_{\ell}(i), \widetilde{\gamma}_r(j) \right)=d+(i-i_0)+(j-j_0),\]
so $d \left( \gamma(i_0-i), \gamma(d+j-j_0) \right)=(d+j-j_0)-(i_0-i)$. Therefore, we have $d \left( \gamma(i'), \gamma(j') \right)=j'-i'$ for $i' \leq 0$ small enough and $j' \geq 0$ large enough, so $\gamma$ is a bi-infinite geodesic.
\end{proof}

\section{Poisson boundary}\label{causal_sec_poisson}

\subsection{General setting}\label{causal_subsec_construct_poisson}

The goal of this subsection is to build a compactification of maps that, as we will later prove, is under some assumptions a realization of their Poisson boundary. We will perform this construction directly in the general framework $\m=\m \left(  \mathbf{T}, (s_i) \right)$. This construction is exactly the same as the construction performed for the PSHIT in Section 3.1 of \cite{B18}.

We recall that $\partial \mathbf{T}$ is the set of infinite rays from $\rho$ in $\mathbf{T}$. If $\gamma, \gamma' \in \partial \mathbf{T}$, we write $\gamma \sim \gamma'$ if $\gamma=\gamma'$ or if $\gamma$ and $\gamma'$ are "consecutive" in the sense that there is no ray between them (in particular, if $\gamma_{\ell}$ and $\gamma_r$ are the leftmost and rightmost rays of $\mathbf{T}$, then $\gamma_{\ell} \sim \gamma_r$). It is equivalent to saying that $\gamma$ and $\gamma'$ are the left and right boundaries of some strip $s_i$ in the map $\m$. Note that a.s., every ray of $\mathbf{T}$ contains infinitely many branching points, so no ray is equivalent to two distinct other rays. It follows that $\sim$ is a.s.~an equivalence relation for which countably many equivalence classes have cardinal $2$, and all the others have cardinal $1$. We write $\bdy=\partial \mathbf{T} / \sim$ and we denote by $\gamma \to \widehat{\gamma}$ the canonical projection from $\partial \mathbf{T}$ to $\bdy$. Finally, for every strip $s_i$, the left and right boundaries of $s_i$ correspond to the same point of $\bdy$, that we denote by $\widehat{\gamma}_i$.

Our goal is now to define a topology on $\m \cup \bdy$. It should be possible to define it by an explicit distance, but such a distance would be tedious to write down, so we prefer to give an "abstract" construction. Let $s_i$ and $s_j$ be two distinct strips of $\m$, and fix $h>0$ such that both $s_i$ and $s_j$ both intersect $B_h(\mathbf{T})$.
Then $\m \backslash \left( B_h (\m) \cup s_i \cup s_j \right)$ has two infinite connected components, that we denote by $\left( s_i, s_j \right)$ and $\left( s_j, s_i \right)$ (the vertices on the boundaries of $s_i$ and $s_j$ do not belong to $\left( s_i, s_j \right)$ and $\left( s_j, s_i \right)$). We also write 
\[ \bd \left( s_i, s_j \right) =\{ \widehat{\gamma} | \mbox{ $\gamma$ is a ray of $\mathbf{T}$ such that $\gamma(k) \in \left( s_i, s_j \right)$ for $k$ large enough} \}.\]
We define $\bd \left( s_j, s_i \right)$ similarly. Note that $\bd \left( s_i, s_j \right)$ and $\bd \left( s_j, s_i \right)$ are disjoint subsets of $\bdy$, and their union is $\bdy \backslash \{ \widehat{\gamma}_i, \widehat{\gamma}_j \}$.

We can now equip the set $\m \cup \bdy$ with the topology generated by the following open sets:
\begin{itemize}
\item
the singletons $\{v\}$, where $v$ is a vertex of $\m$,
\item
the sets $\left( s_i, s_j \right) \cup \bd \left( s_i, s_j \right)$, where $s_i$ and $s_j$ are two distinct strips of $\m$. 
\end{itemize}

This topology is separated (if $\widehat{\gamma}_1 \ne \widehat{\gamma}_2$, then there are two strips separating $\gamma_1$ and $\gamma_2$) and has a countable basis, so it is induced by a distance. Moreover, any open set of our basis intersects $\m$, so $\m$ is dense in $\m \cup \bdy$. Finally, we state an intuitive result about the topology of $\m \cup \bdy$. Its proof in the particular case of the PSHIT can be found in \cite{B18}, and adapts without any change to the general case.

\begin{lem} \label{compactness}
The space $\m \cup \bdy$ is compact, and $\bdy$ is homeomorphic to the unit circle.
\end{lem}

\subsection{Transience away from the boundary in causal slices}

The goal of this section is to prove Proposition \ref{slice_strong_transience}, which is the main tool in the proof of Theorem \ref{thm_2_Poisson}. We recall that $\sli$ is the causal slice associated to a supercritical Galton--Watson tree $T$, and $\partial \sli$ is the \emph{boundary} of $\sli$, i.e. the set of vertices of $\sli$ that are either the leftmost or the rightmost vertex of their generation. We also write $\tau_{\partial \sli} = \min \{ n \geq 0 | X_n \in \partial \sli \}$, where $(X_n)$ is the simple random walk on $\sli$.

\begin{prop}\label{slice_strong_transience}
Almost surely, there is a vertex $x \in \sli$ such that
\[ P_{\sli, x} \left( \tau_{\partial \sli} =+\infty \right)>0. \]
\end{prop}

Note that if such a vertex $x$ exists, then we have $P_{\sli, v} (X_n \notin \partial \sli \mbox{ for $n$ large enough})>0$ for every vertex $v \in \sli$. 
The proof of Proposition \ref{slice_strong_transience} is based on estimates of effective resistances. We will use the following inequality, that holds for every graph and every vertex $x$:
\begin{equation}\label{resistance-hitting}
P_{\sli,x} \left( \tau_{\partial \sli}<+\infty \right) \leq \frac{R_{\eff}^{\sli} (x \leftrightarrow \{ \partial \sli, \infty \})}{R_{\eff}^{\sli} (x \leftrightarrow \partial \sli)} \leq \frac{R_{\eff}^{\sli} (x \leftrightarrow  \infty )}{R_{\eff}^{\sli} (x \leftrightarrow \partial \sli)}.
\end{equation}
For example, this is a particular case of Exercise 2.36 of \cite{LP10}. We will find a sequence $(x_n)$ of vertices satisfying the following two properties:
\begin{enumerate}
\item
we have $R_{\eff}^{\sli} (x_n \leftrightarrow \partial \sli) \to +\infty$ a.s.~when $n \to +\infty$,
\item
for every $n \geq 0$, the resistance $R_{\eff}^{\sli} (x_n \leftrightarrow \infty)$ is stochastically dominated by $R_{\eff}^{\sli} (\rho \leftrightarrow \infty)$. In particular, a.s., $\left( R_{\eff}^{\sli} (x_n \leftrightarrow \infty) \right)$ has a bounded subsequence.
\end{enumerate}
By \eqref{resistance-hitting}, this will guarantee that
\[ P_{\sli, x_n} \left( \tau_{\partial \sli}<+\infty \right) \xrightarrow[n \to +\infty]{} 0\]
along some subsequence, which is enough to prove Proposition \ref{slice_strong_transience}.

We choose for the sequence $(x_n)$ the nonbacktracking random walk on the backbone of $T$. More precisely, we take $x_0=\rho$ and, for every $n \geq 0$, conditionally on $\sli$ and $x_0, \dots, x_n$, the vertex $x_{n+1}$ is chosen uniformly among the children of $x_n$ in $\back(T)$. We can give a "spinal decomposition" of $\back(T)$ along $(x_n)$. We recall that $\boldsymbol{\mu}$ is the offspring distribution of $\back(T)$, cf. \eqref{several_offspring}. For every $n \geq 0$, let $L_n$ (resp. $R_n$) be the number of children of $x_n$ in $\back(T)$ on the left (resp. on the right) of $x_{n+1}$. A vertex $v$ of $\back(T)$ will be called a \emph{spine brother} if the parent of $v$ is equal to $x_n$ for some $n$ but $v \ne x_{n+1}$. Then the pairs $(L_n, R_n)$ are i.i.d. with distribution $\nu$ given by
\begin{equation}\label{equation_ell_r}
\P \left( L_n=\ell, R_n=r \right) = \nu \left( \{ (\ell,r) \} \right)= \frac{1}{r+\ell+1} \boldsymbol{\mu}(r+\ell+1).
\end{equation}
Moreover, conditionally on $(L_n)$ and $(R_n)$, the backbones of the trees of descendants of the spine brothers are i.i.d. Galton--Watson trees with offspring distribution $\boldsymbol{\mu}$. The distribution of $T$ conditionally on this backbone is then given by Theorem \ref{backbone_decomposition}. In particular, for every $n \geq 0$, the tree of descendants of $x_n$ has the same distribution as $T$, so $\sli[x_n]$ has the same distribution as $\sli$.

Therefore, for every $n$, we have
\[R_{\eff}^{\sli} (x_n \leftrightarrow \infty) \leq R_{\eff}^{\sli [x_n]} (x_n \leftrightarrow \infty ),\]
where $R_{\eff}^{\sli [x_n]} (x_n \leftrightarrow \infty )$ has the same distribution as $R_{\eff}^{\sli} (\rho \leftrightarrow \infty )$. This proves the second property that we wanted $(x_n)$ to satisfy.
Hence, it only remains to prove that $R_{\eff}^{\sli} (x_n \leftrightarrow \partial \sli)$ almost surely goes to $+\infty$, which is the goal of the next lemma.

\begin{lem}\label{lem_separating_sets}
There are disjoint vertex sets $(A_k)_{k \geq 0}$, satisfying the following properties:
\begin{enumerate}
\item[(i)]
for any $k \geq 1$, the set $A_k$ separates $\partial \sli$ from all the sets $A_i$ with $i > k$, and from $x_n$ for $n$ large enough,
\item[(ii)]
the parts of $\sli$ lying between $A_{2k}$ and $A_{2k+1}$ for $k \geq 0$ are i.i.d.,
\item[(iii)]
we have $R_{\eff}^{\sli} (A_{2k} \leftrightarrow A_{2k+1})>0$ a.s.~for every $k \geq 0$.
\end{enumerate}
\end{lem}

\begin{proof}[Proof of Proposition \ref{slice_strong_transience} given Lemma \ref{lem_separating_sets}]
Fix $k$ and choose $n$ large enough, so that $x_n$ is separated from $\partial \sli$ by $A_0, A_1, \dots, A_{2k-1}$. Since $A_0, \dots, A_{2k-1}$ are disjoint cutsets separating $x_n$ from $\partial \sli$, we have
\begin{eqnarray*}
R_{\eff}^{\sli} (x_n \leftrightarrow \partial \sli) & \geq & R_{\eff}^{\sli} (\partial \sli \leftrightarrow A_0)  + R_{\eff}^{\sli} (x_n \leftrightarrow A_{2k-1}) + \sum_{i=0}^{2k-2} R_{\eff}^{\sli} (A_i \leftrightarrow A_{i+1})\\
& \geq & \sum_{i=0}^{k-1} R_{\eff}^{\sli} (A_{2i} \leftrightarrow A_{2i+1}).
\end{eqnarray*}
Since the variables $R_{\eff}^{\sli} (A_{2i} \leftrightarrow A_{2i+1})$ for $0 \leq i \leq k-1$ are i.i.d. and a.s.~positive, this goes to $+\infty$ as $k \to +\infty$, so we have $R_{\eff}^{\sli} (x_n \leftrightarrow A_0) \to +\infty$ a.s.~when $n \to +\infty$, which ends the proof.
\end{proof}

\begin{rem}
The law of large numbers even shows that there is a constant $c>0$ such that, for $n$ large enough, we have $R_{\eff}^{\sli} (x_n \leftrightarrow \partial \sli) \geq cn$. This is quite similar to the resistance estimates proved in \cite{Ben14} in the case where $T$ is the complete binary tree, and it might be interesting to apply our results to the study of the Gaussian free field on causal maps. Unfortunately, our estimates only hold in "typical" directions, and not uniformly for all the vertices. We also note that the idea to "cut" $\sli$ along the sets $A_k$ was inspired by the proof of Lemma 1 of \cite{Ben14}.
\end{rem}

We now build the subsets $A_k$. We define by induction the heights $h_k$ and $h'_k$ for $k \geq 0$ by
\begin{eqnarray*}
h_0 &=& \min \{n \geq 0 | L_n>0\},\\
h'_{k} &=& \min \{ n>h_k | R_n>0 \},\\
h_{k+1} &=& \min \{n>h'_k | L_n>0\}.
\end{eqnarray*}
Note that the pairs $(L_n,R_n)$ are i.i.d. with $\P (L_n>0)>0$ and $\P (R_n>0)>0$, so $h_k$ and $h'_k$ are a.s.~well-defined for every $k$.
We define $A_k$ as the union of the leftmost ray of $\back(T)$ from $x_{h_k}$, the rightmost ray of $\back(T)$ from $x_{h'_k}$ and the vertices $x_i$ with $h_k \leq i \leq h'_k$ (see Figure \ref{figure_separating_sets}).

\begin{figure}
\begin{center}
\begin{tikzpicture}

\draw[red, very thick](0,0)--(0,6.5);
\draw[blue, very thick](-0.04,1)--(-0.04,3);
\draw[blue, very thick](-0.04,5)--(-0.04,6);
\draw[blue, very thick](0,1)--(-2,2);
\draw(0,1)--(2,2);
\draw[blue, very thick](-2,2)--(-3,3);
\draw(-2,2)--(-2,3);
\draw (0,2)--(-1,3);
\draw(0,2)--(1,3);
\draw(2,2)--(2,3);
\draw(2,2)--(3,3);
\draw[blue, very thick](-3,3)--(-4,4);
\draw(-2,3)--(-3,4);
\draw(-2,3)--(-2,4);
\draw(-1,3)--(-1.3,4);
\draw(-1,3)--(-0.7,4);
\draw[blue, very thick] (0,3)--(1.5,4);
\draw(3,3)--(3.5,4);
\draw[blue, very thick](-4,4)--(-4,5);
\draw(-2,4)--(-3,5);
\draw(-1.3,4)--(-2,5);
\draw(0,4)--(1,5);
\draw(1.5,4)--(1.5,5);
\draw[blue, very thick] (1.5,4)--(2,5);
\draw(3.5,4)--(3,5);
\draw(3.5,4)--(4,5);
\draw[blue, very thick](-4,5)--(-4.5,6);
\draw(-4,5)--(-4,6);
\draw(-3,5)--(-3,6);
\draw(-2,5)--(-2.5,6);
\draw (-2,5)--(-2,6);
\draw(-2,5)--(-1.5,6);
\draw(0,5)--(-1,6);
\draw[blue, very thick] (0,5)--(-0.5,6);
\draw(1,5)--(1,6);
\draw(1.5,5)--(1.5,6);
\draw[blue, very thick] (2,5)--(2,6);
\draw(3,5)--(3,6);
\draw(3,5)--(4,6);
\draw(4,5)--(5,6);
\draw[blue, very thick](-4.5,6)--(-4.5,6.5);
\draw(-4,6)--(-4,6.5);
\draw(-3,6)--(-3.2,6.5);
\draw(-3,6)--(-2.8,6.5);
\draw(-2,6)--(-2.2,6.5);
\draw(-2,6)--(-1.8,6.5);
\draw(-1.5,6)--(-1.3,6.5);
\draw[blue, very thick] (-0.5,6)--(-0.5,6.5);
\draw[blue, very thick] (0,6)--(0.4,6.5);
\draw(1,6)--(1,6.5);
\draw(1.5,6)--(1.3,6.5);
\draw(1.5,6)--(1.7,6.5);
\draw[blue, very thick] (2,6)--(2.2,6.5);
\draw(3,6)--(2.8,6.5);
\draw(3,6)--(3.2,6.5);
\draw(5,6)--(5,6.5);

\draw(-2,2)--(2,2);
\draw(-3,3)--(3,3);
\draw(-4,4)--(3.5,4);
\draw(-4,5)--(4,5);
\draw(-4.5,6)--(5,6);

\draw(-2,2) node[petitbleu]{};
\draw(2,2) node{};
\draw(-3,3) node[petitbleu]{};
\draw(-2,3) node{};
\draw(-1,3) node{};
\draw(1,3) node{};
\draw(2,3) node{};
\draw(3,3) node{};
\draw(-4,4) node[petitbleu]{};
\draw(-3,4) node{};
\draw(-2,4) node{};
\draw(-1.3,4) node{};
\draw(-0.7,4) node{};
\draw(1.5,4) node[petitbleu]{};
\draw(3.5,4) node{};
\draw(-4,5) node[petitbleu]{};
\draw(-3,5) node{};
\draw(-2,5) node{};
\draw(1,5) node{};
\draw(1.5,5) node{};
\draw(2,5) node[petitbleu]{};
\draw(3,5) node{};
\draw(4,5) node{};
\draw(-4.5,6) node[petitbleu]{};
\draw(-4,6) node{};
\draw(-3,6) node{};
\draw(-2.5,6) node{};
\draw(-2,6) node{};
\draw(-1.5,6) node{};
\draw(-1,6) node{};
\draw(-0.5,6) node[petitbleu]{};
\draw(1,6) node{};
\draw(1.5,6) node{};
\draw(2,6) node[petitbleu]{};
\draw(3,6) node{};
\draw(4,6) node{};
\draw(5,6) node{};
\draw(0,0) node[petitrouge]{};
\draw(0,1) node[rouge]{};
\draw(0,1) node[petitbleu]{};
\draw(0,2) node[rouge]{};
\draw(0,2) node[petitbleu]{};
\draw(0,3) node[rouge]{};
\draw(0,3) node[petitbleu]{};
\draw(0,4) node[petitrouge]{};
\draw(0,5) node[rouge]{};
\draw(0,5) node[petitbleu]{};
\draw(0,6) node[rouge]{};
\draw(0,6) node[petitbleu]{};

\draw(0.3,0) node[texte]{$\rho$};
\draw[red](0.1,6.8) node[texte]{$(x_n)$};
\draw[blue](-4.5,6.8) node[texte]{$A_0$};
\draw[blue](-0.6,6.8) node[texte]{$A_1$};
\draw[blue](0.4,0.9) node[texte]{$x_{h_0}$};
\draw[blue](-0.3,3.3) node[texte]{$x_{h'_0}$};
\draw[blue](-0.3,4.7) node[texte]{$x_{h_1}$};
\draw[blue](0.4,5.7) node[texte]{$x_{h'_1}$};

\end{tikzpicture}
\end{center}
\caption{The slice $\sli$ with the sequences of vertices $(x_n)$ (in red), and the separating sets $A_k$ (in blue). Here we have $h_0=1$, $h'_0=3$, $h_1=5$ and $h'_1=6$.}\label{figure_separating_sets}
\end{figure}
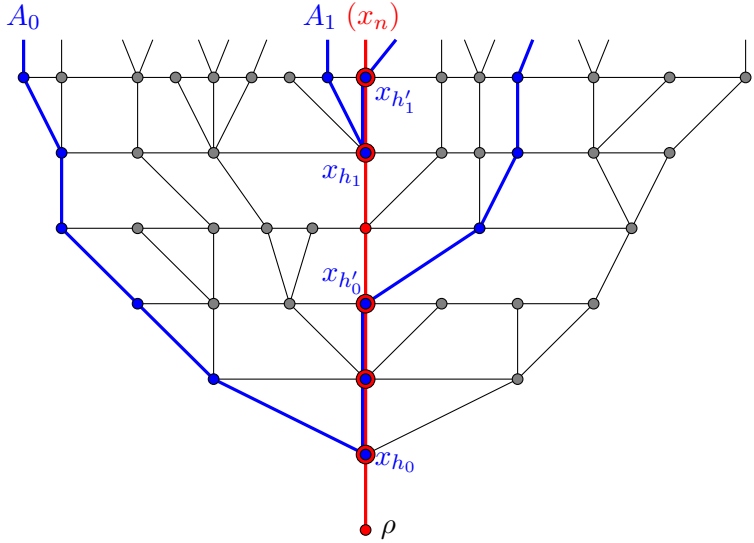

It is easy to see that the sets $A_k$ are disjoint and that $A_k$ separates $\partial \sli$ from $A_i$ for every $i>k$, and from $x_n$ for $n$ large enough, so they satisfy property (i) of Lemma \ref{lem_separating_sets}.

For every $k \geq 0$, let $\uu_k$ be the sub-map of $\sli$ whose vertices are the vertices between $A_k$ and $A_{k+1}$ (the vertices of $A_k$ and $A_{k+1}$ are included), rooted at $x_{h_k}$. By using \eqref{equation_ell_r} and the backbone decomposition, it is quite straightforward to prove that the maps $\uu_{2k}$ for $k \geq 0$ are i.i.d.. We do not give a precise description of the distribution of $\uu_k$, the only property of $\uu_k$ that we will need later is that it contains a copy of $T$ on each side of the spine.


\begin{rem}
It is still true that the $\uu_k$ for $k \geq 0$ are identically distributed. However, $\uu_k$ and $\uu_{k+1}$ are not independent. Indeed, for $h_{k+1} \leq n<h'_{k+1}$, the variables $L_n$ and $R_n$ are not independent, and $L_n$ affects $\uu_{k+1}$, whereas $R_n$ affects $\uu_k$. This is why we restrict ourselves to even values of $k$.
\end{rem}

It only remains to prove property (iii) in Lemma \ref{lem_separating_sets}. Let $\uu$ be distributed as the $\uu_k$, and let $A_b$ (resp. $A_t$) be its bottom (resp. top) boundary, playing the same role as $A_k$ (resp. $A_{k+1}$).
The proof that $R_{\eff}^{\uu} \left( A_b \leftrightarrow A_t \right)>0$ relies on a duality argument. We first recall a classical result about duality of resistances in planar maps. Let $M$ be a finite planar map drawn in the plane, and let $a$ and $z$ be two vertices adjacent to its outer face. We draw two infinite half-lines from $a$ and $z$ that split the outer face in two faces $a^*$ and $z^*$. Now let $M^*$ be the dual planar map whose vertices are $a^*$, $z^*$ and the internal faces of $M$. Then we have
\begin{equation}\label{duality_resistances_finite}
R_{\eff}^{M} (a \leftrightarrow z) = \left( R_{\eff}^{M^*} (a^* \leftrightarrow z^*) \right)^{-1}.
\end{equation}
In our case, the infinite graph $\uu$ has two ends $\infty_{\ell}$ (on the left of the spine) and $\infty_r$ (on the right). Informally, we would like to write
\begin{equation}\label{duality_resistances_infty}
R_{\eff}^{\uu} (A_b \leftrightarrow A_t) = \left( R_{\eff}^{\uu^*} (\infty^*_{\ell} \leftrightarrow \infty^*_r) \right)^{-1},
\end{equation}
which would reduce the problem to the proof of $R_{\eff}^{\uu^*} (\infty^*_{\ell} \leftrightarrow \infty^*_r)<+\infty$. Our first job will be to state and prove (a proper version of) \eqref{duality_resistances_infty}.

More precisely, we denote by $\uu_{\ell}$ (resp. $\uu_r$) the part of $\uu$ lying on the left (resp. on the right) of the spine. We also define $\uu^*$ as the dual map of $\uu$ in the following sense: the vertices of $\uu^*$ are the finite faces of $\uu$, and for every edge of $\uu$ that does not link two vertices of $A_b$ or two vertices of $A_t$, we draw an edge $e^*$ between the two faces adjacent to $e$. Let $\theta^*$ be a flow on $\uu^*$ with no source. We assume that $\theta^*$ is unitary in the sense that the mass of $\theta^*$ crossing the spine from left to right is equal to $1$. We recall that the \emph{energy} $\mathcal{E}(\theta^*)$ of $\theta^*$ is the sum over all edges $e^*$ of $\uu^*$ of $\theta^*(e^*)^2$. For every $n \geq 0$, let $A_b(n)$ (resp. $A_t(n)$) be the set of vertices of $A_b$ (resp. $A_t$) at height at most $n$. We consider the map $\uu(n)$ obtained by cutting $\uu$ above height $n$. The restriction of $\theta$ to the dual of this map becomes a unitary flow crossing $\uu(n)^*$ from left to right. Therefore, the dual resistance from left to right in $\uu(n)^*$ is at most $\mathcal{E}(\theta^*)$ so, by \eqref{duality_resistances_finite}, we obtain
\[R_{\eff}^{\uu} (A_b(n) \leftrightarrow A_t(n)) \geq \mathcal{E}(\theta^*)^{-1}\]
and, by letting $n \to +\infty$, we get $R_{\eff}^{\uu} (A_b \leftrightarrow A_t) \geq \mathcal{E}(\theta^*)^{-1}$.
In particular, if there is such a flow $\theta^*$ with finite energy, then $R_{\eff}^{\uu} (A_b \leftrightarrow A_t)>0$ and Lemma \ref{lem_separating_sets} is proved.

We now define $\uu_{\ell}$ (resp. $\uu_r$) as the part of $\uu$ lying on the left (resp. on the right) of the spine. Let $f_{\ell}$ and $f_r$ be two faces of $\uu$ lying respectively on the left and on the right of the same edge of the spine. A simple way to construct a unitary flow $\theta^*$ with no sources is to concatenate a flow $\theta^*_{\ell}$ from infinity to $f_{\ell}$ in $\uu_{\ell}^*$, a flow of mass $1$ in the dual edge from $f_{\ell}$ to $f_r$ and a flow $\theta_r^*$ from $f_r$ to infinity in $\uu_r^*$. For this flow to have finite energy, we need $\theta^*_{\ell}$ and $\theta_r^*$ to have finite energy, so we need both $\uu_{\ell}^*$ and $\uu_r^*$ to be transient.

We now define $\sli^*$ as the dual of the slice $\sli$ (as above, the vertices of $\sli^*$ are the inner faces of $\sli$). We note that, for every vertex $v_0 \in \uu_{\ell} \cap \back(T)$ that does not belong to the spine, the tree of descendants of $v_0$ has the same distribution as $T$ and is entirely contained in $\uu_{\ell}$. Therefore, $\uu_{\ell}$ contains a copy of $\sli$, so $\uu_ {\ell}^*$ contains a copy of $\sli^*$, and the same is true for $\uu_r^*$. Hence, we have reduced the proof of Proposition \ref{slice_strong_transience} to the next result.

\begin{lem}\label{transience_dual_slice}
The dual slice $\sli^*$ is a.s.~transient.
\end{lem}

\begin{proof}
We will show that we can embed a transient tree in $\sli^*$. The idea will be to follow the branches of the tree $T$ in the dual, to obtain a tree $T^*$ that is similar to $T$. However, vertices of high degree become obstacles: if a vertex $v$ of $T$ has degree $d$ in $T$, we need $d$ dual edges to "move around" $v$ in $\sli^*$. Therefore, it becomes difficult to control the ratio between resistances in $T^*$ and in $T$. To circumvent this problem, we will use the fact that $T$ contains a supercritical Galton--Watson tree with bounded degrees.

More precisely, we fix a constant $c_{\max}$ large enough to have
\[ \sum_{i=0}^{c_{\max}} i \mu(i)>1. \]
If $t$ is a (finite or infinite) tree, for every vertex $v$ with more than $c_{\max}$ children, we remove all the edges between $v$ and its children, and we call $t_{\bdd}$ the connected component of the root. If $T'$ is a Galton--Watson tree with distribution $\mu$, then $T'_{\bdd}$ is a Galton--Watson tree with offspring distribution $\mu_{\bdd}$ given by
\[ \mu_{\bdd} (i)=\begin{cases}
0 & \mbox{if $i>c_{\max}$,}\\
\mu(i) & \mbox{if $0 <i \leq c_{\max}$,}\\
\mu(0) + \sum_{j > c_{\max}} \mu (j) & \mbox{if $i=0$.}
\end{cases} \]
In particular, we have $\sum_i i \mu_{\bdd}(i) = \sum_{i=0}^{c_{\max}} i \mu(i)>1$, so $T'_{\bdd}$ is supercritical, and it survives with positive probability. But $T$ is a Galton--Watson tree conditioned to survive, so it contains infinitely many i.i.d. copies of $T$ (take for examples the trees of descendants of the children of the right boundary). Therefore, there is a.s.~a vertex $v_0 \in T$ that is not on the left boundary of $\sli$, such that $T[v_0]_{\bdd}$ is a Galton--Watson tree and survives. In particular, it is transient and, for every $v \in T[v_0]_{\bdd}$, the number of children of $v$ in $T$ is bounded by $c_{\max}$.


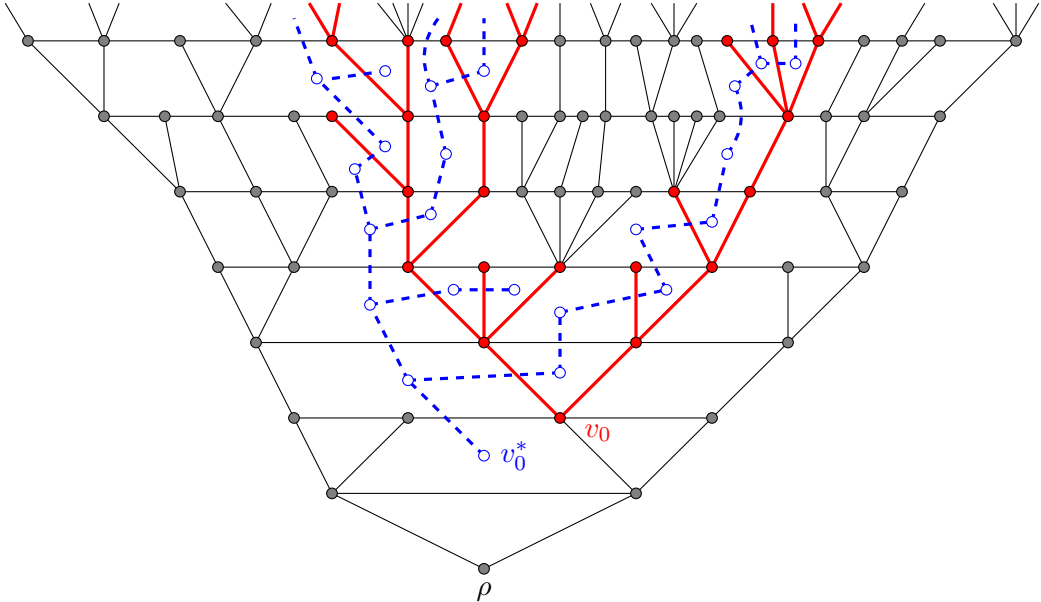
\begin{figure}
\begin{center}
\begin{tikzpicture}
\draw(-2,1)--(2,1);
\draw(-2.5,2)--(3,2);
\draw(-3,3)--(4,3);
\draw(-3.5,4)--(5,4);
\draw(-4,5)--(5.5,5);
\draw(-5,6)--(6,6);
\draw(-6,7)--(7,7);

\draw(0,0)--(-2,1);
\draw(0,0)--(2,1);
\draw(-2,1)--(-2.5,2);
\draw(-2,1)--(-1,2);
\draw(2,1)--(1,2);
\draw(2,1)--(3,2);
\draw(-2.5,2)--(-3,3);
\draw[very thick, red] (1,2)--(0,3);
\draw[very thick, red](1,2)--(2,3);
\draw(3,2)--(4,3);
\draw(-3,3)--(-3.5,4);
\draw(-3,3)--(-2.5,4);
\draw[very thick, red](0,3)--(-1,4);
\draw[very thick, red](0,3)--(0,4);
\draw[very thick, red](0,3)--(1,4);
\draw[very thick, red](2,3)--(2,4);
\draw[very thick, red](2,3)--(3,4);
\draw(4,3)--(4,4);
\draw(4,3)--(5,4);
\draw(-3.5,4)--(-4,5);
\draw(-2.5,4)--(-3,5);
\draw(-2.5,4)--(-2,5);
\draw[very thick, red](-1,4)--(-1,5);
\draw[very thick, red](-1,4)--(0,5);
\draw(1,4)--(0.5,5);
\draw(1,4)--(1,5);
\draw(1,4)--(1.5,5);
\draw(1,4)--(2,5);
\draw[very thick, red](3,4)--(2.5,5);
\draw[very thick, red](3,4)--(3.5,5);
\draw(5,4)--(4.5,5);
\draw(5,4)--(5.5,5);
\draw(-4,5)--(-5,6);
\draw(-4,5)--(-4.2,6);
\draw(-3,5)--(-3.5,6);
\draw(-2,5)--(-2.5,6);
\draw[very thick, red](-1,5)--(-2,6);
\draw[very thick, red](-1,5)--(-1,6);
\draw[very thick, red](0,5)--(0,6);
\draw(0.5,5)--(0.5,6);
\draw(0.5,5)--(1,6);
\draw(1,5)--(1.3,6);
\draw(1.5,5)--(1.6,6);
\draw(2.5,5)--(2.2,6);
\draw(2.5,5)--(2.5,6);
\draw(2.5,5)--(2.8,6);
\draw(2.5,5)--(3.1,6);
\draw[very thick, red](3.5,5)--(4,6);
\draw(4.5,5)--(4.5,6);
\draw(4.5,5)--(5,6);
\draw(5.5,5)--(6,6);

\draw(-5,6)--(-6,7);
\draw(-5,6)--(-5,7);
\draw(-3.5,6)--(-4,7);
\draw(-3.5,6)--(-3,7);
\draw[very thick, red](-1,6)--(-2,7);
\draw[very thick, red](-1,6)--(-1,7);
\draw[very thick, red](0,6)--(-0.5,7);
\draw[very thick, red](0,6)--(0.5,7);
\draw(1,6)--(1,7);
\draw(1.6,6)--(1.6,7);
\draw(2.2,6)--(2,7);
\draw(2.2,6)--(2.5,7);
\draw(3.1,6)--(2.8,7);
\draw[very thick, red](4,6)--(3.2,7);
\draw[very thick, red](4,6)--(3.8,7);
\draw[very thick, red](4,6)--(4.4,7);
\draw(4.5,6)--(5,7);
\draw(5,6)--(5.5,7);
\draw(5,6)--(6,7);
\draw(6,6)--(7,7);
\draw(-6,7)--(-6.3,7.5);
\draw(-5,7)--(-5.2,7.5);
\draw(-5,7)--(-4.8,7.5);
\draw(-3,7)--(-3.4,7.5);
\draw(-3,7)--(-2.8,7.5);
\draw[very thick, red](-2,7)--(-2.3,7.5);
\draw[very thick, red](-2,7)--(-1.9,7.5);
\draw(-1,7)--(-1.4,7.5);
\draw(-1,7)--(-1.2,7.5);
\draw(-1,7)--(-1,7.5);
\draw(-1,7)--(-0.8,7.5);
\draw[very thick, red](-0.5,7)--(-0.3,7.5);
\draw[very thick, red](0.5,7)--(0.3,7.5);
\draw[very thick, red](0.5,7)--(0.7,7.5);
\draw(1,7)--(1,7.5);
\draw(1.6,7)--(1.4,7.5);
\draw(1.6,7)--(1.8,7.5);
\draw(2.5,7)--(2.2,7.5);
\draw(2.5,7)--(2.5,7.5);
\draw(2.5,7)--(2.8,7.5);
\draw[very thick, red](3.8,7)--(3.8,7.5);
\draw[very thick, red](4.4,7)--(4.2,7.5);
\draw[very thick, red](4.4,7)--(4.7,7.5);
\draw(5.5,7)--(5.8,7.5);
\draw(7,7)--(6.7,7.5);
\draw(7,7)--(7,7.5);
\draw(7,7)--(7.3,7.5);

\draw (0,0)node{};
\draw (-2,1)node{};
\draw (2,1)node{};
\draw (-2.5,2)node{};
\draw (-1,2)node{};
\draw (1,2)node[petitrouge]{};
\draw (3,2)node{};
\draw (-3,3)node{};
\draw (0,3)node[petitrouge]{};
\draw (2,3)node[petitrouge]{};
\draw (4,3)node{};
\draw (-3.5,4)node{};
\draw (-2.5,4)node{};
\draw (-1,4)node[petitrouge]{};
\draw (0,4)node[petitrouge]{};
\draw (1,4)node[petitrouge]{};
\draw (2,4)node[petitrouge]{};
\draw (3,4)node[petitrouge]{};
\draw (4,4)node{};
\draw (5,4)node{};
\draw (-4,5)node{};
\draw (-3,5)node{};
\draw (-2,5)node{};
\draw (-1,5)node[petitrouge]{};
\draw (0,5)node[petitrouge]{};
\draw (0.5,5)node{};
\draw (1,5)node{};
\draw (1.5,5)node{};
\draw (2,5)node{};
\draw (2.5,5)node[petitrouge]{};
\draw (3.5,5)node[petitrouge]{};
\draw (4.5,5)node{};
\draw (5.5,5)node{};
\draw(-5,6)node{};
\draw(-4.2,6)node{};
\draw(-3.5,6)node{};
\draw(-2.5,6)node{};
\draw(-2,6)node[petitrouge]{};
\draw(-1,6)node[petitrouge]{};
\draw(0,6)node[petitrouge]{};
\draw(0.5,6)node{};
\draw(1,6)node{};
\draw(1.3,6)node{};
\draw(1.6,6)node{};
\draw(2.2,6)node{};
\draw(2.5,6)node{};
\draw(2.8,6)node{};
\draw(3.1,6)node{};
\draw(4,6)node[petitrouge]{};
\draw(4.5,6)node{};
\draw(5,6)node{};
\draw(6,6)node{};
\draw(-6,7)node{};
\draw(-5,7)node{};
\draw(-4,7)node{};
\draw(-3,7)node{};
\draw(-2,7)node[petitrouge]{};
\draw(-1,7)node[petitrouge]{};
\draw(-0.5,7)node[petitrouge]{};
\draw(0.5,7)node[petitrouge]{};
\draw(1,7)node{};
\draw(1.6,7)node{};
\draw(2,7)node{};
\draw(2.5,7)node{};
\draw(2.8,7)node{};
\draw(3.2,7)node[petitrouge]{};
\draw(3.8,7)node[petitrouge]{};
\draw(4.4,7)node[petitrouge]{};
\draw(5,7)node{};
\draw(5.5,7)node{};
\draw(6,7)node{};
\draw(7,7)node{};

\draw[very thick, dashed, blue](0,1.5)--(-1,2.5);
\draw[very thick, dashed, blue](-1,2.5)--(1,2.6);
\draw[very thick, dashed, blue](-1,2.5)--(-1.5,3.5);
\draw[very thick, dashed, blue](-1.5,3.5)--(-0.4,3.7);
\draw[very thick, dashed, blue](-0.4,3.7)--(0.4,3.7);
\draw[very thick, dashed, blue](1,2.6)--(1,3.4);
\draw[very thick, dashed, blue](1,3.4)--(2.4,3.7);
\draw[very thick, dashed, blue](-1.5,3.5)--(-1.5,4.5);
\draw[very thick, dashed, blue](-1.5,4.5)--(-0.7,4.7);
\draw[very thick, dashed, blue](2.4,3.7)--(2,4.5);
\draw[very thick, dashed, blue](2,4.5)--(3,4.6);
\draw[very thick, dashed, blue](-1.5,4.5)--(-1.7,5.3);
\draw[very thick, dashed, blue](-1.7,5.3)--(-1.3,5.6);
\draw[very thick, dashed, blue](-0.7,4.7)--(-0.5,5.5);
\draw[very thick, dashed, blue](3,4.6)--(3.2,5.5);
\draw[very thick, dashed, blue](-1.3,5.6)--(-2.2,6.5);
\draw[very thick, dashed, blue](-2,6.5)--(-1.3,6.6);
\draw[very thick, dashed, blue](-0.5,5.5)--(-0.7,6.4);
\draw[very thick, dashed, blue](-0.7,6.4)--(0,6.6);
\draw[very thick, dashed, blue](3.2,5.5) to[bend right] (3.3,6.4);
\draw[very thick, dashed, blue](3.3,6.4)--(3.65,6.7);
\draw[very thick, dashed, blue](3.65,6.7)--(4.1,6.7);
\draw[very thick, dashed, blue](-2.2,6.5)--(-2.5,7.3);
\draw[very thick, dashed, blue](-0.7,6.4)to[bend left](-0.6,7.3);
\draw[very thick, dashed, blue](0,6.6)--(0,7.3);
\draw[very thick, dashed, blue](3.65,6.7)--(3.5,7.3);
\draw[very thick, dashed, blue](4.1,6.7)--(4.1,7.3);

\draw (0,1.5) node[dual]{};
\draw (-1,2.5) node[dual]{};
\draw (1,2.6) node[dual]{};
\draw (-1.5,3.5) node[dual]{};
\draw (-0.4,3.7) node[dual]{};
\draw (0.4,3.7) node[dual]{};
\draw (1,3.4) node[dual]{};
\draw (2.4,3.7) node[dual]{};
\draw (-1.5,4.5) node[dual]{};
\draw (-0.7,4.7) node[dual]{};
\draw (2,4.5) node[dual]{};
\draw (3,4.6) node[dual]{};
\draw (-1.7,5.3) node[dual]{};
\draw (-1.3,5.6) node[dual]{};
\draw (-0.5,5.5) node[dual]{};
\draw (3.2,5.5) node[dual]{};
\draw (-2.2,6.5) node[dual]{};
\draw (-1.3,6.6) node[dual]{};
\draw (-0.7,6.4) node[dual]{};
\draw (0,6.6) node[dual]{};
\draw (3.3,6.4) node[dual]{};
\draw (3.65,6.7) node[dual]{};
\draw (4.1,6.7) node[dual]{};

\draw (0,-0.3) node[texte]{$\rho$};
\draw[red] (1.5,1.8) node[texte]{$v_0$};
\draw[blue] (0.4,1.5) node[texte]{$v_0^*$};
\end{tikzpicture}
\end{center}
\caption{The construction of the dual tree $T^*$ (in blue) from the tree $T[v_0]_{\bdd}$ (in red). Here, we have taken $c_{\max}=3$. The vertices of $T^*$ are the faces adjacent to the vertices of $T[v_0]_{\bdd}$ at their bottom-left corners.}\label{figure_dual_tree}
\end{figure}

From here, can can build a tree $T^*$ in $\sli^*$ whose branches follow the branches of $T[v_0]_{\bdd}$ on their left, and which circumvents branching points of $T[v_0]_{\bdd}$ by the top. See Figure \ref{figure_dual_tree} for the construction of this tree. The tree $T^*$ is then a subgraph of $\sli^*$. Therefore, it is enough to prove that $T^*$ is transient.  But since the vertex degrees in $T[v_0]_{\bdd}$ are bounded, it is easy to see that $T[v_0]_{\bdd}$ and $T^*$ are quasi-isometric, so $T^*$ is also transient (by e.g. Section 2.4.4 of \cite{DS84}).
\end{proof}

\subsection{Consequences on the Poisson boundary}\label{causal_subsec_poisson}

We recall that $(X_n)$ is the simple random walk on $\cau$ started from $\rho$. By a result of Hutchcroft and Peres (Theorem 1.3 of \cite{HP15}), the first point of Theorem \ref{thm_2_Poisson} implies the second.

\begin{proof}[Proof of Theorem \ref{thm_2_Poisson}]
We first show the almost sure convergence of $(X_n)$. By compactness (Lemma \ref{compactness}), it is enough to prove that $(X_n)$ a.s.~has a unique subsequential limit in $\widehat{\partial} T$. Note that if $\widehat{\gamma}_1 \ne \widehat{\gamma}_2$ are two distinct points of $\widehat{\partial} T$, then there are two vertices $v,v' \in \back(T)$ such that the slices $\sli[v]$ and $\sli[v']$ separate $\gamma_1$ from $\gamma_2$. Therefore, if $\gamma_1$ and $\gamma_2$ are subsequential limits of $(X_n)$, by transience of $\cau(T)$, the walk $(X_n)$ crosses infinitely many times either $\sli[v]$ or $\sli[v']$ horizontally. Therefore, it is enough to prove that for every $v_0 \in \back(T)$, the walk $(X_n)$ cannot cross infinitely many times $\sli[v_0]$ horizontally.

For every $v \in \sli[v_0]$, let $f(v)=P_{\sli[v_0], v} \left( \tau_{\partial \sli[v_0]} < +\infty \right)$. The function $f$ is harmonic on $\sli[v_0] \backslash \partial \sli[v_0]$.
Moreover, by Proposition \ref{slice_strong_transience}, there is a vertex $v_1 \in \sli[v_0]$ such that $f(v_1)<1$. Let $(Y_n)$ be a simple random walk started from $v_1$ and killed when it hits $\partial \sli[v_0]$. Then $f(Y_n)$ is a martingale and, by the martingale convergence theorem, it converges a.s.~to $\mathbbm{1}_{\tau_{\partial \sli[v_0]}<+\infty}$. In particular, it has limit zero with positive probability, so there is an infinite path $(w_k)$ going to infinity in $\sli[v_0]$, such that $f(w_k) \to 0$.

We fix $k_0>0$. Everytime the walk $(X_n)$ crosses $\sli[v_0]$ horizontally at a large enough height, it must cross the path $(w_k)_{k \geq 0}$. Since $\cau$ is transient, if $X$ crosses $\sli[v_0]$ infinitely many times, it must cross $(w_k)_{k \geq k_0}$, and then hit $\partial \sli[v_0]$. If this happens, let $K$ be such that $w_K$ is the first of the points $(w_k)_{k \geq k_0}$ to be hit by $X$ (if none of these points is hit, we take $K=+\infty$). We have
\[\P \left( \mbox{$X$ hits $(w_k)_{k \geq k_0}$ and then $\partial \sli[v_0]$} \right) = \E \left[ \mathbbm{1}_{K<+\infty} f(w_K) \right] \leq \E \left[ \sup_{k \geq k_0} f(w_k) \right].\]
Since $f(w_k) \to 0$, by dominated convergence, this goes to $0$, which proves that $X$ cannot cross $\sli[v_0]$ infinitely many times. This implies the almost sure convergence of $X$ to a point $X_{\infty}$ of $\widehat{\partial} T$.

The proof that $X_{\infty}$ has full support is quite easy. Let $v_0 \in \back(T)$. Then $(X_n)$ has a positive probability to visit the slice $\sli[v_0]$ and, by Proposition \ref{slice_strong_transience}, it a.s.~has a positive probability to stay there ever after. But if $X_n \in \sli[v_0]$ for $n$ large enough, then $X_{\infty}$ must correspond to a ray of descendants of $v_0$, so the distribution of $X_{\infty}$ gives a positive mass to rays that are descendants of $v_0$. This is almost surely true for any $v_0 \in \back(T)$, so the distribution of $X_{\infty}$ has a.s.~full support.

Finally, to prove the almost sure nonatomicity, it is enough to prove that if $X$ and $Y$ are two independent simple random walks on $\cau$, then $X_{\infty} \ne Y_{\infty}$ almost surely. The idea of the proof is that everytime $X$ and $Y$ reach a new height for the first time, by Proposition \ref{slice_strong_transience}, they have a positive probability to get "swallowed" in two different slices of the form $\sli[x]$ and $\sli[y]$, so this will almost surely happen at some height.

More precisely, \textbf{in this proof and in this proof only}, until the end of Section \ref{causal_subsec_poisson}, we assume that $T$ is a \textbf{non-conditioned} Galton-Watson tree with offspring distribution $\mu$. We recall that $Z_h$ is the number of vertices of $T$ at height $h$. For every $h \geq 0$, let
\[\tau^X_h=\min \{n \geq 0 | h(X_n)=h \} \hspace{5mm} \mbox{ and } \hspace{5mm} \tau^Y_h=\min \{ n \geq 0 | h(Y_n)=h\}.\]
Note that if $T$ survives, then $\tau^X_h, \tau^Y_h <+\infty$ for every $h$. Let also $\f_h$ be the $\sigma$-algebra generated by $B_h  \left( \cau(T) \right)$, $(X_n)_{0 \leq n \leq \tau_h^X}$ and $(Y_n)_{0 \leq n \leq \tau_h^Y}$, and let $\f_{\infty}$ be the $\sigma$-algebra generated by $\bigcup_{h \geq 0} \f_h$. We note right now that $(\f_h)_{h \geq 0}$ is nondecreasing and that $\f_{\infty}$ is the $\sigma$-algebra generated by $(\cau, X, Y)$. Finally, for every $h>0$, let $A_h$ be the event

$\Big\{$There are four distinct vertices $(x_i)_{1 \leq i \leq 4}$ of $T$ at height $h$ such that:
\begin{itemize}
\item
the vertices $x_1$, $x_2$, $x_3$ and $x_4$ lie in this cyclic order,
\item
the trees $T[x_i]$ all survive,
\item
for every $n \geq \tau_h^X$, we have $X_n \in T[x_1]$ 
\item
for every $n \geq \tau_h^Y+2$, we have $Y_n \in T[x_3]. \Big\}$.
\end{itemize}

\begin{lem}\label{lem_separation_XY}
There is a constant $\delta>0$ such that for every $h$, if $Z_h \geq 4$, then
\[ \P \left( A_h | \f_h \right) \geq \delta. \]
\end{lem}

Once this lemma is known, the end of the proof is quite easy: let $A=\bigcup_{h \geq 0} A_h$. If $T$ survives, then $Z_h \geq 4$ for $h$ large enough, so
\[ \delta \leq \P (A_h | \f_h ) \leq \P (A | \f_h ) \xrightarrow[h \to +\infty]{a.s.} \P (A | \f_{\infty} ) =\mathbbm{1}_A, \]
by the martingale convergence theorem and the fact that $(\cau, X, Y)$ is $\f_{\infty}$-measurable. Therefore, almost surely, if $T$ survives, there is an $h$ such that $A_h$ occurs. But if it does, the slices $\sli[x_2]$ and $\sli[x_4]$ separate $X$ and $Y$ eventually, so they separate $X_{\infty}$ from $Y_{\infty}$, so $X_{\infty} \ne Y_{\infty}$, which ends the proof of Theorem \ref{thm_2_Poisson}.
\end{proof}

\begin{rem}\label{rem_poisson_general_setting}
It is easy to show by using Proposition \ref{slice_strong_transience} that $\P \left( A_h | \f_h \right)>0$ a.s.. However, this is not sufficient to prove Lemma \ref{lem_separation_XY}. Indeed, this is the one point in our proof of Theorem \ref{thm_2_Poisson} at which our argument fails to hold in a more general setting.
More precisely, in the setting of a tree $\mathbf{T}$ filled with i.i.d. strips, the lower degree of $X_{\tau_h^X}$ (i.e. the number of edges joining this vertex to a lower vertex) is not constant but depends on $\f_h$, and we might imagine that it goes to $+\infty$ as $h \to +\infty$. In this case, we might have $\P (A_h | \f_h) \to 0$ (with high probability, $X$ goes back down right after $\tau_h^X$). This problem does not occur in $\cau(T)$, where the lower degree is always equal to $1$, but it explains why the proof of Lemma \ref{lem_separation_XY} needs to be treated with some care.
\end{rem}

\begin{proof}[Proof of Lemma \ref{lem_separation_XY}]
The proof will be split into three cases: the case where $X_{\tau^X_h}=Y_{\tau^Y_h}$, the case where $X_{\tau^X_h}$ and $Y_{\tau^Y_h}$ are distinct but neighbours, and the case where they are not neighbours. We treat carefully the first one, which is slightly more complicated than the others.

In the first case, we write $x_1=X_{\tau^X_h}=Y_{\tau^Y_h}$. We also denote by $x_2$ and $x_3$ the two vertices at height $h$ on the right of $x_1$, and by $x_4$ the left neighbour of $x_1$. Let also $A'_h$ be the following event:
\begin{center}
$\Big\{$The trees $T[x_i]$ for $1 \leq i \leq 4$ survive. Moreover, we have $X_n \in T[x_1]$ for every $n \geq \tau_h^X$ and $Y_{\tau_h^Y+1}=x_2$, $Y_{\tau_h^Y+2}=x_3$ and $Y_n \in T[x_3]$ for every  $n \geq \tau_h^Y+2. \Big\}$.
\end{center}
If $A'_h$ occurs, then so does $A_h$. Moreover, we claim that the probability for $A'_h$ to occur is independent of $h$ and $\f_h$. The reason why this is true is that for every vertex $v$ in one of these trees (say $T[x_1]$), the number of neighbours of $v$ in $\cau$ that are not in $\sli[x_1]$ is fixed: there are $3$ such neighbours if $v=x_1$, there is $1$ such neighbour if $v \ne x_1$ is on the boundary of $\sli[x_1]$ and $0$ if it is not. Therefore, the probability given $\cau$ that $X$ stays in $\sli[x_1]$ after time $\tau_h^X$ only depends on $T[x_1]$. Similarly, the probability for $Y$ to perform the right first two steps after time $\tau_h^Y$ only depends on the numbers of children of $x_1$ and $x_2$, and the probability to stay in $T[x_3]$ ever after only depends on $T[x_3]$. All this is independent of $h$ and $\f_h$, which proves our claim. If we write $\delta_1=\P (A'_h | \f_h)$, we have $\delta_1>0$ by Proposition \ref{slice_strong_transience} and $\P \left( A_h | \f_h \right) \geq \delta_1$ for every $h$ in this first case.

The other two cases can be treated similarly with minor adaptations in the choices of the vertices $x_i$, and the first two steps of $Y$ after $\tau_h^Y$. While we needed to control exactly the first two steps of $Y$ in the first case, we only need one step for the second case and zero step for the third one. The other two cases yield two constants $\delta_2$ and $\delta_3$, which proves the lemma by taking $\delta=\min (\delta_1, \delta_2, \delta_3)$.

%

\end{proof}

\subsection{Robustness of Proposition \ref{slice_strong_transience}}

The goal of this subsection is to explain why Proposition \ref{slice_strong_transience} still holds in a quite general setting and to deduce the following result. We recall that a graph $G$ is \emph{Liouville} if its Poisson boundary is trivial, i.e. if every bounded harmonic function on $G$ is constant.

\begin{thm}\label{thm_2_bis}
Let $\mathbf{T}$ be a supercritical Galton--Watson tree with offspring distribution $\mu$ such that $\mu(0)=0$, and let $(S_i)_{i \geq 0}$ be an i.i.d. sequence of random strips. We assume that $(S_i)$ is also independent from $\mathbf{T}$.
\begin{enumerate}
\item
Proposition \ref{slice_strong_transience} holds if we replace $\sli$ by $\sli \left( \mathbf{T}, (S_i)_{i \geq 0} \right)$.
\item
The map $\m \left( \mathbf{T}, (S_i) \right)$ is a.s.~non-Liouville.
\item
If furthermore the strips $(S_i)$ are a.s.~recurrent and bounded-degree, then $\widehat{\partial} \mathbf{T}$ is the Poisson boundary of $\m \left( \mathbf{T}, (S_i) \right)$.
\end{enumerate}
\end{thm}

Note that the assumption that the $S_i$ are recurrent is necessary. For example, if some strips $S_i$ have a non-trivial Poisson boundary, then the Poisson boundary of $\m \left( \mathbf{T}, (S_i) \right)$ is larger than $\widehat{\partial} \mathbf{T}$. See Section \ref{causal_sec_counter} for a more developed discussion.

\begin{proof}[Proof of the first point]
Most of the proof works exactly along the same lines as the proof of Proposition \ref{slice_strong_transience}, with $\mathbf{T}$ playing the same role as the backbone tree. In particular, we choose for $(x_n)$ a nonbacktracking random walk on $\mathbf{T}$, and the sets $A_k$ are built in the same way as in the original proof, but from $\mathbf{T}$ instead of $\back(T)$. The proof of Proposition \ref{slice_strong_transience} from Lemma \ref{lem_separating_sets} is very similar, as well as points (i) and (ii) of Lemma \ref{lem_separating_sets}. The only difference is that the proof of point (ii) of Lemma \ref{lem_separating_sets} is easier in our new framework, because of the independence of the strips, and we do not need anymore to restrict ourselves to even values of $k$.

Exactly as in the first proof, by using the "self-similarity" property of $\sli \left( \mathbf{T}, (S_i)_{i \geq 0} \right)$, the proof of point (iii) of Lemma \ref{lem_separating_sets} can be reduced to the proof that the dual map of $\sli \left( \mathbf{T}, (S_i)_{i \geq 0} \right)$ is transient (it is also important that the sets $A_k$ do not touch each other, which is why we have required that the boundaries of the strips are simple). The adaptation of the proof of Lemma \ref{transience_dual_slice} (transience of the dual slice), however, is not obvious.

More precisely, let $\sli^*$ be the graph whose vertices are the finite faces of $\sli \left( \mathbf{T}, (S_i)_{i \geq 0} \right)$ and where, for every edge $e$ of $\sli \left( \mathbf{T}, (S_i)_{i \geq 0} \right)$ that is adjacent to two finite faces, we draw an edge $e^*$ between these two faces. We note right now that, since all the strips have only finite faces, the graph $\sli^*$ is a connected graph, and we need to prove that it is transient.
The idea of the proof is the following: we will build a genealogy on the set of strips, which contains a complete binary tree. As in the proof of Lemma \ref{transience_dual_slice}, we will then kill the strips whose root is "too far" from its children, in order to preserve some quasi-isometry.

We first build a genealogy on the set of strips. We recall that the \emph{root} of a strip $S_i$ is the lowest vertex of its boundary, and is denoted by $\rho_i$. The \emph{height} of $S_i$ is the height of $\rho_i$. We call two strips \emph{adjacent} if their respective boundaries share at least one edge. If $S_i$ is a strip, we consider the first vertex on the left boundary of $S_i$ (apart from $\rho_i$) that is also a branching point of $\mathbf{T}$. This vertex is also the root of some strips, exactly one of which is adjacent to $S_i$. We call this strip the \emph{left child} of $S_i$ (cf. Figure \ref{fig_genealogy_strips}). We can similarly define its \emph{right child}. Note that almost surely, every branch of $\mathbf{T}$ branches eventually, so these childs always exist. We now fix a strip $S_1$. We claim that the restriction of this genealogy to the set of descendants of $S_1$ is encoded by a complete binary tree, which we denote by $T_{\str}$. Indeed, all the descendants of the left child of a strip $S$ lie on the left of $S$, whereas all the descendants of its right child lie on its right. Therefore, it is not possible to obtain the same strip by two different genealogical lines from $S_1$ (see Figure \ref{fig_genealogy_strips}).

\begin{figure}
\begin{center}
\begin{tikzpicture}
\draw[red, very thick] (0,0)--(-3,1);
\draw[red, very thick] (0,0)--(3,1);
\draw[red, very thick] (-3,1)--(-3,2);
\draw[red, very thick] (3,1)--(1,2);
\draw[red, very thick] (3,1)--(3,2);
\draw[red, very thick] (3,1)--(5,2);
\draw[red, very thick] (-3,2)--(-4,3);
\draw[red, very thick] (-3,2)--(-2,3);
\draw[red, very thick] (1,2)--(0,3);
\draw[red, very thick] (1,2)--(1.5,3);
\draw[red, very thick] (3,2)--(2.5,3);
\draw[red, very thick] (3,2)--(3.5,3);
\draw[red, very thick] (5,2)--(5.5,3);
\draw[red, very thick] (-4,3)--(-5,4);
\draw[red, very thick] (-4,3)--(-3.5,4);
\draw[red, very thick] (-2,3)--(-2,4);
\draw[red, very thick] (0,3)--(-0.5,4);
\draw[red, very thick] (0,3)--(0.5,4);
\draw[red, very thick] (1.5,3)--(1.5,4);
\draw[red, very thick] (2.5,3)--(2,4);
\draw[red, very thick] (3.5,3)--(3,4);
\draw[red, very thick] (3.5,3)--(4.5,4);
\draw[red, very thick] (5.5,3)--(5,4);
\draw[red, very thick] (5.5,3)--(6,4);
\draw[red, very thick] (-5,4)--(-6,5);
\draw[red, very thick] (-5,4)--(-5,5);
\draw[red, very thick] (-3.5,4)--(-4,5);
\draw[red, very thick] (-2,4)--(-3,5);
\draw[red, very thick] (-2,4)--(-2,5);
\draw[red, very thick] (-0.5,4)--(-1.5,5);
\draw[red, very thick] (-0.5,4)--(-0.5,5);
\draw[red, very thick] (0.5,4)--(0.5,5);
\draw[red, very thick] (1.5,4)--(1.5,5);
\draw[red, very thick] (2,4)--(2,5);
\draw[red, very thick] (2,4)--(3,5);
\draw[red, very thick] (3,4)--(3.5,5);
\draw[red, very thick] (4.5,4)--(4.5,5);
\draw[red, very thick] (5,4)--(5,5);
\draw[red, very thick] (5,4)--(6,5);
\draw[red, very thick] (6,4)--(7,5);
\draw[red, very thick] (-6,5)--(-7,5.7);
\draw[red, very thick] (-5,5)--(-6,5.7);
\draw[red, very thick] (-5,5)--(-5,5.7);
\draw[red, very thick] (-4,5)--(-4.5,5.7);
\draw[red, very thick] (-4,5)--(-4,5.7);
\draw[red, very thick] (-3,5)--(-3.5,5.7);
\draw[red, very thick] (-2,5)--(-2.5,5.7);
\draw[red, very thick] (-1.5,5)--(-2,5.7);
\draw[red, very thick] (-0.5,5)--(-1,5.7);
\draw[red, very thick] (0.5,5)--(0,5.7);
\draw[red, very thick] (1.5,5)--(0.5,5.7);
\draw[red, very thick] (1.5,5)--(1.5,5.7);
\draw[red, very thick] (2,5)--(2,5.7);
\draw[red, very thick] (3,5)--(2.5,5.7);
\draw[red, very thick] (3,5)--(3,5.7);
\draw[red, very thick] (3.5,5)--(3.5,5.7);
\draw[red, very thick] (4.5,5)--(4,5.7);
\draw[red, very thick] (4.5,5)--(5,5.7);
\draw[red, very thick] (5,5)--(5.5,5.7);
\draw[red, very thick] (6,5)--(6.5,5.7);
\draw[red, very thick] (7,5)--(7,5.7);
\draw[red, very thick] (7,5)--(8,5.7);

\draw(0,0)node[petitrouge]{};
\draw(-3,1)node[petitrouge]{};
\draw(3,1)node[petitrouge]{};
\draw(-3,2)node[petitrouge]{};
\draw(1,2)node[petitrouge]{};
\draw(3,2)node[petitrouge]{};
\draw(5,2)node[petitrouge]{};
\draw(-4,3)node[petitrouge]{};
\draw(-2,3)node[petitrouge]{};
\draw(0,3)node[petitrouge]{};
\draw(1.5,3)node[petitrouge]{};
\draw(2.5,3)node[petitrouge]{};
\draw(3.5,3)node[petitrouge]{};
\draw(5.5,3)node[petitrouge]{};
\draw(-5,4)node[petitrouge]{};
\draw(-3.5,4)node[petitrouge]{};
\draw(-2,4)node[petitrouge]{};
\draw(-0.5,4)node[petitrouge]{};
\draw(0.5,4)node[petitrouge]{};
\draw(1.5,4)node[petitrouge]{};
\draw(2,4)node[petitrouge]{};
\draw(3,4)node[petitrouge]{};
\draw(4.5,4)node[petitrouge]{};
\draw(5,4)node[petitrouge]{};
\draw(6,4)node[petitrouge]{};
\draw(-6,5)node[petitrouge]{};
\draw(-5,5)node[petitrouge]{};
\draw(-4,5)node[petitrouge]{};
\draw(-3,5)node[petitrouge]{};
\draw(-2,5)node[petitrouge]{};
\draw(-1.5,5)node[petitrouge]{};
\draw(-0.5,5)node[petitrouge]{};
\draw(0.5,5)node[petitrouge]{};
\draw(1.5,5)node[petitrouge]{};
\draw(2,5)node[petitrouge]{};
\draw(3,5)node[petitrouge]{};
\draw(3.5,5)node[petitrouge]{};
\draw(4.5,5)node[petitrouge]{};
\draw(5,5)node[petitrouge]{};
\draw(6,5)node[petitrouge]{};
\draw(7,5)node[petitrouge]{};

\draw[dashed, very thick, blue](0,1)--(2,2);
\draw[dashed, very thick, blue](0,1)--(-3,3);
\draw[dashed, very thick, blue](-3,3)--(-4.2,4);
\draw[dashed, very thick, blue](-3,3)--(-2.5,5);
\draw[dashed, very thick, blue](-4.2,4)--(-5.5,5);
\draw[dashed, very thick, blue](2,2)--(0.8,3);
\draw[dashed, very thick, blue](2,2)--(3,3);
\draw[dashed, very thick, blue](0.8,3)--(0,4);
\draw[dashed, very thick, blue](0.8,3)--(1,5.7);
\draw[dashed, very thick, blue](0,4)--(-1,5);
\draw[dashed, very thick, blue](3,3)--(3.5,4);
\draw[dashed, very thick, blue](3,3)--(2.5,5);
\draw[dashed, very thick, blue](-5.5,5)--(-5.5,5.7);
\draw[dashed, very thick, blue](2.5,5)--(2.75,5.7);
\draw[dashed, very thick, blue](3.5,4)--(4.5,5.7);
\draw[dashed, very thick, blue](3.5,4)to[bend right=15](3.8,5.7);
\draw[dashed, very thick, blue](-4.2,4)to[bend left=15](-4.2,5.7);
\draw[dashed, very thick, blue](0,4)to[bend left=15](-0.2,5.7);
\draw[dashed, very thick, blue](-5.5,5)--(-6.5,5.7);
\draw[dashed, very thick, blue](-2.5,5)--(-3,5.7);
\draw[dashed, very thick, blue](-2.5,5)--(-2.7,5.7);
\draw[dashed, very thick, blue](-1,5)--(-1.2,5.7);
\draw[dashed, very thick, blue](-1,5)--(-1.6,5.7);
\draw[dashed, very thick, blue](2.5,5)--(2.2,5.7);

\draw(0,1)node[dual]{};
\draw(2,2)node[dual]{};
\draw(-3,3)node[dual]{};
\draw(-4.2,4)node[dual]{};
\draw(-2.5,5)node[dual]{};
\draw(-5.5,5)node[dual]{};
\draw(0.8,3)node[dual]{};
\draw(3,3)node[dual]{};
\draw(0,4)node[dual]{};
\draw(1,5.7)node[dual]{};
\draw(-1,5)node[dual]{};
\draw(3.5,4)node[dual]{};
\draw(2.5,5)node[dual]{};
\draw(-5.5,5.7)node[dual]{};
\draw(-4.2,5.7)node[dual]{};
\draw(2.75,5.7)node[dual]{};
\draw(4.5,5.7)node[dual]{};

\draw[red](0,-0.3)node[texte]{$\rho$};
\draw[blue](0,1.5)node[texte]{$S_1$};
\draw[blue](2.6,3.5)node[texte]{$S_2$};
\end{tikzpicture}
\end{center}
\caption{The tree $\mathbf{T}$ (in red), and the tree $T_{\str}$ of descendants of the strip $S_1$ (in blue). Note that the strip $S_2$ has one parent on each side, but the right parent is not a descendant of $S_1$.}\label{fig_genealogy_strips}
\end{figure}
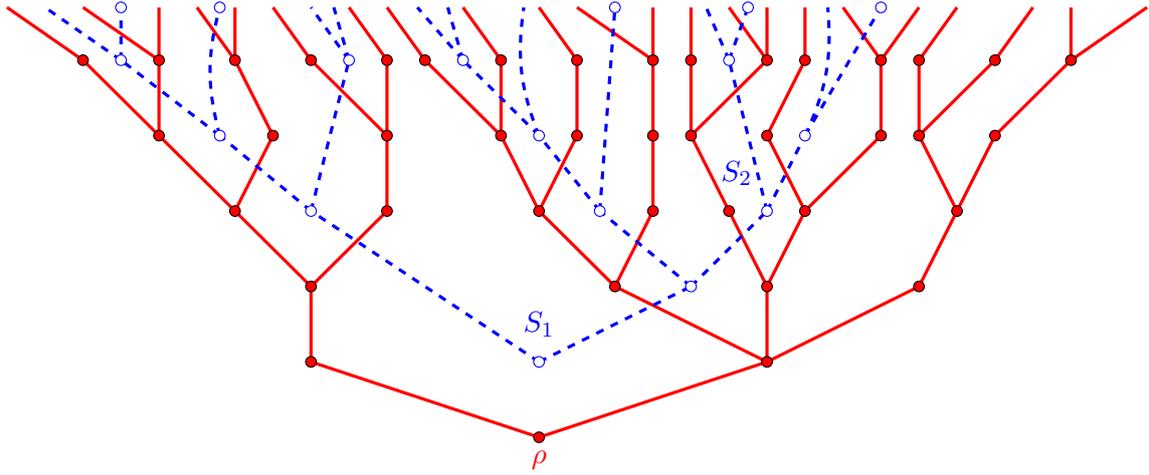

We now kill some of the strips. We fix a constant $\ell_{\max}>0$. For every strip $S_i$, let $e_i^{\ell}$ (resp. $e_i^r$) be the first edge on the left (resp. right) boundary of $S_i$ that is also adjacent to its left (resp. right) child. We call a strip $S_i$ \emph{good} if, for every face $f$ of $S_i$ that is adjacent to $\rho_i$, the dual of $S_i$ contains a path of length at most $\ell_{\max}$ from $f$ to $e_i^{\ell}$, and similarly for $e_i^r$.

Note that the fact that $S_i$ is good or not only depends on the internal geometry of $S_i$, and on the numbers of children of the vertices on the part of $\partial S_i$ lying between $e_i^{\ell}$ and $e_i^r$. These parts for different values of $i \in T_{\str}$ are disjoint. Hence, since the strips are i.i.d. and $\mathbf{T}$ is a Galton--Watson tree, the events
\[ \left\{ \mbox{$S_i$ is good} \right\}\]
for $i \in T_{\str}$ are independent, and have the same probability. Therefore, removing from $T_{\str}$ all the strips that are not good is equivalent to performing a Bernoulli site percolation on the complete binary tree $T_{\str}$. Moreover, the probability for a strip to be good goes to $1$ as $\ell_{\max}$ goes to $+\infty$, so we can find $\ell_{\max}$ such that this percolation is supercritical. We fix such an $\ell_{\max}$ until the end of the proof.

Let $T'_{\str}$ be an infinite connected component of $T_{\str}$ containing only good strips. Then $T'_{\str}$ is a supercritical Galton--Watson tree and survives, so it is transient. We can now define a submap $\sli_{\bdd}^*$ of $\sli^*$. For every strip $S_i \in T'_{\str}$, let $p_i^{\ell}$ (resp. $p_i^r$) be a dual path of length at most $\ell_{\max}$ joining the face of $S_i$ that is adjacent to its parent to the face adjacent to $e_i^{\ell}$ (resp. $e_i^r$). Then the edges of  $\sli_{\bdd}^*$ are the edges of these paths for all $i$ such that $S_i \in T'_{\str}$, as well as the dual edges of the edges $e_i^{\ell}$ and $e_i^r$ (cf. Figure \ref{fig_construction_sbdd_general}). The vertices of $\sli_{\bdd}^*$ are simply the vertices adjacent to these edges. Since the lengths of the paths $p_i^{\ell}$ and $p_i^r$ are bounded by $\ell_{\max}$, it is easy to see that $\sli_{\bdd}^*$ is quasi-isometric to the tree $T'_{\str}$, so it is transient. This implies that $\sli^*$ is transient as well, which concludes the proof of the first point.

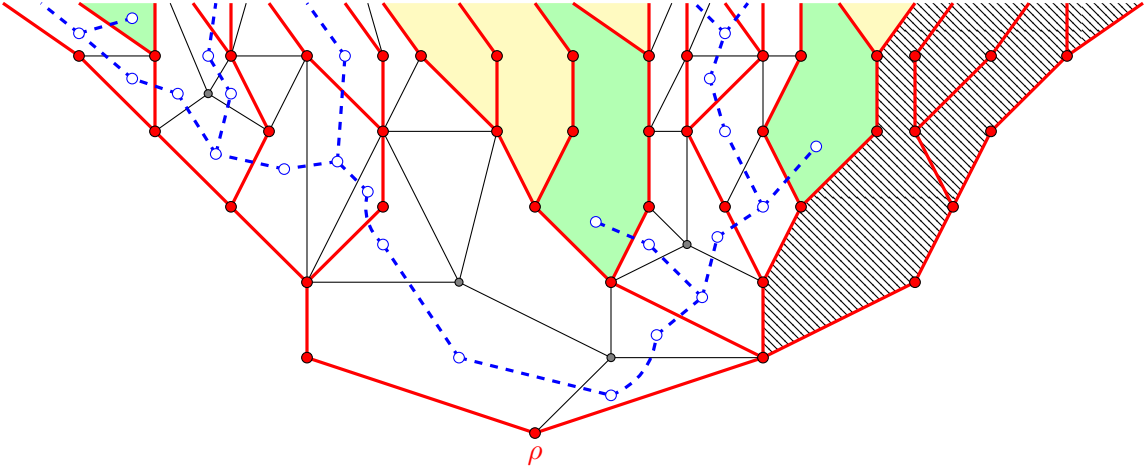
\begin{figure}
\begin{center}
\begin{tikzpicture}
\fill[yellow!30] (-2,5.7)--(-1.5,5)--(-0.5,4)--(0,3)--(1.5,3)--(1.5,5.7);
\fill[yellow!30] (4,5.7)--(4.5,5)--(5,5.7);
\fill[green!30] (0,5.7)--(0.5,5)--(0.5,4)--(0,3)--(1,2)--(1.5,3)--(1.5,4)--(1.5,5)--(0.5,5.7);
\fill[green!30] (3.5,5.7)--(3.5,5)--(3,4)--(3.5,3)--(4.5,4)--(4.5,5)--(4,5.7);
\fill[green!30] (-6,5.7)--(-5,5)--(-5,5.7);
\fill[pattern=north west lines] (5,5.7)--(4.5,5)--(4.5,4)--(3.5,3)--(3,2)--(3,1)--(5,2)--(6,4)--(7,5)--(8,5.7);

\draw[red, very thick] (0,0)--(-3,1);
\draw[red, very thick] (0,0)--(3,1);
\draw[red, very thick] (-3,1)--(-3,2);
\draw[red, very thick] (3,1)--(1,2);
\draw[red, very thick] (3,1)--(3,2);
\draw[red, very thick] (3,1)--(5,2);
\draw[red, very thick] (-3,2)--(-4,3);
\draw[red, very thick] (-3,2)--(-2,3);
\draw[red, very thick] (1,2)--(0,3);
\draw[red, very thick] (1,2)--(1.5,3);
\draw[red, very thick] (3,2)--(2.5,3);
\draw[red, very thick] (3,2)--(3.5,3);
\draw[red, very thick] (5,2)--(5.5,3);
\draw[red, very thick] (-4,3)--(-5,4);
\draw[red, very thick] (-4,3)--(-3.5,4);
\draw[red, very thick] (-2,3)--(-2,4);
\draw[red, very thick] (0,3)--(-0.5,4);
\draw[red, very thick] (0,3)--(0.5,4);
\draw[red, very thick] (1.5,3)--(1.5,4);
\draw[red, very thick] (2.5,3)--(2,4);
\draw[red, very thick] (3.5,3)--(3,4);
\draw[red, very thick] (3.5,3)--(4.5,4);
\draw[red, very thick] (5.5,3)--(5,4);
\draw[red, very thick] (5.5,3)--(6,4);
\draw[red, very thick] (-5,4)--(-6,5);
\draw[red, very thick] (-5,4)--(-5,5);
\draw[red, very thick] (-3.5,4)--(-4,5);
\draw[red, very thick] (-2,4)--(-3,5);
\draw[red, very thick] (-2,4)--(-2,5);
\draw[red, very thick] (-0.5,4)--(-1.5,5);
\draw[red, very thick] (-0.5,4)--(-0.5,5);
\draw[red, very thick] (0.5,4)--(0.5,5);
\draw[red, very thick] (1.5,4)--(1.5,5);
\draw[red, very thick] (2,4)--(2,5);
\draw[red, very thick] (2,4)--(3,5);
\draw[red, very thick] (3,4)--(3.5,5);
\draw[red, very thick] (4.5,4)--(4.5,5);
\draw[red, very thick] (5,4)--(5,5);
\draw[red, very thick] (5,4)--(6,5);
\draw[red, very thick] (6,4)--(7,5);
\draw[red, very thick] (-6,5)--(-7,5.7);
\draw[red, very thick] (-5,5)--(-6,5.7);
\draw[red, very thick] (-5,5)--(-5,5.7);
\draw[red, very thick] (-4,5)--(-4.5,5.7);
\draw[red, very thick] (-4,5)--(-4,5.7);
\draw[red, very thick] (-3,5)--(-3.5,5.7);
\draw[red, very thick] (-2,5)--(-2.5,5.7);
\draw[red, very thick] (-1.5,5)--(-2,5.7);
\draw[red, very thick] (-0.5,5)--(-1,5.7);
\draw[red, very thick] (0.5,5)--(0,5.7);
\draw[red, very thick] (1.5,5)--(0.5,5.7);
\draw[red, very thick] (1.5,5)--(1.5,5.7);
\draw[red, very thick] (2,5)--(2,5.7);
\draw[red, very thick] (3,5)--(2.5,5.7);
\draw[red, very thick] (3,5)--(3,5.7);
\draw[red, very thick] (3.5,5)--(3.5,5.7);
\draw[red, very thick] (4.5,5)--(4,5.7);
\draw[red, very thick] (4.5,5)--(5,5.7);
\draw[red, very thick] (5,5)--(5.5,5.7);
\draw[red, very thick] (6,5)--(6.5,5.7);
\draw[red, very thick] (7,5)--(7,5.7);
\draw[red, very thick] (7,5)--(8,5.7);

\draw(0,0)--(1,1);
\draw(1,1)--(3,1);
\draw(1,2)--(1,1);
\draw(1,1)--(-1,2);
\draw(-1,2)--(-3,2);
\draw(-1,2)--(-2,4);
\draw(-1,2)--(-0.5,4);
\draw(-2,4)--(-0.5,4);
\draw(-2,4)--(-1.5,5);
\draw(-2,5)--(-2.2,5.7);
\draw(-3,2)--(-2,4);
\draw(-3,2)--(-3,5);
\draw(-3,5)--(-3.5,4);
\draw(-3,5)--(-4,5);
\draw(-3.8,5.7)--(-4,5);
\draw(-4.3,4.5)--(-5,4);
\draw(-4.3,4.5)--(-3.5,4);
\draw(-4.3,4.5)--(-4,5);
\draw(-4.3,4.5)--(-4.8,5.7);
\draw(-5,5)--(-6,5);
\draw(2,2.5)--(1,2);
\draw(2,2.5)--(3,2);
\draw(2,2.5)--(1.5,3);
\draw(2,2.5)--(2,4);
\draw(2,4)--(1.5,4);
\draw(1.5,4)--(2,5);
\draw(1.5,5)--(1.8,5.7);
\draw(2.5,3)--(3,4);
\draw(3,4)--(3,5);
\draw(3,5)--(3.5,5);
\draw(2,5)--(3,5);
\draw(2,5)--(2.3,5.7);

\draw(0,0)node[petitrouge]{};
\draw(-3,1)node[petitrouge]{};
\draw(3,1)node[petitrouge]{};
\draw(-3,2)node[petitrouge]{};
\draw(1,2)node[petitrouge]{};
\draw(3,2)node[petitrouge]{};
\draw(5,2)node[petitrouge]{};
\draw(-4,3)node[petitrouge]{};
\draw(-2,3)node[petitrouge]{};
\draw(0,3)node[petitrouge]{};
\draw(1.5,3)node[petitrouge]{};
\draw(2.5,3)node[petitrouge]{};
\draw(3.5,3)node[petitrouge]{};
\draw(5.5,3)node[petitrouge]{};
\draw(-5,4)node[petitrouge]{};
\draw(-3.5,4)node[petitrouge]{};
\draw(-2,4)node[petitrouge]{};
\draw(-0.5,4)node[petitrouge]{};
\draw(0.5,4)node[petitrouge]{};
\draw(1.5,4)node[petitrouge]{};
\draw(2,4)node[petitrouge]{};
\draw(3,4)node[petitrouge]{};
\draw(4.5,4)node[petitrouge]{};
\draw(5,4)node[petitrouge]{};
\draw(6,4)node[petitrouge]{};
\draw(-6,5)node[petitrouge]{};
\draw(-5,5)node[petitrouge]{};
\draw(-4,5)node[petitrouge]{};
\draw(-3,5)node[petitrouge]{};
\draw(-2,5)node[petitrouge]{};
\draw(-1.5,5)node[petitrouge]{};
\draw(-0.5,5)node[petitrouge]{};
\draw(0.5,5)node[petitrouge]{};
\draw(1.5,5)node[petitrouge]{};
\draw(2,5)node[petitrouge]{};
\draw(3,5)node[petitrouge]{};
\draw(3.5,5)node[petitrouge]{};
\draw(4.5,5)node[petitrouge]{};
\draw(5,5)node[petitrouge]{};
\draw(6,5)node[petitrouge]{};
\draw(7,5)node[petitrouge]{};

\draw(1,1)node[small]{};
\draw(-1,2)node[small]{};
\draw(-4.3,4.5)node[small]{};
\draw(2,2.5)node[small]{};

\draw[dashed, very thick, blue](-1,1)--(-2,2.5);
\draw[dashed, very thick, blue](-1,1)--(1,0.5);
\draw[dashed, very thick, blue](1,0.5)to[bend right](1.6,1.3);
\draw[dashed, very thick, blue](-2,2.5)to[bend left](-2.2,3.2);
\draw[dashed, very thick, blue](-2.2,3.2)--(-2.6,3.6);
\draw[dashed, very thick, blue](-2.6,3.6)--(-3.3,3.5);
\draw[dashed, very thick, blue](-2.6,3.6)--(-2.5,5);
\draw[dashed, very thick, blue](-3.3,3.5)--(-4.2,3.7);
\draw[dashed, very thick, blue](-4.2,3.7)--(-4.7,4.5);
\draw[dashed, very thick, blue](-4.2,3.7)--(-4,4.5);
\draw[dashed, very thick, blue](-4,4.5)--(-4.3,5);
\draw[dashed, very thick, blue](-4.7,4.5)--(-5.3,4.7);
\draw[dashed, very thick, blue](-5.3,4.7)--(-6,5.3);
\draw[dashed, very thick, blue](-6,5.3)--(-6.5,5.7);
\draw[dashed, very thick, blue](-6,5.3)--(-5.3,5.5);
\draw[dashed, very thick, blue](-4.3,5)--(-4.2,5.7);
\draw[dashed, very thick, blue](-2.5,5)--(-3,5.7);
\draw[dashed, very thick, blue](1.6,1.3)--(2.2,1.8);
\draw[dashed, very thick, blue](2.2,1.8)--(1.5,2.5);
\draw[dashed, very thick, blue](2.2,1.8)--(2.4,2.6);
\draw[dashed, very thick, blue](1.5,2.5)--(0.8,2.8);
\draw[dashed, very thick, blue](2.4,2.6)--(3,3);
\draw[dashed, very thick, blue](3,3)--(2.5,4);
\draw[dashed, very thick, blue](3,3)--(3.7,3.8);
\draw[dashed, very thick, blue](2.5,4)--(2.3,4.7);
\draw[dashed, very thick, blue](2.3,4.7)--(2.5,5.3);
\draw[dashed, very thick, blue](2.5,5.3)--(2.8,5.7);
\draw[dashed, very thick, blue](2.5,5.3)--(2.1,5.7);

\draw(-1,1)node[dual]{};
\draw(-2,2.5)node[dual]{};
\draw(1,0.5)node[dual]{};
\draw(1.6,1.3)node[dual]{};
\draw(-2.2,3.2)node[dual]{};
\draw(-2.6,3.6)node[dual]{};
\draw(-3.3,3.5)node[dual]{};
\draw(-2.5,5)node[dual]{};
\draw(-4.2,3.7)node[dual]{};
\draw(-4.7,4.5)node[dual]{};
\draw(-4,4.5)node[dual]{};
\draw(-4.3,5)node[dual]{};
\draw(-5.3,4.7)node[dual]{};
\draw(-6,5.3)node[dual]{};
\draw(-5.3,5.5)node[dual]{};
\draw(2.2,1.8)node[dual]{};
\draw(1.5,2.5)node[dual]{};
\draw(2.4,2.6)node[dual]{};
\draw(0.8,2.8)node[dual]{};
\draw(3,3)node[dual]{};
\draw(2.5,4)node[dual]{};
\draw(3.7,3.8)node[dual]{};
\draw(2.3,4.7)node[dual]{};
\draw(2.5,5.3)node[dual]{};

\draw[red](0,-0.3)node[texte]{$\rho$};
\end{tikzpicture}
\end{center}
\caption{Construction of $\sli_{\bdd}^*$ (in blue). The tree $\mathbf{T}$ is in red. The hatched strips do not belong to $T_{\str}$. The green strips are the strips that are not good, and the yellow ones are the descendants of the green ones, so they are not in $T'_{\str}$. For the sake of clarity, we have not drawn the interiors of the strips that do not belong to $T'_{\str}$.}\label{fig_construction_sbdd_general}
\end{figure}
\end{proof}

\begin{proof}[Proof of point 2 of Theorem \ref{thm_2_bis}]
We mimic the beginning of the proof of Theorem \ref{thm_2_Poisson}, but we use the first point of Theorem \ref{thm_2_bis} instead of Proposition \ref{slice_strong_transience}. The proof of the first two points is robust, and shows that almost surely, the simple random walk $X$ on $\m \left( \mathbf{T}, (S_i)_{i \geq 0}\right)$ converges a.s. to a point $X_{\infty}$ of $\partial \mathbf{T}$, and that the distribution of $X_{\infty}$ has a.s.~full support. In particular, for every $x \in \mathbf{T}$, let $\widehat{\partial} \mathbf{T}[y]$ be the set of the classes $\widehat{\gamma}$, where $\gamma$ is a ray passing through $y$. Let $y_1, y_2$ be two vertices such that $\widehat{\partial} \mathbf{T}[y_1] \cap \widehat{\partial} \mathbf{T}[y_2]=\emptyset$. We define the function $h$ on $\m$ by
\[h(x)=P_{\m,x} \left( X_{\infty} \in \partial \mathbf{T}[y_1] \right). \]
Then $h$ is harmonic and bounded on $\m$. Moreover, by the same argument as in the beginning of the proof of Theorem \ref{thm_2_Poisson}, there is a sequence $(x_n)$ of vertices in $\sli[y_1]$ such that $h(x_n) \to 1$. On the other hand, by the first point of Theorem \ref{thm_2_bis}, there is a positive probability that $X$ stays in $\sli[y_2]$ eventually, so $h(x)<1$ for every $x$, so $h$ is non-constant. It follows that $\m$ is non-Liouville.
\end{proof}

\begin{proof}[Proof of point 3 of Theorem \ref{thm_2_bis}]
We know from the proof of point 2 of Theorem \ref{thm_2_bis} that the simple random walk $X$ on $\m \left( \mathbf{T}, (S_i)\right)$ converges a.s.~to a point $X_{\infty}$ on $\widehat{\partial} \mathbf{T}$. As in the proof of Theorem \ref{thm_2_Poisson}, it is enough to prove that the law of $X_{\infty}$ is a.s.~non-atomic, i.e. that if $X$ and $Y$ are two independent simple random walks on the same instance of $\m \left( \mathbf{T}, (S_i)\right)$, then $X_{\infty} \ne Y_{\infty}$ a.s..

To prove this, we rely on a variant of Lemma \ref{lem_separation_XY}. For $h \geq 0$, we denote by $\tau_h^X$ the smallest $n$ for which $X_n$ is a point of $\mathbf{T}$ of height at least $h$. We define $\tau_h^Y$ similarly. The recurrence of the strips guarantees that the times $\tau_h^X$ and $\tau_h^Y$ are all a.s.~finite. However, the vertices $X_{\tau_h^X}$ and $Y_{\tau_h^Y}$ may have height larger than $h$, which prevents us to re-use Lemma \ref{lem_separation_XY} as such. For $h \geq 1$, let $\f$ be the $\sigma$-algebra generated by the finite tree $B_h(\mathbf{T})$, the family of the strips whose root is at height at most $h-1$ and the paths $(X_n)_{0 \leq n \leq \tau_h^X}$ and $(Y_n)_{0 \leq n \leq \tau_h^Y}$. Note that this makes sense since up to time $\tau_h^X$, the walk $X$ may only visit strips whose root as height at most $h-1$.

Let also $A'_h$ be the following event:

$\Big\{$There are four distinct vertices $(x_i)_{1 \leq i \leq 4}$ of $T$ at height $\geq h$, neither of which is an ancestor of another, such that:
\begin{itemize}
\item
the trees $\mathbf{T}[x_1]$, $\mathbf{T}[x_2]$, $\mathbf{T}[x_3]$ and $\mathbf{T}[x_4]$ lie in this cyclic order,
\item
for every $n \geq \tau_h^X+2$, we have $X_n \in \sli \left( \mathbf{T}, (S_i) \right)[x_1]$ 
\item
for every $n \geq \tau_h^Y+2$, we have $Y_n \in \sli \left( \mathbf{T}, (S_i) \right)[x_3]. \Big\}$.
\end{itemize}
As in the proof of Theorem \ref{thm_2_Poisson}, it is enough to prove that there is $\delta>0$ such that almost surely, for $h$ large enough, we have
\begin{equation}\label{eqn_lem_XY_bis}
\P \left( A'_h | \f_h \right) \geq \delta.
\end{equation}
As in the proof of Lemma \ref{lem_separation_XY}, there are several cases to treat separately according to the relative positions of $X_{\tau_h^X}$ and $Y_{\tau_h^Y}$. Since this is the most different case from the proof of Lemma \ref{lem_separation_XY}, let us treat in details the case where $Y_{\tau_h^Y}$ is a strict descendant of $X_{\tau_h^X}$.

If this is the case, then right before $\tau_h^Y$ the walk $Y$ lies in a strip $S^Y$ intersecting $B_{h-1}(\mathbf{T})$. Without loss of generality, we assume that $S^Y$ is on the right of $X_{\tau_h^X}$. As on Figure \ref{fig_lem_XY_bis}, the point $Y_{\tau_h^Y}$ must be the rightmost descendant of $X_{\tau_h^X}$ at its height. Note also that the conditioning on $\f_h$ does not give any information about the numbers of children of the vertices of $\mathbf{T}$ between $X_{\tau_h^X}$ and $Y_{\tau_h^Y}$.

\begin{figure}
\begin{center}
\begin{tikzpicture}
\fill[yellow!40] (2,5)to[bend right=15](2,2)to[bend right=15](4,5);
\fill[yellow!40] (-3.5,6)to[bend right=15](-2,3)to[bend right=15](-2.5,6);
\fill[yellow!40] (-1.5,7)to[bend right=15](-1,5)to[bend right=15](-0.5,7);
\fill[yellow!40] (0.5,7)to[bend right=15](1,5)to[bend right=15](1.5,7);

\draw[red](0,0)--(-2,1);
\draw[red](0,0)--(2,1);
\draw[red](-2,1)--(-2,2);
\draw[red](2,1)--(2,2);
\draw[red](2,2)to[bend left=15](2,5);
\draw[red](2,2)to[bend right=15](4,5);
\draw[red](-2,2)--(-4,3);
\draw[red](-2,2)--(-2,3);
\draw[red](-2,2)--(0,3);
\draw[red](-4,3)--(-4.5,3.5);
\draw[red](-2,3)to[bend right=15](-2.5,6);
\draw[red](-2,3)to[bend left=15](-3.5,6);
\draw[red](0,3)--(0,4);
\draw[red](0,4)--(-1,5);
\draw[red](0,4)--(1,5);
\draw[red](-1,5)to[bend left=15](-1.5,7);
\draw[red](-1,5)to[bend right=15](-0.5,7);
\draw[red](1,5)to[bend left=15](0.5,7);
\draw[red](1,5)to[bend right=15](1.5,7);

\draw[dashed](-4,2)--(4,2);

\draw[blue, thick](0,0)to[out=180,in=270](-3,1)to[out=90, in=225](-2,2);
\draw[blue, thick](-2,2) to[bend left] (-2,3);
\draw[blue, thick](-2,3) to[out=105,in=270](-2.5,4)to[out=90,in=90](-2,4)to[out=270, in=270] (-3,6);

\draw[olive, thick](0,0)to[bend left=15](2,1);
\draw[olive, thick](2,1)to[out=105,in=90] (0.5,2.5) to[out=270,in=180](1,2.2)to[out=0,in=0](0,4);
\draw[olive, thick](0,4)to[bend left=15](1,5);
\draw[olive, thick](1,5)to[out=90,in=270](1.2,7);

\draw(0,0)node[petitrouge]{};
\draw(-2,1)node[petitrouge]{};
\draw(2,1)node[petitrouge]{};
\draw(-2,2)node[petitrouge]{};
\draw(2,2)node[petitrouge]{};
\draw(-4,3)node[petitrouge]{};
\draw(-2,3)node[petitrouge]{};
\draw(0,3)node[petitrouge]{};
\draw(0,4)node[petitrouge]{};
\draw(-1,5)node[petitrouge]{};
\draw(1,5)node[petitrouge]{};

\draw[red](0,-0.3)node[texte]{$\rho$};
\draw(-4.3,2)node[texte]{$h$};
\draw[orange](-2.8,6.3)node[texte]{$\sli[x_1]$};
\draw[orange](-1,7.3)node[texte]{$\sli[x_2]$};
\draw[orange](1,7.3)node[texte]{$\sli[x_3]$};
\draw[orange](3,5.3)node[texte]{$\sli[x_4]$};
\draw[blue](-3.3,1)node[texte]{$X$};
\draw[olive](1.5,1.5)node[texte]{$Y$};
\draw(-1.6,3)node[texte]{$x_1$};
\draw(-1,4.7)node[texte]{$x_2$};
\draw(1.1,4.7)node[texte]{$x_3$};
\draw(2.3,1.7)node[texte]{$x_4$};
\draw(-1.5,1.6)node[texte]{$X_{\tau_h^X}$};
\draw(-0.4,3.9)node[texte]{$Y_{\tau_h^Y}$};

\end{tikzpicture}
\end{center}
\caption{Sketch of the proof of \eqref{eqn_lem_XY_bis} in the case where $Y_{\tau_h^Y}$ is a descendant of $X_{\tau_h^X}$. The tree $\mathbf{T}$ is in red. The trajectories of $X$ and $Y$ are in blue and green respectively. For the sake of clarity, the interiors of the strips have not been drawn.}\label{fig_lem_XY_bis}
\end{figure}
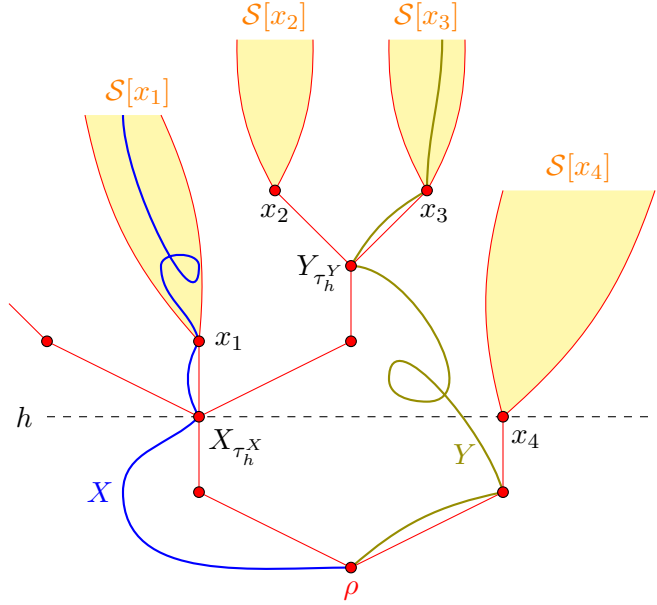

Let $K$ be a bound on the degrees in the strips. With conditional probability at least $(1-\mu(1)) \times \frac{1}{3K}$, the vertex $X_{\tau_h^X}$ is a branching point in $\mathbf{T}$ and $X_{\tau_h^X+1}$ is a child of $X_{\tau_h^X}$ which is not the rightmost one. If this occurs, let $x_1=X_{\tau_h^X+1}$. Then the slice $\sli \left( \mathbf{T}, (S_i) \right)[x_1]$ has the same law as $\sli \left( \mathbf{T}, (S_i) \right)$, so by the first point of Theorem \ref{thm_2_bis}, the walk $X$ has a probability bounded away from $0$ to stay in $\sli \left( \mathbf{T}, (S_i) \right) [x_1]$ aver after $\tau_h^X+1$. Similarly, independently of the behaviour of $X$, there is a probability at least $\frac{1-\mu(1)}{3K}$ that $x_3=Y_{\tau_h^Y+1}$ is a child of $Y_{\tau_h^Y}$ which is not the leftmost one, and that $Y$ stays in $\sli \left( \mathbf{T}, (S_i) \right) [x_3]$ ever after $\tau_h^Y$. To see that $A'_h$ occurs if this is the case, we just need to take as $x_2$ the leftmost child of $Y_{\tau_h^Y}$, and as $x_4$ a vertex of $\mathbf{T}$ which is different from $X_{\tau_h^X}$ (which a.s.~exists if $h$ is large enough).

This proves \eqref{eqn_lem_XY_bis} in the case where $Y_{\tau_h^Y}$ is a strict descendant of $X_{\tau_h^X}$. As in the proof of Theorem \ref{thm_2_Poisson}, the other cases can be treated in a very similar way. This finishes the proof of Theorem \ref{thm_2_bis}.
\end{proof}

\section{Positive speed of the simple random walk}\label{causal_sec_speed}

\subsection{Sketch of the proof and definition of the half-plane model \texorpdfstring{$\hcau$}{hcau}}\label{causal_subsec_halfplane}

The goal of this section is to prove Theorem \ref{thm_3_positive}. In all this section, we assume $\mu(0)=0$. We first give a quick sketch of the proof. Our main task will be to prove positive speed in a half-planar model $\hcau$, constructed from an infinite forest of supercritical Galton--Watson trees, and where we have more vertical stationarity than in $\cau$. The results of Section \ref{causal_sec_poisson} will allow us to pass from $\hcau$ to $\cau$. Note that Theorem \ref{thm_3_positive} is very easy in the case where $\mu(0)=\mu(1)=0$, since then the height of the simple random walk dominates a random walk on $\Z$ with positive drift. Therefore, we need to make sure that the vertices with only one child do not slow down the walk too much. Our proof in $\hcau$ relies on two ingredients:
\begin{itemize}
\item
An exploration method of $\hcau$ will allow us to prove that the walk cannot be too far away from a point with at least two children. This will guarantee that the walk spends some time at vertices with at least two children. At these vertices, the height of the walk accumulates a positive drift. This will give a "quasi-positive speed" result: the height at time $n$ is $n^{1-o(1)}$.
\item
Thanks to the stationarity properties of $\hcau$, we can study regeneration times, i.e. times at which the walk reaches some height for the first time, and stays above this height ever after. The estimates obtained in the first point are sufficient to prove that the first regeneration time has finite expectation. As in many other models (like random walks in random environments, see \cite{SZ99}), this is enough to ensure positive speed.
\end{itemize}

We now define our half-plane model. Let $(T_i)_{i \in \Z}$ be a family of i.i.d. Galton--Watson trees with offspring distribution $\mu$. We draw the trees $T_i$ with their roots on a horizontal line, and for every $h \geq 0$, we add horizontal connections between successive vertices of height $h$. Finally, for every vertex $v$ of height $0$, we add a parent of $v$ at height $-1$, which is linked only to $v$. We root the obtained map at the root of $T_0$ (which has height $0$), and denote it by $\hcau$ (see Figure \ref{figure_halfplane}). As for $\cau$, if $v$ is a vertex of $\hcau$, the height $h(v)$ of $v$ in $\hcau$ is defined as its height in the forest $(T_i)$. We denote by $\partial \hcau$ the set of vertices of $\hcau$ at height $-1$.

\begin{figure}
\begin{center}
\begin{tikzpicture}[scale=0.87]
\draw (-3,0)--(-3.5,1);
\draw (-3,0)--(-2.5,1);
\draw (-3.5,1)--(-3.5,2);
\draw (-2.5,1)--(-2.5,2);
\draw (-3.5,2)--(-3.5,3);
\draw (-2.5,2)--(-2.5,3);
\draw (-1,0)--(-1,1);
\draw (-1,1)--(-1.5,2);
\draw (-1,1)--(-0.5,2);
\draw (-1.5,2)--(-2,3);
\draw (-1.5,2)--(-1,3);
\draw (-0.5,2)--(-0.5,3);
\draw (0,0)--(0,1);
\draw (0,1)--(0,2);
\draw (0,2)--(0,3);
\draw (1.5,0)--(1.5,1);
\draw (1.5,1)--(1,2);
\draw (1.5,1)--(2,2);
\draw (1,2)--(0.5,3);
\draw (1,2)--(1.5,3);
\draw (2,2)--(2,3);
\draw (3,0)--(3,1);
\draw (3,1)--(3,2);
\draw (3,2)--(2.5,3);
\draw (3,2)--(3.5,3);

\draw (-3.5,3)--(-3.7,3.5);
\draw (-3.5,3)--(-3.3,3.5);
\draw (-2.5,3)--(-2.5,3.5);
\draw (-2,3)--(-2,3.5);
\draw (-1,3)--(-1,3.5);
\draw (-0.5,3)--(-0.5,3.5);
\draw (0,3)--(-0.2,3.5);
\draw (0,3)--(0.2,3.5);
\draw (0.5,3)--(0.5,3.5);
\draw (1.5,3)--(1.3,3.5);
\draw (1.5,3)--(1.7,3.5);
\draw (2,3)--(2,3.5);
\draw (2.5,3)--(2.5,3.5);
\draw (3.5,3)--(3.3,3.5);
\draw (3.5,3)--(3.7,3.5);

\draw (-3,0) node{};
\draw (-3.5,1) node{};
\draw (-2.5,1) node{};
\draw (-3.5,2) node{};
\draw (-2.5,2) node{};
\draw (-3.5,3) node{};
\draw (-2.5,3) node{};
\draw (-1,0) node{};
\draw (-1,1) node{};
\draw (-1.5,2) node{};
\draw (-0.5,2) node{};
\draw (-2,3) node{};
\draw (-1,3) node{};
\draw (-0.5,3) node{};
\draw (0,0) node[petitrouge]{};
\draw (0,1) node{};
\draw (0,2) node{};
\draw (0,3) node{};
\draw (1.5,0) node{};
\draw (1.5,1) node{};
\draw (1,2) node{};
\draw (2,2) node{};
\draw (0.5,3) node{};
\draw (1.5,3) node{};
\draw (2,3) node{};
\draw (3,0) node{};
\draw (3,1) node{};
\draw (3,2) node{};
\draw (2.5,3) node{};
\draw (3.5,3) node{};

\draw (-3,-1) node[texte]{$T_{-2}$};
\draw (-1,-1) node[texte]{$T_{-1}$};
\draw (0,-1) node[texte]{$T_0$};
\draw (1.5,-1) node[texte]{$T_1$};
\draw (3,-1) node[texte]{$T_2$};

\draw (-4,1) node[texte]{$\dots$};
\draw (4,1) node[texte]{$\dots$};

\begin{scope}[shift={(10,0)}]
\draw (-3,0)--(-3.5,1);
\draw (-3,0)--(-2.5,1);
\draw (-3.5,1)--(-3.5,2);
\draw (-2.5,1)--(-2.5,2);
\draw (-3.5,2)--(-3.5,3);
\draw (-2.5,2)--(-2.5,3);
\draw (-1,0)--(-1,1);
\draw (-1,1)--(-1.5,2);
\draw (-1,1)--(-0.5,2);
\draw (-1.5,2)--(-2,3);
\draw (-1.5,2)--(-1,3);
\draw (-0.5,2)--(-0.5,3);
\draw (0,0)--(0,1);
\draw (0,1)--(0,2);
\draw (0,2)--(0,3);
\draw (1.5,0)--(1.5,1);
\draw (1.5,1)--(1,2);
\draw (1.5,1)--(2,2);
\draw (1,2)--(0.5,3);
\draw (1,2)--(1.5,3);
\draw (2,2)--(2,3);
\draw (3,0)--(3,1);
\draw (3,1)--(3,2);
\draw (3,2)--(2.5,3);
\draw (3,2)--(3.5,3);

\draw (-3.5,3)--(-3.7,3.5);
\draw (-3.5,3)--(-3.3,3.5);
\draw (-2.5,3)--(-2.5,3.5);
\draw (-2,3)--(-2,3.5);
\draw (-1,3)--(-1,3.5);
\draw (-0.5,3)--(-0.5,3.5);
\draw (0,3)--(-0.2,3.5);
\draw (0,3)--(0.2,3.5);
\draw (0.5,3)--(0.5,3.5);
\draw (1.5,3)--(1.3,3.5);
\draw (1.5,3)--(1.7,3.5);
\draw (2,3)--(2,3.5);
\draw (2.5,3)--(2.5,3.5);
\draw (3.5,3)--(3.3,3.5);
\draw (3.5,3)--(3.7,3.5);

\draw (-4,0)--(4,0);
\draw (-4,1)--(4,1);
\draw (-4,2)--(4,2);
\draw (-4,3)--(4,3);

\draw (-3,-1)--(-3,0);
\draw (-1,-1)--(-1,0);
\draw (0,-1)--(0,0);
\draw (1.5,-1)--(1.5,0);
\draw (3,-1)--(3,0);

\draw (-3,0) node{};
\draw (-3.5,1) node{};
\draw (-2.5,1) node{};
\draw (-3.5,2) node{};
\draw (-2.5,2) node{};
\draw (-3.5,3) node{};
\draw (-2.5,3) node{};
\draw (-1,0) node{};
\draw (-1,1) node{};
\draw (-1.5,2) node{};
\draw (-0.5,2) node{};
\draw (-2,3) node{};
\draw (-1,3) node{};
\draw (-0.5,3) node{};
\draw (0,0) node[petitrouge]{};
\draw (0,1) node{};
\draw (0,2) node{};
\draw (0,3) node{};
\draw (1.5,0) node{};
\draw (1.5,1) node{};
\draw (1,2) node{};
\draw (2,2) node{};
\draw (0.5,3) node{};
\draw (1.5,3) node{};
\draw (2,3) node{};
\draw (3,0) node{};
\draw (3,1) node{};
\draw (3,2) node{};
\draw (2.5,3) node{};
\draw (3.5,3) node{};

\draw (-3,-1) node{};
\draw (-1,-1) node{};
\draw (0,-1) node{};
\draw (1.5,-1) node{};
\draw (3,-1) node{};

\draw[red] (0.3,-0.3) node[texte]{$\rho$};
\end{scope}
\end{tikzpicture}
\end{center}
\caption{On the left, an infinite forest $(T_i)_{i \in \Z}$. On the right, the half-planar map $\hcau$ obtained from this forest. The root vertex is in red.}\label{figure_halfplane}
\end{figure}
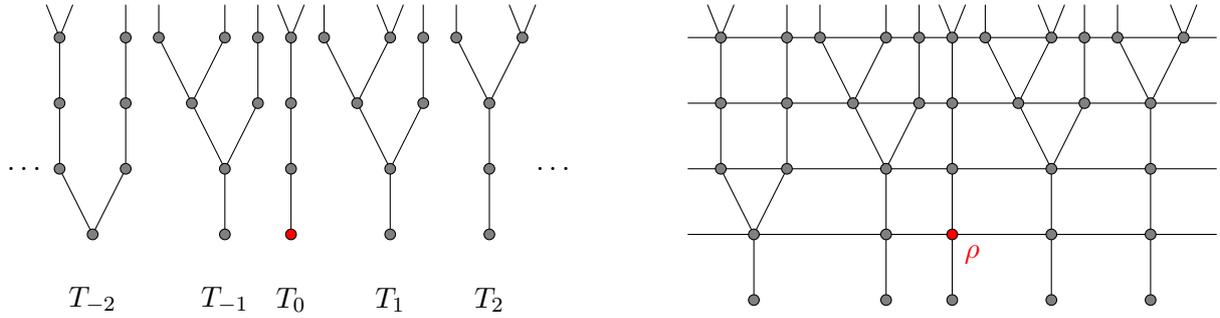

We first explain why it is enough to prove the result in $\hcau$ instead of $\cau$. We denote by $X^G$ the simple random walk on a graph $G$.

\begin{lem}\label{hcau_to_cau}
If in $\hcau$ we have $\frac{1}{n}h(X_n^{\hcau}) \to v_{\mu}>0$ a.s.~as $n \to +\infty$, then Theorem \ref{thm_3_positive} is true.
\end{lem}

\begin{proof}
For every $i \in \Z$, let $\sli(T_i)$ be the slice associated to the tree $T_i$. Then almost surely, if the walk $(X_n)$ stays in $\sli(T_i)$ eventually, the distance between $X_n$ and the root of $T_i$ is equivalent to $v_{\mu}n$. Since $\sli(T_i)$ has the same distribution as $\sli$, the simple random walk on $\sli$ has a.s.~speed $v_{\mu}$ on the event that it does not hit $\partial \sli$.

Back in $\cau$, by Theorem \ref{thm_2_Poisson}, we know that $X^{\cau}$ converges to a point $X_{\infty} \in \widehat{\partial} T$ and that $X_{\infty}$ is a.s.~not the point of $\widehat{\partial} T$ corresponding to the leftmost and rightmost rays of $T$. Therefore, almost surely, the walk $X^{\cau}$ only hits the boundary of $\sli(T)$ finitely many times, so it has a.s.~speed $v_{\mu}>0$.
\end{proof}

\subsection{An exploration method of \texorpdfstring{$\hcau$}{hcau}}\label{causal_subsec_exploration}

Let $k>0$ and let $x$ be a vertex of $\hcau$ at height $h \geq 0$. We write $x_0=x$. Let also $x_{-1}, \dots, x_{-k}$ be the $k$ first neighbours at height $h$ on the left of $x$, and let $x_{1}, \dots, x_{k}$ be the $k$ first neighbours of $x$ on its right.
\begin{defn}\label{def_bad}
Let $x$ be a vertex of $\hcau$ at height $h \geq 0$. We say that $x$ is \emph{$k$-bad} if all the descendants of $x_{-k}, x_{-(k-1)}, \dots, x_k$ at heights $h, h+1, \dots, h+k$ have only one child (cf. Figure \ref{figure_bad_point}).
\end{defn}

\begin{figure}
\begin{center}
\begin{tikzpicture}

\draw(-2.5,-1)--(2.5,-1);
\draw(-2,-1)--(-2.5,-0.5);
\draw(-2,-1)--(-2,0);
\draw(-2,-1)--(-1,0);
\draw(0.5,-1)--(0,0);
\draw(0.5,-1)--(1,0);
\draw(2,-1)--(2,0);
\draw(-2.5,0)--(2.5,0);
\draw(-2.5,1)--(2.5,1);
\draw(-2.5,2)--(2.5,2);
\draw(-2.5,3)--(2.5,3);
\draw(-2,0)--(-2,3);
\draw(-1,0)--(-1,3);
\draw(0,0)--(0,3);
\draw(1,0)--(1,3);
\draw(2,0)--(2,3);
\draw(-2,3)--(-2.3,3.5);
\draw(-2,3)--(-1.7,3.5);
\draw(-1,3)--(-1,3.5);
\draw(0,3)--(0,3.5);
\draw(1,3)--(0.7,3.5);
\draw(1,3)--(1.3,3.5);
\draw(2,3)--(1.7,3.5);
\draw(2,3)--(2.3,3.5);

\draw(-2,-1) node{};
\draw(0.5,-1) node{};
\draw(2,-1) node{};
\draw(-2,0) node{};
\draw(-1,0) node{};
\draw(-0,0) node{};
\draw(1,0) node{};
\draw(2,0) node{};
\draw(-2,1) node{};
\draw(-1,1) node{};
\draw(-0,1) node{};
\draw(1,1) node{};
\draw(2,1) node{};
\draw(-2,2) node{};
\draw(-1,2) node{};
\draw(-0,2) node{};
\draw(1,2) node{};
\draw(2,2) node{};
\draw(-2,3) node{};
\draw(-1,3) node{};
\draw(-0,3) node{};
\draw(1,3) node{};
\draw(2,3) node{};

\draw(-2.3,-0.3) node[texte]{$x_{-2}$};
\draw(-0.9,-0.3) node[texte]{$x_{-1}$};
\draw(-0.1,-0.3) node[texte]{$x_0$};
\draw(1.1,-0.3) node[texte]{$x_1$};
\draw(2.3,-0.3) node[texte]{$x_2$};
\end{tikzpicture}
\end{center}
\caption{An example for Definition \ref{def_bad}: the point $x=x_0$ is $2$-bad, but not $3$-bad.}\label{figure_bad_point}
\end{figure}
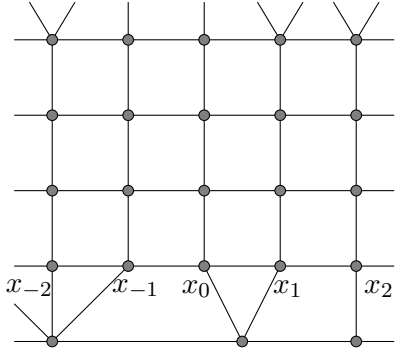

The goal of this section is to show that the probability for the walk to visit a "very bad" point in its first $n$ steps is very small. More precisely, we will prove the following result.

\begin{lem}\label{lem_bad_points}
There is a constant $c>0$ such that for every $k,n>0$, we have
\[\P \left( \mbox{one of the points $X_0, X_1, \dots, X_n$ is $k$-bad} \right) \leq ck(n+1)^2 \mu(1)^{k^2}.\]
\end{lem}

The presence of the factor $\mu(1)^{k^2}$ is not surprising. For example, the probability for the root vertex $X_0$ to be $k$-bad is exactly $\mu(1)^{(k+1)(2k+1)}$. The idea of the proof is to explore $\hcau$ at the same time as the random walk moves, in such a way that every time we discover a new vertex, either its $k$ neighbours on the right or its $k$ neighbours on the left have all their descendants undiscovered. If this is the case, the probability for the discovered vertex to be $k$-bad is at most $\mu(1)^{k^2}$. The factor $ck(n+1)^2$ means that our exploration method needs to explore at most $ck(n+1)^2$ vertices to discover $\{X_0, X_1, \dots, X_n\}$.

The rest of Section \ref{causal_subsec_exploration} is devoted to the proof of Lemma \ref{lem_bad_points}. We first define our exploration method. We then state precisely the properties that we need this exploration to satisfy (Lemmas \ref{lem_expl_markov}, \ref{lem_expl_phi} and \ref{lem_expl_k_free}), and explain how to conclude the proof from here. Finally, we prove these properties.

\paragraph{Exploration methods.}
By an \emph{exploration method}, we mean a nondecreasing sequence $(E_i)_{i \geq 0}$ of finite sets of vertices of $\hcau$. For every $i \geq 0$, the part of $\hcau$ discovered at time $i$ is the finite map formed by:
\begin{itemize}
\item[$\bullet$]
the vertices of $E_i$,
\item[$\bullet$]
the edges of $\hcau$ whose two endpoints belong to $E_i$, 
\item[$\bullet$]
for every edge of $\hcau$ with exactly one endpoint in $E_i$, the half-edge adjacent to $E_i$.
\end{itemize}
We denote this map by $\expl_i$ (see Figure \ref{fig_exploration_method}). In particular, when we explore a vertex, we know how many children it has, even if these children are yet undiscovered. Moreover, for every $i$, the map $\expl_i$ will be equipped with an additional marked oriented edge or half-edge $e_i$ describing the current position of the simple random walk. In all the explorations we will consider, we can pass from $\expl_{i-1}$ to $\expl_{i}$ by either adding one vertex to the explored set $E_{i-1}$, or moving the marked edge to a neighbour edge or half-edge. In the first case, we call $i$ an \emph{exploration step} and, in the second, we call $i$ a \emph{walk step}. In particular, the time of the exploration is not the same as the time of the random walk.

We say that a vertex $v \in E_i$ is \emph{$k$-free on the right} in $E_i$ if none of the $k$ first neighbours on the right of the rightmost child of $E_i$ belongs to $E_i$.  We define similarly a \emph{$k$-free on the left} vertex. We would like to build an exploration of $\hcau$ such that at every exploration step $i$, the unique vertex of $E_{i+1} \backslash E_i$ is either $k$-free on the left or on the right in $E_{i+1}$. Note that it is not always possible by looking at $\expl_i$ to decide whether a vertex $v \in E_i$ is $k$-free or not (cf. vertex $v_2$ on Figure \ref{fig_exploration_method}). However, as we will see later (proof of Lemma \ref{lem_expl_k_free}), it is sometimes possible to be sure that a vertex is $k$-free.

\begin{figure}
\begin{center}
\begin{tikzpicture}

\draw(-1,-1)--(-1,0);
\draw(5,-1)--(5,0);
\draw(-1.4,0)--(5.4,0);
\draw(-1,0)--(-1,1);
\draw(1,0)--(1.5,1);
\draw(2.5,0)--(2,1);
\draw(2.5,0)--(3,1);
\draw(4,0)--(4.5,1);
\draw(5,0)--(5,1);
\draw(-1.4,1)--(5.4,1);
\draw(-1,1)--(-1.5,2);
\draw(-1,1)--(-1,2);
\draw(0,1)--(-0.5,2);
\draw(1,1)--(1.4,2);
\draw(1.5,1)--(1.8,2);
\draw(2,1)--(2.2,2);
\draw(3,1)--(2.6,2);
\draw(3,1)--(3,2);
\draw(4.5,1)--(4,2);
\draw(4.5,1)--(4.5,2);
\draw(5,1)--(5,2);
\draw(-1.9,2)--(5.4,2);
\draw(-1.5,2)--(-2,3);
\draw(-1,2)--(-1.5,3);
\draw(-0.5,2)--(-1,3);
\draw(-0.5,2)--(-0.5,3);
\draw(0,2)--(0,3);
\draw(1,2)--(0.5,3);
\draw(1,2)--(1,3);
\draw(1.4,2)--(1.3,3);
\draw(1.8,2)--(1.6,3);
\draw(1.8,2)--(1.9,3);
\draw(2.2,2)--(2.2,3);
\draw(2.6,2)--(2.5,3);
\draw(3,2)--(2.9,3);
\draw(3.5,2)--(3.3,3);
\draw(3.5,2)--(3.7,3);
\draw(4,2)--(4.1,3);
\draw(4.5,2)--(4.5,3);
\draw(5,2)--(4.9,3);
\draw(5,2)--(5.3,3);
\draw(-2.3,3)--(5.7,3);

\draw(-2,3)--(-2.2,3.3);
\draw(-2,3)--(-1.8,3.3);
\draw(-1.5,3)--(-1.5,3.3);
\draw(-1,3)--(-1,3.3);
\draw(-0.5,3)--(-0.5,3.3);
\draw(0,3)--(-0.2,3.3);
\draw(0,3)--(0.2,3.3);
\draw(0.5,3)--(0.5,3.3);
\draw(1,3)--(1,3.3);
\draw(1.3,3)--(1.3,3.3);
\draw(1.6,3)--(1.6,3.3);
\draw(1.9,3)--(1.9,3.3);
\draw(2.2,3)--(2.1,3.3);
\draw(2.2,3)--(2.3,3.3);
\draw(2.5,3)--(2.6,3.3);
\draw(2.9,3)--(2.9,3.3);
\draw(3.3,3)--(3.2,3.3);
\draw(3.3,3)--(3.4,3.3);
\draw(3.7,3)--(3.7,3.3);
\draw(4.1,3)--(4,3.3);
\draw(4.1,3)--(4.2,3.3);
\draw(4.5,3)--(4.5,3.3);
\draw(4.9,3)--(4.8,3.3);
\draw(5.3,3)--(5.1,3.3);
\draw(5.3,3)--(5.5,3.3);

\draw[very thick, red] (0,-1)--(0,0);
\draw[very thick, red] (1,-1)--(1,0);
\draw[very thick, red] (2.5,-1)--(2.5,0);
\draw[very thick, red] (4,-1)--(4,0);
\draw[very thick, red] (-0.4,0)--(4.4,0);
\draw[very thick, red] (0,0)--(0,1);
\draw[very thick, red] (1,0)--(1,1);
\draw[very thick, orange, ->] (1,0)--(1.25,0.5);
\draw[very thick, red] (2.5,0)--(2.25,0.5);
\draw[very thick, red] (2.5,0)--(2.75,0.5);
\draw[very thick, red] (4,0)--(3.5,1);
\draw[very thick, red] (4,0)--(4.25,0.5);
\draw[very thick, red] (-0.4,1)--(1.25,1);
\draw[very thick, red] (3.25,1)--(3.9,1);
\draw[very thick, red] (0,1)--(-0.25,1.5);
\draw[very thick, red] (0,1)--(0,2);
\draw[very thick, red] (1,1)--(1,2);
\draw[very thick, red] (1,1)--(1.2,1.5);
\draw[very thick, red] (3.5,1)--(3.5,2);
\draw[very thick, red] (-0.25,2)--(1.2,2);
\draw[very thick, red] (3.25,2)--(3.75,2);
\draw[very thick, red] (0,2)--(0,2.5);
\draw[very thick, red] (1,2)--(0.75,2.5);
\draw[very thick, red] (1,2)--(1,2.5);
\draw[very thick, red] (3.5,2)--(3.4,2.5);
\draw[very thick, red] (3.5,2)--(3.6,2.5);

\draw(-1,-1)node{};
\draw(5,-1)node{};
\draw(-1,0)node{};
\draw(5,0)node{};
\draw(-1,1)node{};
\draw(1.5,1)node{};
\draw(2,1)node{};
\draw(3,1)node{};
\draw(4.5,1)node{};
\draw(5,1)node{};
\draw(-1.5,2)node{};
\draw(-1,2)node{};
\draw(-0.5,2)node{};
\draw(1.4,2)node{};
\draw(1.8,2)node{};
\draw(2.2,2)node{};
\draw(2.6,2)node{};
\draw(3,2)node{};
\draw(4,2)node{};
\draw(4.5,2)node{};
\draw(5,2)node{};
\draw(-2,3)node{};
\draw(-1.5,3)node{};
\draw(-1,3)node{};
\draw(-0.5,3)node{};
\draw(0,3)node{};
\draw(0.5,3)node{};
\draw(1,3)node{};
\draw(1.3,3)node{};
\draw(1.6,3)node{};
\draw(1.9,3)node{};
\draw(2.2,3)node{};
\draw(2.5,3)node{};
\draw(2.9,3)node{};
\draw(3.3,3)node{};
\draw(3.7,3)node{};
\draw(4.1,3)node{};
\draw(4.5,3)node{};
\draw(4.9,3)node{};
\draw(5.3,3)node{};

\draw (0,-1)node[moyenrouge]{};
\draw (1,-1)node[moyenrouge]{};
\draw (2.5,-1)node[moyenrouge]{};
\draw (4,-1)node[moyenrouge]{};
\draw (0,0)node[moyenrouge]{};
\draw (1,0)node[moyenrouge]{};
\draw (2.5,0)node[moyenrouge]{};
\draw (4,0)node[moyenrouge]{};
\draw (0,1)node[moyenrouge]{};
\draw (1,1)node[moyenrouge]{};
\draw (3.5,1)node[moyenrouge]{};
\draw (0,2)node[moyenrouge]{};
\draw (1,2)node[moyenrouge]{};
\draw (3.5,2)node[moyenrouge]{};

\draw (0.7,0.3)node[texte]{$x_1$};
\draw (3.8,1.3)node[texte]{$x_2$};
\draw (6,0.5)node[texte]{$\longrightarrow$};

\begin{scope}[shift={(7,0)}]
\draw (0,-1)--(0,0);
\draw (1,-1)--(1,0);
\draw (2.5,-1)--(2.5,0);
\draw (4,-1)--(4,0);
\draw (-0.4,0)--(4.4,0);
\draw (0,0)--(0,1);
\draw (1,0)--(1,1);
\draw[very thick, orange, ->] (1,0)--(1.25,0.5);
\draw (2.5,0)--(2.25,0.5);
\draw (2.5,0)--(2.75,0.5);
\draw (4,0)--(3.5,1);
\draw (4,0)--(4.25,0.5);
\draw (-0.4,1)--(1.25,1);
\draw (3.25,1)--(3.9,1);
\draw (0,1)--(-0.25,1.5);
\draw (0,1)--(0,2);
\draw (1,1)--(1,2);
\draw (1,1)--(1.2,1.5);
\draw (3.5,1)--(3.5,2);
\draw (-0.25,2)--(1.2,2);
\draw (3.25,2)--(3.75,2);
\draw (0,2)--(0,2.5);
\draw (1,2)--(0.75,2.5);
\draw (1,2)--(1,2.5);
\draw (3.5,2)--(3.4,2.5);
\draw (3.5,2)--(3.6,2.5);

\draw (0,-1)node{};
\draw (1,-1)node{};
\draw (2.5,-1)node{};
\draw (4,-1)node{};
\draw (0,0)node{};
\draw (1,0)node{};
\draw (2.5,0)node{};
\draw (4,0)node{};
\draw (0,1)node{};
\draw (1,1)node{};
\draw (3.5,1)node{};
\draw (0,2)node{};
\draw (1,2)node{};
\draw (3.5,2)node{};

\draw (0.7,0.3)node[texte]{$x_1$};
\draw (3.8,1.3)node[texte]{$x_2$};
\end{scope}

\end{tikzpicture}
\end{center}
\caption{On the left, the map $\hcau$ and the set of vertices $E_i$ (in red). On the right, the explored map $\expl_i$. The edge in orange means that at the current time, the random walk is leaving the vertex $x_1$ towards its right child. The vertex $x_1$ is $2$-free on the right, but not $3$-free. The vertex $x_2$ is $k$-free on the right for every $k \geq 0$. It is also $5$-free on the left, but not $6$-free. By only looking at $\expl_i$, we can be sure that $x_1$ is $2$-free on the right and $x_2$ is $4$-free on the left, but not that $x_2$ is $5$-free on the left.}\label{fig_exploration_method}
\end{figure}
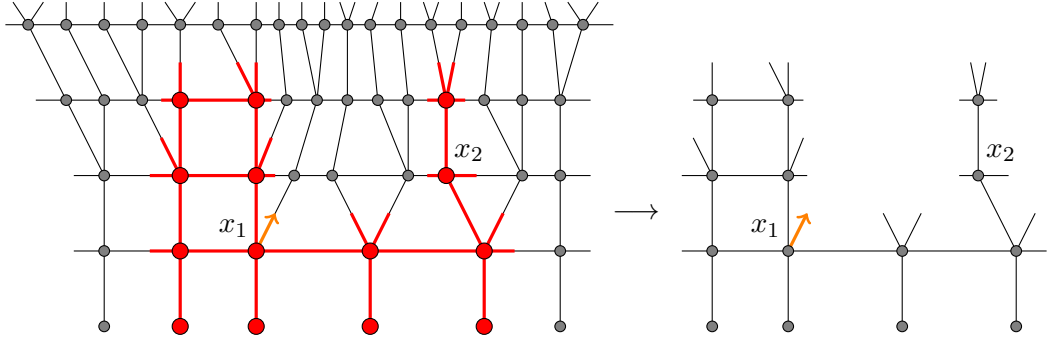

\paragraph{Choice of the exploration method.}
The simplest exploration method coming into mind is to explore a vertex when it is hit for the first time by the walk $(X_n)$. However, this is not suitable for our purpose. If for example we discover some descendants of a vertex $v$ and explore $v$ afterwards, then we have some partial information about the descendance of $v$, so we cannot control the probability for $v$ to be $k$-bad. For this reason, we never want to discover a vertex before discovering all its ancestors. We will therefore require that the sets $E_i$ are \emph{stable}, which means that for every vertex $v \in E_i$, all the ancestors of $v$ lie in $E_i$ as well.

A second natural exploration method is now the following: everytime the walk $(X_n)$ hits a vertex $v$ for the first time, we discover all the ancestors of $v$ that are yet undiscovered (including $v$), from the lowest to the highest. Although more convenient than the first one, this method is not sufficient either. Indeed, assume that at some point we explore a vertex $v$ such that the $k$ first neighbours on the left and on the right of $v$ have already been discovered, as well as many of their descendants. Then we have accumulated some partial information about the descendances of all the neighbours of $v$, so we cannot control the probability for $v$ to be $k$-bad.

More generally, this problem occurs if we allow the formation of "narrow pits" in the explored part of $\hcau$. Therefore, the idea of our exploration is the following: everytime the walk $(X_n)$ hits a vertex $v$ for the first time, we explore all the ancestors of $v$ that are yet undiscovered (including $v$), from the lowest to the highest. If by doing so we create a pit of width at most $2k$, then we explore completely the bottom of this pit, until it has width greater than $2k$.

To define this exploration more precisely, we need to define precisely a pit. We assume that the set $E_i$ is stable, which implies that $\expl_i$ is simply connected. Then there is a unique way to move around the map $\expl_i$ from left to right by crossing all the half-edges exactly once, as on Figure \ref{figure_pit}. A \emph{pit} is a sequence of consecutive half-edges $p_0, p_1, \dots, p_{j+1}$ such that:
\begin{itemize}
\item
the half-edges $p_1, \dots, p_j$ point upwards, and start from the same height $h$,
\item
the half-edge $p_0$ points to the right and lies at height $h+1$,
\item
the half-edge $p_{j+1}$ points to the left and lies at height $h+1$.
\end{itemize}
Note that all the pits are half-edge-disjoint. We call $j$ and $h$ the \emph{width} and the \emph{height} of the pit. Finally, we say that $\expl_i$ is \emph{$k$-flat} if it has no pit of width $j \leq 2k$ (see Figure \ref{figure_pit}).

\begin{figure}
\begin{center}
\begin{tikzpicture}[scale=1.2]
\draw (0,-1)--(0,0);
\draw (1,-1)--(1,0);
\draw (2,-1)--(2,0);
\draw (3,-1)--(3,0);
\draw (-0.5,0)--(3.5,0);
\draw (0,0)--(0,1);
\draw (1,0)--(1,1);
\draw (1,0)--(1.4,0.5);
\draw (2,0)--(1.8,0.5);
\draw (2,0)--(2.2,0.5);
\draw (3,0)--(2.5,1);
\draw (3,0)--(3.2,0.5);
\draw (-0.5,1)--(1.5,1);
\draw (2,1)--(3,1);
\draw (0,1)--(-0.5,1.5);
\draw (0,1)--(0,2);
\draw (1,1)--(1,2);
\draw (1,1)--(1.4,1.5);
\draw (2.5,1)--(2.5,2);
\draw (-0.5,2)--(1.5,2);
\draw (2,2)--(3,2);
\draw (0,2)--(0,2.5);
\draw (1,2)--(0.8,2.5);
\draw (1,2)--(1.2,2.5);
\draw (2.5,2)--(2.3,2.5);
\draw (2.5,2)--(2.7,2.5);

\draw[very thick, olive, dashed](-0.3,-0.5)--(-0.3,1) to[out=90, in=270] (-0.2,1.5)--(-0.2,2) to[out=90, in=180] (0,2.3)--(1,2.3) to[out=0, in=90] (1.3,2)--(1.3,0.5) to[out=270, in=180] (1.5,0.2)--(2.3,0.2) to[out=0, in=270] (2.2,1)--(2.2,2) to[out=90, in=90] (2.8,2)--(2.8,1) to[out=270, in=90] (3.3,0)--(3.3,-0.5);

\draw (0,-1)node{};
\draw (1,-1)node{};
\draw (2,-1)node{};
\draw (3,-1)node{};
\draw (0,0)node{};
\draw (1,0)node{};
\draw (2,0)node{};
\draw (3,0)node{};
\draw (0,1)node{};
\draw (1,1)node{};
\draw (2.5,1)node{};
\draw (0,2)node{};
\draw (1,2)node{};
\draw (2.5,2)node{};

\begin{scope}[shift={(6,0)}]
\draw (0,-1)--(0,0);
\draw (1,-1)--(1,0);
\draw (2,-1)--(2,0);
\draw (3,-1)--(3,0);
\draw (-0.5,0)--(3.5,0);
\draw (0,0)--(0,1);
\draw (1,0)--(1,1);
\draw (1,0)--(1.4,0.5);
\draw (2,0)--(1.8,0.5);
\draw (2,0)--(2.2,0.5);
\draw (3,0)--(2.5,1);
\draw (3,0)--(3.2,0.5);
\draw (-0.5,1)--(1.5,1);
\draw (2,1)--(3,1);
\draw (0,1)--(-0.5,1.5);
\draw (0,1)--(0,2);
\draw (1,1)--(1,2);
\draw (1,1)--(1.4,1.5);
\draw (2.5,1)--(2.5,2);
\draw (-0.5,2)--(1.5,2);
\draw (2,2)--(3,2);
\draw (0,2)--(0,2.5);
\draw (1,2)--(0.8,2.5);
\draw (1,2)--(1.2,2.5);
\draw (2.5,2)--(2.3,2.5);
\draw (2.5,2)--(2.7,2.5);

\draw[very thick, blue, dashed](1.3,1)--(1.3,0.5) to[out=270, in=180] (1.5,0.2)--(2.3,0.2) to[out=0, in=270] (2.2,1);

\draw (0,-1)node{};
\draw (1,-1)node{};
\draw (2,-1)node{};
\draw (3,-1)node{};
\draw (0,0)node{};
\draw (1,0)node{};
\draw (2,0)node{};
\draw (3,0)node{};
\draw (0,1)node{};
\draw (1,1)node{};
\draw (2.5,1)node{};
\draw (0,2)node{};
\draw (1,2)node{};
\draw (2.5,2)node{};

\draw[blue] (1.35,1.15)node[texte]{$p_0$};
\draw[blue] (1.15,0.5)node[texte]{$p_1$};
\draw[blue] (1.7,0.65)node[texte]{$p_2$};
\draw[blue] (2.1,0.65)node[texte]{$p_3$};
\draw[blue] (2.2,1.15)node[texte]{$p_4$};

\end{scope}

\end{tikzpicture}
\end{center}
\caption{On the left, we move around $\expl_i$ by crossing all the half-edges. On the right, the unique pit of $\expl_i$. It has width $3$ and height $0$. In particular, the map $\expl_i$ is $1$-flat, but not $2$-flat.}\label{figure_pit}
\end{figure}
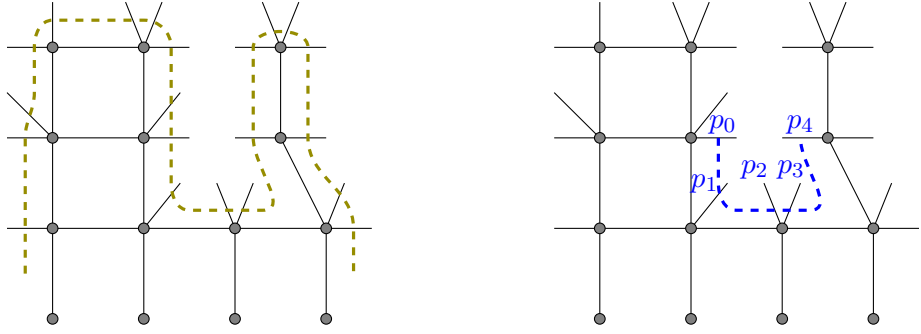

If $e$ is an oriented edge or half-edge, we will denote by $e^-$ its starting point and $e^+$ its endpoint. We can now describe our exploration algorithm precisely. We take for $E_0$ the set formed by the root vertex $\rho$ of $\hcau$ and its parent at height $-1$, and pick $e_0$ uniformly among all the edges and half-edges started from $\rho$. For every $i \geq 0$, we recall that $e_i$ is the oriented edge or half-edge of $\expl_i$ marking the position of the simple random walk. For every $i \geq 1$, given $(\expl_{i-1}, e_{i-1})$, we construct $(\expl_i, e_i)$ as follows.
\begin{enumerate}
\item[(i)]
If the marked edge $e_{i-1}$ is a full edge and $\expl_{i-1}$ is $k$-flat, we perform a \emph{walk step}: we set $E_{i}=E_{i-1}$ and pick $e_{i}$ uniformly among all the edges and half-edges whose starting point is $e_{i-1}^+$.
\item[(ii)]
If $e_{i-1}$ is a half-edge, we perform an \emph{exploration step}: we denote by $v_{i}$ the lowest ancestor of $e_{i-1}^+$ that does not belong to $E_{i-1}$. If $v_i$ lies at height $-1$, then $E_{i}$ is the union of $E_{i-1}$, the vertex $v_i$ and its child at height $0$ (this is the only case where we explore two vertices at once, to make sure that $\expl_i$ remains connected). If not, then $E_{i}=E_{i-1} \cup \{v_i\}$. Note that if $v_i=e_{i-1}^+$, then the marked half-edge becomes a full edge.
\item[(iii)]
If $e_{i-1}$ is a full edge but $\expl_{i-1}$ is not $k$-flat, we also perform an \emph{exploration step}: let $(p_0, p_1, \dots, p_{j+1})$ be the leftmost pit of width at most $2k$. We then take $E_{i}=E_{i-1} \cup \{p_1^+\}$. This means that we explore the endpoint of the leftmost vertical half-edge of the pit.
\end{enumerate}
It is easy to check that the sets $E_i$ we just defined are all stable.

\paragraph{The exploration is Markovian.}
An important feature of our exploration that we need to check is that it is Markovian. More precisely, for every $i \geq 0$, let $\partial E_i$ be the set of vertices consisting of:
\begin{itemize}
\item
the endpoints of the half-edges of $\expl_i$ pointing upwards,
\item
the vertices of $\hcau$ of height $0$ that do not belong to $E_i$.
\end{itemize}

\begin{lem}\label{lem_expl_markov}
Conditionally on $\left( \expl_j, e_j \right)_{0 \leq j \leq i}$, the trees of descendants of the vertices of $\partial E_i$ are independent Galton--Watson trees with offspring distribution $\mu$.
\end{lem}

\begin{proof}
Given the Markovian structure of supercritical Galton--Watson tree, it is enough to check that at every step, our exploration is independent of the part of $\hcau$ that has not yet been discovered. This is easy in the case (i): conditionally on $\expl_i$, the choice of $e_i$ is independent of the rest of $\hcau$. We claim that in the other two cases, the discovered vertex is either the endpoint of a vertical half-edge which is a deterministic function of $(\expl_i,e_i)$, or the root of the first infinite tree on the left or on the right of $\expl_i$.
This claim is obvious in the case (iii).

In the case (ii), the half-edge $e_i$ points either to the top, the left or the right.
If it points to the top, then all the ancestors of $e_i^+$ have been discovered, so the explored vertex is $e_i^+$.
If $e_i$ does not point to the top, we assume without loss of generality that it points to the right. We then start from the half-edge $e_i$ and move around $\expl_i$ towards the right. We first cross horizontal half-edges pointing to the right at decreasing heights, until either we reach height $0$, or we cross a vertical half-edge pointing to the top.
If we cross a first vertical half-edge $e$, then $e$ must be an ancestor of $e_i^+$ (indeed, $e$ has descendants at the same height as $e_i^+$, and there is no other half-edge pointing to the top between $e_i^-$ and $e_*^-$, see Figure \ref{figure_first_ancestor}). Therefore, the explored vertex must be $e_*^+$.
Finally, if we reach the bottom boundary, then the explored vertex is the root of the first tree on the right of $\expl_i$ (and its parent at height $-1$).

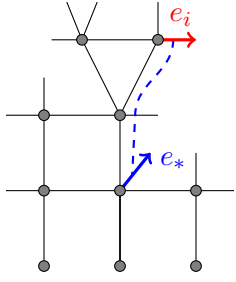
\begin{figure}
\begin{center}
\begin{tikzpicture}

\draw (0,-1)--(0,0);
\draw (1,-1)--(1,0);
\draw (2,-1)--(2,0);
\draw (0,0)--(0,1);
\draw (1,0)--(1,1);
\draw (2,0)--(2,0.5);
\draw (1,1)--(0.5,2);
\draw (1,1)--(1.5,2);
\draw (1,-1)--(1,0);
\draw (1,-1)--(1,0);
\draw (0,1)--(0,1.5);
\draw (0.5,2)--(0.3,2.5);
\draw (0.5,2)--(0.7,2.5);
\draw (1.5,2)--(1.5,2.5);
\draw (-0.5,0)--(2.5,0);
\draw (-0.5,1)--(1.5,1);
\draw (0.1,2)--(1.9,2);

\draw[dashed, blue, thick] (1.7,2) to[out=270, in=90] (1.2,1) -- (1.16,0.2);
\draw[->, red, very thick] (1.5,2)--(2,2);
\draw[->, blue, very thick] (1,0)--(1.4, 0.5);

\draw(0,-1) node{};
\draw(1,-1) node{};
\draw(2,-1) node{};
\draw(0,0) node{};
\draw(1,0) node{};
\draw(2,0) node{};
\draw(0,1) node{};
\draw(1,1) node{};
\draw(0.5,2) node{};
\draw(1.5,2) node{};

\draw[red] (1.8,2.3) node[texte]{$e_i$};
\draw[blue] (1.7,0.4) node[texte]{$e_*$};

\end{tikzpicture}
\end{center}
\caption{The second case of our exploration algorithm. We move around the boundary of $\expl_i$ towards the right: we encounter one horizontal half-edge, and then a vertical half-edge $e_*$. This $e_*$ has descendents at the same height as $e_i$, and the leftmost such descendant is $e_i^+$.}\label{figure_first_ancestor}
\end{figure}
\end{proof}

We denote by $\varphi(n)$ the $n$-th walk step of our exploration, with $\varphi(0)=0$. Note that the marked edge or half-edge from the time $\varphi(n)$ to the time $\varphi(n+1)-1$ corresponds to the edge in $\hcau$ from $X_n$ to $X_{n+1}$. Therefore, at time $\varphi(n)$, the explored part covers $\{X_0, X_1, \dots, X_n\}$. As explained in the beginning of this subsection, it is important to control $\varphi$. We can now state the two important properties that our exploration satisfies.

\begin{lem}\label{lem_expl_phi}
There is a deterministic constant $c$ such that for every $n \geq 0$, we have
\[\varphi(n+1)-\varphi(n) \leq ck(n+1).\]
\end{lem}

\begin{lem}\label{lem_expl_k_free}
For every exploration step $i$, the unique vertex $v_i$ of $E_i \backslash E_{i-1}$ that does not lie at height $-1$ is either $k$-free on the left or $k$-free on the right.
\end{lem}

Note that our exploration method was precisely designed to satisfy Lemma \ref{lem_expl_k_free}. Both these lemmas are completely deterministic: they hold if we replace $\hcau$ by any infinite causal map with no leaf, and $(X_n)$ by any infinite path. Before proving them, we explain how to conclude the proof of Lemma \ref{lem_bad_points} given these two results.

\begin{proof}[Proof of Lemma \ref{lem_bad_points} given Lemmas \ref{lem_expl_phi} and \ref{lem_expl_k_free}]
For every $i \geq 0$, let $\mathcal{F}_i$ be the $\sigma$-algebra generated by $\left( \expl_j, e_j \right)_{0 \leq j \leq i}$. We first show that, for every exploration step $i \geq 0$, we have
\begin{equation}\label{equa_bad_point}
\P \left( \mbox{$v_i$ is $k$-bad} | \mathcal{F}_i \right) \leq \mu(1)^{k^2}.
\end{equation}
Let $i \geq 0$ be an exploration step. By Lemma \ref{lem_expl_k_free}, without loss of generality, we may assume that $v_i$ is $k$-free on the right. Let $v_i^1, \dots, v_i^k$ be the $k$ first neighbours on the right of the rightmost child of $v_i$. Since $v_i$ is $k$-free on the right, these vertices do not belong to $E_i$. By Lemma \ref{lem_expl_markov}, conditionally on $\mathcal{F}_i$, their trees of descendants are i.i.d. Galton--Watson trees with offspring distribution $\mu$.
If the vertex $v_i$ is $k$-bad, each of the vertices $v_i^1, \dots, v_i^k$ has only one descendant at height $h+k+1$, so we have
\begin{eqnarray*}
\P \left( \mbox{$v_i$ is $k$-bad} | \mathcal{F}_i \right) & \leq & \P \left( \mbox{$v_i^1, \dots, v_i^k$ have one descendant each at height $h+k+1$} | \mathcal{F}_i \right)\\
&=& \prod_{j=1}^k \P \left( \mbox{$v_i^j$ has exactly one descendant at height $h+k+1$} | \mathcal{F}_i \right)\\
&=& \left( \mu(1)^k \right)^k,
\end{eqnarray*}
and we obtain \eqref{equa_bad_point}, which implies
\[ \P \left( \mbox{$i$ is an exploration step and $v_i$ is $k$-bad} \right) \leq \mu(1)^{k^2}\]
for every $i \geq 0$.
By summing Lemma \ref{lem_expl_phi}, we obtain $\varphi(n+1) \leq ck(n+1)^2$. Therefore, the vertices $X_0, X_1, \dots, X_n$ all lie in $\expl_{ck(n+1)^2}$. If one of these points is $k$-bad, it cannot lie at height $-1$ by definition of a bad point, so it is equal to $v_i$ for some $0 \leq i \leq ck(n+1)^2$. Therefore, we have
\begin{eqnarray*}
\P \left( \mbox{one of the points $X_0, X_1, \dots, X_n$ is $k$-bad} \right) & \leq & \sum_{i=0}^{ck(n+1)^2} \P \left( \mbox{$v_i$ is $k$-bad} \right)\\
& \leq & ck(n+1)^2 \mu(1)^{k^2},
\end{eqnarray*}
which ends the proof.
\end{proof}

\begin{proof}[Proof of Lemma \ref{lem_expl_phi}]
To bound $\varphi(n+1)-\varphi(n)$, we describe precisely what happens between the times $\varphi(n)$ and $\varphi(n+1)$. Note that $\expl_{\varphi(n)}=\expl_{\varphi(n)-1}$ is $k$-flat (if it was not, $\varphi(n)$ would have to be an exploration step). Hence, if $e_{\varphi(n)}$ is a full edge, then $\varphi(n)+1$ is a walk step and we have $\varphi(n+1)=\varphi(n)+1$, so it is only necessary to treat the case where $e_{\varphi(n)}$ is a half-edge.

In this case, the first thing our algorithm does is to explore the vertex $e_{\varphi(n)}^+$ and all its undiscovered ancestors. We know that $e_{\varphi(n)}^+=X_{n+1}$, so in particular its height is at most $n+1$. Therefore, exploring all its ancestors takes at most $n+2$ steps.

We can now perform the $(n+1)$-th walk step, except if exploring $e_{\varphi(n)}^+$ and its ancestors has created a new pit of width at most $2k$. This can happen in two different ways, as on Figure \ref{fig_creating_pits}:
\begin{itemize}
\item
if $e_{\varphi(n)}$ is vertical, exploring $e_{\varphi(n)}^+$ may split an existing pit in two,
\item
if $e_{\varphi(n)}$ is horizontal (say it points to the right), exploring its ancestors may decrease the width of an existing pit on the right of $e_{\varphi(n)}$.
\end{itemize}

\begin{figure}
\begin{center}
\begin{tikzpicture}
\draw (0,-1)--(0,0);
\draw (1,-1)--(1,0);
\draw (2.5,-1)--(2.5,0);
\draw (4,-1)--(4,0);
\draw (-0.5,0)--(4.5,0);
\draw (0,0)--(0,1);
\draw (1,0)--(1,1);
\draw (1,0)--(1.4,0.5);
\draw (2.5,0)--(2.3,0.5);
\draw (2.5,0)--(2.7,0.5);
\draw (4,0)--(3.5,1);
\draw (4,0)--(4.2,0.5);
\draw (-0.5,1)--(1.5,1);
\draw (3,1)--(4,1);
\draw (0,1)--(-0.5,1.5);
\draw (0,1)--(0,2);
\draw (1,1)--(1,2);
\draw (3.5,1)--(3.5,2);
\draw (-0.5,2)--(1.5,2);
\draw (3,2)--(4,2);
\draw (0,2)--(0,2.5);
\draw (1,2)--(0.8,2.5);
\draw (1,2)--(1.2,2.5);
\draw (3.5,2)--(3.3,2.5);
\draw (3.5,2)--(3.7,2.5);

\draw[->, red, very thick] (1,2)--(1.5,2);

\draw (0,-1)node{};
\draw (1,-1)node{};
\draw (2.5,-1)node{};
\draw (4,-1)node{};
\draw (0,0)node{};
\draw (1,0)node{};
\draw (2.5,0)node{};
\draw (4,0)node{};
\draw (0,1)node{};
\draw (1,1)node{};
\draw (3.5,1)node{};
\draw (0,2)node{};
\draw (1,2)node{};
\draw (3.5,2)node{};

\draw[red] (1.5,1.7)node[texte]{$e_{\varphi(n)}$};
\draw (5.5,0.5)node[texte]{$\longrightarrow$};
\draw (2,-1.5)node[texte]{$\expl_{\varphi(n)}$};

\begin{scope}[shift={(7,0)}]
\draw (0,-1)--(0,0);
\draw (1,-1)--(1,0);
\draw (2.5,-1)--(2.5,0);
\draw (4,-1)--(4,0);
\draw (-0.5,0)--(4.5,0);
\draw (0,0)--(0,1);
\draw (1,0)--(1,1);
\draw (1,0)--(1.8,1);
\draw (2.5,0)--(2.3,0.5);
\draw (2.5,0)--(2.7,0.5);
\draw (4,0)--(3.5,1);
\draw (4,0)--(4.2,0.5);
\draw (-0.5,1)--(2.2,1);
\draw (3,1)--(4,1);
\draw (0,1)--(-0.5,1.5);
\draw (0,1)--(0,2);
\draw (1,1)--(1,2);
\draw (3.5,1)--(3.5,2);
\draw (-0.5,2)--(1.5,2);
\draw (3,2)--(4,2);
\draw (0,2)--(0,2.5);
\draw (1,2)--(0.8,2.5);
\draw (1,2)--(1.2,2.5);
\draw (3.5,2)--(3.3,2.5);
\draw (3.5,2)--(3.7,2.5);
\draw (1.8,1)--(1.8,2.5);
\draw (1.8,2)--(2.2,2);

\draw[->, red, very thick] (1,2)--(1.5,2);
\draw[red, very thick] (1,2)--(1.8,2);

\draw[dashed, blue, very thick] (2,1) to[out=270, in=180] (2.4,0.25)--(2.6,0.25) to[out=0, in=270] (3.25,1);

\draw (0,-1)node{};
\draw (1,-1)node{};
\draw (2.5,-1)node{};
\draw (4,-1)node{};
\draw (0,0)node{};
\draw (1,0)node{};
\draw (2.5,0)node{};
\draw (4,0)node{};
\draw (0,1)node{};
\draw (1,1)node{};
\draw (3.5,1)node{};
\draw (0,2)node{};
\draw (1,2)node{};
\draw (3.5,2)node{};
\draw (1.8,1)node{};
\draw (1.8,2)node{};

\draw[red] (1.6,1.6)node[texte]{$e_{\varphi(n)+2}$};
\draw (2,-1.5)node[texte]{$\expl_{\varphi(n)+2}$};
\end{scope}

\begin{scope}[shift={(0,5)}]
\draw (0,-1)--(0,0);
\draw (1,-1)--(1,0);
\draw (2.5,-1)--(2.5,0);
\draw (4,-1)--(4,0);
\draw (-0.5,0)--(4.5,0);
\draw (0,0)--(0,1);
\draw (1,0)--(1,1);
\draw (1,0)--(1.4,0.5);
\draw (2.5,0)--(2.3,0.5);
\draw (2.5,0)--(2.7,0.5);
\draw (4,0)--(3.5,1);
\draw (4,0)--(4.2,0.5);
\draw (-0.5,1)--(1.5,1);
\draw (3,1)--(4,1);
\draw (0,1)--(-0.5,1.5);
\draw (0,1)--(0,2);
\draw (1,1)--(1,2);
\draw (3.5,1)--(3.5,2);
\draw (-0.5,2)--(1.5,2);
\draw (3,2)--(4,2);
\draw (0,2)--(0,2.5);
\draw (1,2)--(0.8,2.5);
\draw (1,2)--(1.2,2.5);
\draw (3.5,2)--(3.3,2.5);
\draw (3.5,2)--(3.7,2.5);

\draw[->, red, very thick] (2.5,0)--(2.3,0.5);

\draw (0,-1)node{};
\draw (1,-1)node{};
\draw (2.5,-1)node{};
\draw (4,-1)node{};
\draw (0,0)node{};
\draw (1,0)node{};
\draw (2.5,0)node{};
\draw (4,0)node{};
\draw (0,1)node{};
\draw (1,1)node{};
\draw (3.5,1)node{};
\draw (0,2)node{};
\draw (1,2)node{};
\draw (3.5,2)node{};

\draw[red] (1.9,0.3)node[texte]{$e_{\varphi(n)}$};
\draw (5.5,0.5)node[texte]{$\longrightarrow$};
\draw (2,-1.5)node[texte]{$\expl_{\varphi(n)}$};
\end{scope}

\begin{scope}[shift={(7,5)}]
\draw (0,-1)--(0,0);
\draw (1,-1)--(1,0);
\draw (2.5,-1)--(2.5,0);
\draw (4,-1)--(4,0);
\draw (-0.5,0)--(4.5,0);
\draw (0,0)--(0,1);
\draw (1,0)--(1,1);
\draw (1,0)--(1.4,0.5);
\draw (2.5,0)--(2.3,0.5);
\draw (2.5,0)--(2.7,0.5);
\draw (4,0)--(3.5,1);
\draw (4,0)--(4.2,0.5);
\draw (-0.5,1)--(1.5,1);
\draw (3,1)--(4,1);
\draw (0,1)--(-0.5,1.5);
\draw (0,1)--(0,2);
\draw (1,1)--(1,2);
\draw (3.5,1)--(3.5,2);
\draw (-0.5,2)--(1.5,2);
\draw (3,2)--(4,2);
\draw (0,2)--(0,2.5);
\draw (1,2)--(0.8,2.5);
\draw (1,2)--(1.2,2.5);
\draw (3.5,2)--(3.3,2.5);
\draw (3.5,2)--(3.7,2.5);
\draw (2.1,1)--(1.9,1.5);
\draw (2.1,1)--(2.3,1.5);
\draw (1.8,1)--(2.6,1);

\draw[->, red, very thick] (2.5,0)--(2.26,0.6);
\draw[red, very thick] (2.5,0)--(2.1,1);

\draw[dashed, very thick, blue] (1.2,1) to[out=270, in=135] (1.2,0.25) to[out=315, in=270] (1.9,1);
\draw[dashed, very thick, blue] (2.4,1) to[out=270, in=135] (2.6,0.25) to[out=315, in=270] (3.2,1);

\draw (0,-1)node{};
\draw (1,-1)node{};
\draw (2.5,-1)node{};
\draw (4,-1)node{};
\draw (0,0)node{};
\draw (1,0)node{};
\draw (2.5,0)node{};
\draw (4,0)node{};
\draw (0,1)node{};
\draw (1,1)node{};
\draw (3.5,1)node{};
\draw (0,2)node{};
\draw (1,2)node{};
\draw (3.5,2)node{};
\draw (2.1,1)node{};

\draw[red] (1.8,-0.3)node[texte]{$e_{\varphi(n)+1}$};
\draw (2,-1.5)node[texte]{$\expl_{\varphi(n)+1}$};
\end{scope}

\end{tikzpicture}
\end{center}
\caption{The two ways our exploration can create new pits between two walk steps. On the top, a pit of width $3$ is split into two pits of width $1$. On the bottom, the width of a pit decreases from $3$ to $2$. The new pits of width at most $2$ are indicated in blue.}\label{fig_creating_pits}
\end{figure}
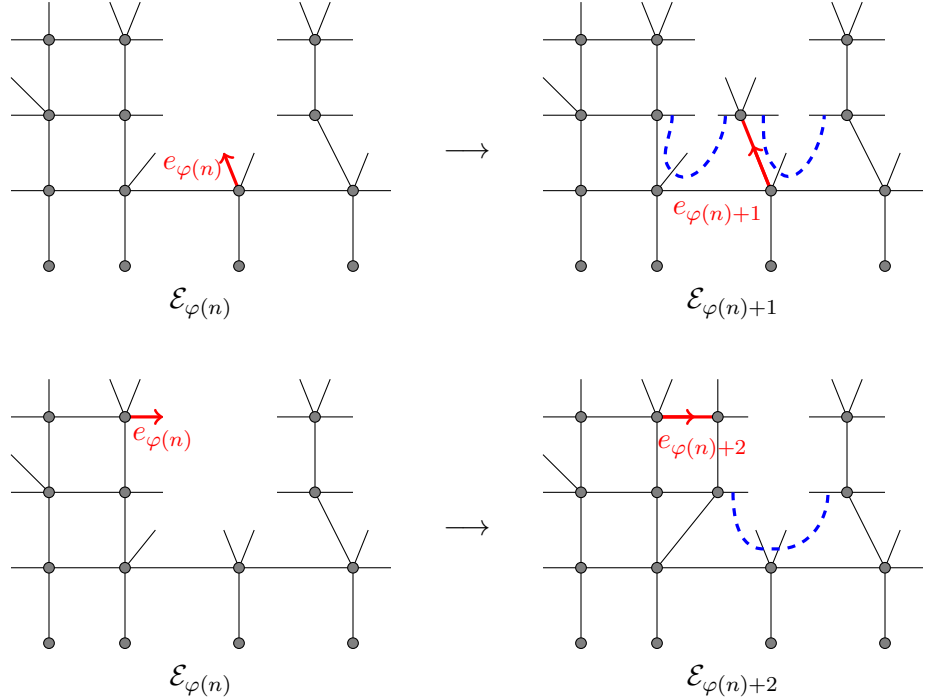

Note that in the first case, we can create at most two new pits, whereas in the second, we can shrink only one (cf. Figure \ref{fig_creating_pits}). Hence, we will create at most $2$ narrow pits, at the same height $h$.

In the first case (top of Figure \ref{fig_creating_pits}), our algorithm will then fill the pit on the left if it has width at most $2k$, and then the pit on the right. The number of steps this takes is at most $2 \times 2k=4k$. By doing so, we may create a new pit at height $h+1$. However, since the trees we work with have no leaf, this pit is at least as wide as the pit of $\expl_{\varphi(n)}$ in which $e_{\varphi(n)}$ lies. Hence, the new pit at height $h+1$ has width greater than $2k$, and does not need to be filled. Therefore, in the first case, the number of exploration steps needed to fill all the narrow pits is at most $4k$.

In the second case (bottom of Figure \ref{fig_creating_pits}), if the pit has width at most $2k$, our algorithm will explore all the vertical half-edges of this pit, from left to right. This takes at most $2k$ steps. Once again, this creates a new pit at height $h+1$, but this time this pit may have width $2k$ or less. If this is the case, our algorithm will explore all its half-edges and perhaps create a pit at height $h+2$, and so on. Note that the maximal height of $\expl_{\varphi(n)}$ is at most $n$. Indeed, the only times at which this maximal height increases is when the random walk reaches some height for the first time, so the maximal height of $\expl_{\varphi(n)}$ cannot be larger than the maximal height of the random walk during its first $n$ steps. Therefore, we will need to fill at most $n$ pits (at heights between $0$ and $n-1$), each of which taking at most $2k$ steps. Hence, filling the narrow pits takes at most $2kn$ steps.

Therefore, in both cases, the number of steps needed to obtain a $k$-flat map and perform a new walk step is bounded by $\max(4k, 2kn)$. If we add the number of steps needed to explore the ancestors of $e_{\varphi(n)}^+$ and the walk step $\varphi(n+1)$, we obtain
\[ \varphi(n+1)-\varphi(n) \leq \max(4k,2kn)+ (n+2)+1 \leq ck(n+1),\]
with e.g. $c=7$.
\end{proof}

We finally prove Lemma \ref{lem_expl_k_free}. The proof will make use of the "step by step" description of our exploration that we also used in the last proof. We recall that for every exploration step $i$, we call $v_i$ the unique vertex of nonnegative height in $E_i \backslash E_{i-1}$, and $e_i$ the oriented edge or half-edge marking the current position of the random walk.

\begin{proof}[Proof of Lemma \ref{lem_expl_k_free}]
We fix an exploration step $i \geq 0$. Note that the vertex $v_i$ is always the endpoint of some half-edge of $\expl_{i-1}$, that we denote by $e_*$.

Before moving on to the details of the proof, we explain how it is possible, by only looking at the map $\expl_{i-1}$, to be sure that the vertex $v_i$ is $k$-free on the right in $E_i$. We move along the boundary of $\expl_{i-1}$ from $e_*$ towards the right, and stop when we encounter a vertex of height $h(v_i)+1$. If this never occurs, it means that in $\expl_{i-1}$ and $\expl_i$, there is no vertex at height $h(v_i)+1$ on the right of $v_i$, so $v_i$ is $k$-free on the right. If this occurs, assume that by moving so, we cross at least $k$ vertical half-edges. Since the trees we consider have no leaf, all these vertical half-edges have descendants at height $h(v_i)+1$, which lie on the right of all the children of $v_i$. Moreover, none of these descendants belongs to $\expl_{i}$. Therefore, $v_i$ must be $k$-free on the right in $\expl_i$. Of course, this is also true for $k$-free on the left vertices (see the end of the caption of Figure \ref{fig_exploration_method} for an example). This remark will be implicitly used in all the cases below.

Let $n$ be the integer such that $\varphi(n)<i<\varphi(n+1)$. We distinguish two cases, corresponding to the two "phases" of exploration between $\varphi(n)$ and $\varphi(n+1)$ that we described in the proof of Lemma \ref{lem_expl_phi}. Both of these cases will be separated in a few subcases.

\begin{itemize}
\item
We first treat the case where $e_{i-1}$ is a half-edge, so the explored vertex $v_i$ is an ancestor of $e_{i-1}^+$.
\begin{itemize}
\item
We start with the subcase where $e_{i-1}$ is vertical, and lies in a pit $p$. Since $e_{i-1}$ is vertical, exploring the ancestors of $e_{i-1}^+$ takes only one step, so $i=\varphi(n)+1$ and $v_i=e_{i-1}^+$. Since $i-1$ is a walk step, the pit $p$ has width at least $2k+1$. Without loss of generality, we may assume that at least $k$ of the vertical half-edges of $p$ are on the right of $e_{i-1}$, so $v_i$ is $k$-free on the right.
\item
If $e_{i-1}$ is vertical but is not in a pit, as in the previous case, we have $i=\varphi(n)+1$ and $v_i$ is the endpoint of $e_{i-1}$. Moreover, there is a direction (left or right) such that when we start from $e_{i-1}$ and move along the boundary of $\expl_{i-1}$ in this direction, the height decreases before increasing for the first time (if not, $e_{i-1}$ would be in a pit). Without loss of generality, this direction is the right. Let $p$ be the first pit that we encounter on the right of $e_{i-1}$. If $p$ does not exist, it means that the height never increases again, so there is no vertex on the right of $e_{i-1}$ in $\expl_{i-1}$ that is higher than $e_{i-1}^-$. Therefore, the vertex $v_i=e_{i-1}^+$ is $k$-free on the right in $\expl_i$. If $p$ exists, it has width at least $2k+1$, so $v_i$ is $(2k+1)$-free on the right, and in particular $k$-free.
\item
If $e_{i-1}$ is horizontal, without loss of generality it points to the right. As explained earlier (see Figure \ref{figure_first_ancestor}), the edge $e_*$ is in this case the first vertical half-edge we meet when we move around $\expl_{i-1}$ from $e_{i-1}$ towards the right (except if no such vertical half-edge exists, in which case $v_i$ is the root of a new tree on the right of $\expl_{i-1}$, and $v_i$ is obviously $k$-free on the right). By the same argument as in the previous case, if there is no pit on the right of $e_*$, then $v_i$ is $k$-free on the right (and even $\infty$-free). If there is one and $p$ is the first such pit, note that between the times $\varphi(n)$ and $i$, the pit $p$ has either been untouched, or has been shrunk by $1$. Therefore, at time $i$ it has width at least $2k$, so $v_i$ is $k$-free in $E_i$.
\end{itemize}
\item
We now consider the second "phase", i.e. the case where $\expl_{i-1}$ has a pit of width at most $2k$, and the goal of the exploration step $i$ is to fill it.
\begin{itemize}
\item
If $e_{i-1}$ points to the top, then the pit has been created at time $\varphi(n)+1$ as in the top part of Figure \ref{fig_creating_pits}. Hence, the half-edge $e_*$ belonged at time $\varphi(n)$ to a pit $(p_0, p_1, \dots, p_{j+1})$ of height $h$ and width $j>2k$. Therefore, either $k$ of the vertical half-edges $p_1, \dots, p_j$ lie on the left of $e_*$, or $k$ of them lie on its right (the two cases are not symmetric since the pit is filled from left to right). If $k$ of these half-edges lie on the left of $e_*$, then their $k$ endpoints (of height $h+1$) have been explored before $v_i$, but none of the descendants of these endpoints has been discovered. Therefore, the map $\expl_i$ contains at least $k$ vertical half-edges at height $h+1$ on the left of $v_i$, so $v_i$ is $k$-free on the left. This case is the reason why, in the definition of a $k$-free vertex, we asked the neighbours of the \emph{children} of $v$ to be undiscovered, and not simply the neighbours of $v$.
If $k$ of the half-edges of $p$ lie on the right of $e_*$, the argument is similar (it is actually simpler since the half-edges on the right of $e_*$ have not yet been explored).
\item
If $e_{i-1}$ points to the right, then we are in the bottom case of Figure \ref{fig_creating_pits}: a pit $p$ of width $2k+1$ has been shrunk to width $2k$ during the first phase, resulting in a pit $(p_0, p_1, \dots, p_{2k+1})$ of width $2k$ at some height $h$. Let also $h' \geq h$ be the height of $e_*$. Since the pit is filled layer by layer from the bottom, the half-edge $e_*$ must be a descendant of a half-edge $p_{\ell_0}$ with $1 \leq \ell_0 \leq 2k$. Moreover, our algorithm fills the layers from left to right. Therefore, at time $i$, for every $1 \leq \ell \leq 2k$, we have already explored the descendants of $e_{\ell}$ up to height $h'+1$ if $\ell \leq \ell_0$ and up to height $h'$ if $\ell>\ell_0$. But the vertex $v_i$ lies at height $h'+1$. Therefore, if $\ell_0 \geq k+1$, then $v_i$ has $k$ vertical half-edges on its left so it is $k$-free on the left. On the other hand, if $\ell_0 \leq k$, then $v_i$ has $k$ vertical half-edges on its right at height $h$, so it is $k$-free on the right. This concludes the proof.
\end{itemize}
\end{itemize}

\end{proof}

\subsection{Quasi-positive speed in \texorpdfstring{$\hcau$}{hcau}}\label{causal_subsec_quasi}

The goal of this subsection is to use Lemma \ref{lem_bad_points} to prove that the walk has a speed $n^{1-o(1)}$, which is slightly weaker than positive speed. We will need to "bootstrap" this result in Section \ref{causal_subsec_renewal} to obtain positive speed. We denote by $H_n$ the height of $X_n$. We also write
\[D_n= \max \{ H_k-H_{\ell} |0 \leq k < \ell \leq n \}\]
for the greatest "descent" of $X$ before time $n$.

\begin{prop}\label{prop_quasi_positive}
Let $0<\delta<1$ and $\beta>0$. Then we have
\[\P \left( H_n \leq n^{1-\delta} \right) = o \left( n^{-\beta} \right) \quad \mbox{ and } \quad \P \left( D_n \geq n^{\delta} \right) = o \left( n^{-\beta} \right)\]
as $n \to +\infty$.
\end{prop}

\begin{proof}
We start with the proof of the first estimate, the proof of the second will follow the same lines. We call a vertex of $\hcau$ \emph{good} if it has height $-1$ or if it has at least two children. The idea of the proof is the following: by Lemma \ref{lem_bad_points}, with high probability, the walk does not visit any $k$-bad point before time $n$ for some $k$. Hence, it is never too far from a good point. Therefore, the walk always has a reasonable probability to reach a good point in a near future. It follows that $X$ will visit many good points, so $(H_n)$ will accumulate a large positive drift.

More precisely, let $a>0$ (we will take $a$ large later). We define two events formalizing the ideas we just explained:
\begin{eqnarray*}
A_1 &=& \{ \mbox{none of the vertices $X_0, X_1, \dots, X_n$ is $a \sqrt{\log n}$-bad} \},\\
A_2 &=& \{ \mbox{for every $0 \leq m \leq n-n^{\delta/2}$, one of the points $X_m, X_{m+1}, \dots, X_{m+n^{\delta/2}}$ is good} \}.
\end{eqnarray*}
Then we have
\begin{equation}\label{quasi_positive_speed_decomposition}
\P \left( H_n \leq n^{1-\delta} \right) \leq \P(A_1^c) + \P (A_1 \backslash A_2) + \P \left( A_2 \cap \{ H_n \leq n^{1-\delta} \} \right).
\end{equation}
We start with the first term. By Lemma \ref{lem_bad_points}, we have
\[\P (A_1^c) \leq ca \sqrt{\log n} \, (n+1)^2\, \mu(1)^{a^2 \log n}.\]
Hence, if we choose $a$ large enough (i.e. $a^2>\frac{\beta+2}{-\log \mu(1)}$), we have $\P(A_1^c)=o ( n^{-\beta} )$.

We now bound the second term of \eqref{quasi_positive_speed_decomposition}. For every $0 \leq m \leq n$, let $\mathcal{F}_m$ be the $\sigma$-algebra generated by $\hcau$ and $(X_0, X_1, \dots, X_m)$. If $X_m$ is not $a \sqrt{\log n}$-bad (which is an $\mathcal{F}_m$-measurable event), let $Y$ be the closest good vertex from $X_m$ (we may have $Y=X_m$). We have $d_{\hcau}(X_m, Y) \leq 2a \sqrt{\log n}$, so there is a path from $X_m$ to $Y$ of length at most $2a \sqrt{\log n}$, and visiting only vertices of degree $4$ (except of course $Y$). Therefore, we have
\[ \P \left( \mbox{$X$ visits the vertex $Y$ between time $m$ and time $m+2a \sqrt{\log n}$} | \mathcal{F}_m \right) \geq \left( \frac{1}{4} \right)^{2a \sqrt{\log n}}\]
if $X_m$ is not $a \sqrt{\log n}$-bad. By induction on $i$, we easily obtain, for every $i \geq 0$,
\begin{eqnarray*}
\P \left( \mbox{$X_m, X_{m+1}, \dots, X_{m+2i a\sqrt{\log n}}$ are neither good nor $a\sqrt{\log n}$-bad} \right) & \leq & \left( 1- \frac{1}{4^{2a\sqrt{\log n}}}\right)^i\\
& \leq & \exp \left( -\frac{i}{4^{2a \sqrt{\log n}}} \right).
\end{eqnarray*}
In particular, by taking $i=\frac{n^{\delta/2}}{2a \sqrt{\log n}}$, we obtain, for every $m$:
\[ \P \left( \mbox{$X_m, X_{m+1}, \dots, X_{m+n^{\delta/2}}$ are neither good nor $a\sqrt{\log n}$-bad} \right) \leq \exp \left( -\frac{n^{\delta/2}}{2a \sqrt{\log n} \, 4^{2a \sqrt{\log n}}} \right).\]
If the event $A_1 \backslash A_2$ occurs, then there is an $m$ with $0 \leq m \leq n-n^{\delta/2}$ such that the above event occurs. Therefore, by summing the last equation over $0 \leq m \leq n-n^{\delta/2}$, we obtain
\[
\P (A_1 \backslash A_2)  \leq n \exp \left( -\frac{n^{\delta/2}}{2a \sqrt{\log n} \, 4^{2a \sqrt{\log n}}} \right)\\
= o (n^{-\beta}).
\]
Finally, we bound the third term of \eqref{quasi_positive_speed_decomposition} by the Azuma inequality. For every $n \geq 0$, let
\[M_n=H_n-\sum_{i=0}^{n-1} E_{\hcau} \left[ H_{i+1}-H_i | X_0, X_1, \dots, X_i \right]. \]
It is clear that $M$ is a martingale with $|M_{n+1}-M_n| \leq 2$ for every $n$, and $M_0=0$. Moreover, we have
\[ E_{\hcau} \left[ H_{i+1}-H_i | X_0, X_1, \dots, X_i \right] = \frac{c(X_i)-1}{c(X_i)+3} \mathbbm{1}_{X_i \notin \partial \hcau} +  \mathbbm{1}_{X_i \in \partial \hcau},\]
where we recall that $c(v)$ is the number of children of a vertex $v$. In particular, we have $E_{\hcau} \left[ H_{i+1}-H_i | X_0, X_1, \dots, X_i \right] \geq 0$, and $E_{\hcau} \left[ H_{i+1}-H_i | X_0, X_1, \dots, X_i \right] \geq \frac{1}{5}$ if $X_i$ is a good vertex. If $A_2$ occurs, the walk $X$ must visit at least $n^{1-\delta/2}$ good vertices before time $n$, so we have
\[ \sum_{i=0}^{n-1} E_{\hcau} \left[ H_{i+1}-H_i | X_0, X_1, \dots, X_i \right] \geq \frac{1}{5} n^{1-\delta/2}. \]
Therefore, if the event in the third term of \eqref{quasi_positive_speed_decomposition} occurs, we have
\[M_n \leq n^{1-\delta}-\frac{1}{5} n^{1-\delta/2}<0.\]
On the other hand, the Azuma inequality applied to $M$ gives
\[ P_{\hcau, \rho} \left(  M_n \leq n^{1-\delta}-\frac{1}{5}n^{1-\delta/2} \right) \leq \exp \left( -\frac{1}{8n}  \left( \frac{1}{5} n^{1-\delta/2} - n^{1-\delta} \right)^2 \right),\]
so
\[ \P \left(  M_n \leq n^{1-\delta}-\frac{1}{5}n^{1-\delta/2} \right)=o(n^{-\beta}),\]
which bounds the third term of \eqref{quasi_positive_speed_decomposition}, and proves the first part of Proposition \ref{prop_quasi_positive}.

To prove the second part, we decompose the event $\{D_n \geq n^{\delta} \}$ in the same way as in \eqref{quasi_positive_speed_decomposition}. By the definition of $D_n$, it is enough to show
\begin{equation}\label{final_equation}
\max_{0 \leq k \leq \ell \leq n} \P \left( A_2 \cap \{ H_{\ell} - H_k \leq -n^{\delta} \} \right)=o(n^{-(\beta+2)}),
\end{equation}
and then to sum over $k$ and $\ell$. To prove \eqref{final_equation}, note that if $\ell < k+n^{\delta}$, then $H_{\ell} -H_k > -n^{\delta}$ deterministically. If $\ell \geq k+n^{\delta}$, we use the same argument based on the Azuma inequality as for the first part. Let $0 \leq k < k+n^{\delta} \leq \ell \leq n$. If $A_2$ occurs, then $X$ visits at least $\frac{\ell-k}{n^{\delta/2}}$ good vertices between times $k$ and $\ell$, so
\[ \sum_{i=k}^{\ell-1} E_{\hcau} [H_{i+1}-H_i | X_0, \dots, X_i] \geq \frac{1}{5} \frac{\ell-k}{n^{\delta/2}}.\]
Hence, if the event of \eqref{final_equation} occurs for $k$ and $\ell$, we have
\[ M_{\ell}-M_{k} \leq H_{\ell}-H_k - \frac{1}{5} \frac{\ell-k}{n^{\delta/2}} \leq -n^{\delta}-\frac{1}{5} \frac{\ell-k}{n^{\delta/2}} \leq -\frac{2}{\sqrt{5}} (\ell-k)^{1/2} n^{\delta/4}. \]
But the Azuma inequality gives
\begin{eqnarray*}
\P \left( M_{\ell}-M_k \leq  -\frac{2}{\sqrt{5}} (\ell-k)^{1/2}n^{\delta/4} \right) & \leq & \exp \left( -\frac{( 2 (\ell-k)^{1/2}n^{\delta/4} )^2}{8 \times 5(\ell-k)} \right)\\
& = & \exp \left( -\frac{n^{\delta/2}}{10} \right)\\
&=& o(n^{-(\beta+2)}),
\end{eqnarray*}
which proves \eqref{final_equation} and the second point of Proposition \ref{prop_quasi_positive}.
\end{proof}

\subsection{Positive speed in \texorpdfstring{$\hcau$}{hcau} via regeneration times}\label{causal_subsec_renewal}

For every $0 \leq h < h' \leq +\infty$, we denote by $\bcau_{h,h'}$ the map formed by the vertices of $\hcau$ with height in $\{h, h+1, \dots, h'\}$, in which for every vertex $v$ at height $h$, we have added a vertex below $v$ that is linked only to $v$. We root $\bcau_{h,h'}$ at the vertex $\rho_h$ corresponding to the leftmost descendant of $\rho$ at generation $h$. The \emph{height} of a vertex in $\bcau_{h,h'}$ is its height in $\hcau$, minus $h$, and the height of the additional vertices is $-1$. We denote by $\partial \bcau_{h,h'}$ the set of these additional vertices. Note that for any $h \geq 0$, the rooted map $\left( \bcau_{h, \infty}, \rho_h \right)$ is independent of $\bcau_{0,h}$ and has the same distribution as $\left( \hcau, \rho \right)$. Since this distribution is invariant by horizontal root translation, this is still true for any choice of the root vertex of $\bcau_{h,h'}$ at height $0$, as long as the choice of the root is independent of $\bcau_{h, \infty}$.

\begin{defn}
We say that $n>0$ is a \emph{regeneration time} if $H_i<H_n$ for every $i<n$, and $H_i \geq H_n$ for every $i \geq n$.
We denote by $\tau^1<\tau^2<\dots$ the list of regeneration times in increasing order.
\end{defn}

We also denote by $T_{\partial}$ the first time at which the simple random walk $X$ on $\hcau$ hits $\partial \hcau$. The key of the proof of Theorem \ref{thm_3_positive} will be to combine the two following results.

\begin{prop}\label{regeneration_expectation}
We have $\E \left[ \tau^1 \right]<+\infty$. In particular, $\tau^1<+\infty$ a.s..
\end{prop}

\begin{lem}\label{independence_regeneration}
\begin{enumerate}
\item
Almost surely, $\tau^j <+\infty$ for every $j \geq 1$.
\item
The path-decorated maps $\left( \bcau_{H_{\tau^{j}}, H_{\tau^{j+1}}}, (X_{\tau^j+i})_{0 \leq i \leq \tau^{j+1}-\tau^j} \right)$ for $j \geq 1$ are i.i.d. and have the same distribution as $\left( \bcau_{0, H_{\tau^{1}}}, (X_i)_{0 \leq i \leq \tau^{1}} \right)$ conditioned on $\{ T_{\partial}=+\infty \}$.
\item
In particular, the pairs $\left( \tau^{j+1}-\tau^j, H_{\tau^{j+1}}-H_{\tau^j} \right)$ for $j \geq 1$ are i.i.d. and have the same distribution as $(\tau^1, H_{\tau^1})$ conditioned on $\{ T_{\partial}=+\infty \}$.
\end{enumerate}
\end{lem}

Proposition \ref{regeneration_expectation} will be deduced from the results of Sections \ref{causal_subsec_exploration} and \ref{causal_subsec_quasi}. On the other hand, Lemma \ref{independence_regeneration} is the reason why regeneration times have been used to prove positive speed for many other models. The same property has already been observed and used in various contexts such as random walks in random environments \cite{SZ99}, or biased random walks on Galton--Watson trees \cite{LPP96}. Although the proof is basically the same for our model, we write it formally in Appendix \ref{causal_appendix_renewal}.

Finally, we note that the finiteness of the times $\tau^i$ could be deduced directly from the results of Section \ref{causal_sec_poisson}, even in the case $\mu(0)>0$. However, this is not sufficient to ensure positive speed.

We now explain how to conclude the proof of Theorem \ref{thm_3_positive} from the last two results.

\begin{proof}[Proof of Theorem \ref{thm_3_positive} given Proposition \ref{regeneration_expectation} and Lemma \ref{independence_regeneration}]
By Lemma \ref{hcau_to_cau}, it is enough to prove the result on $\hcau$. By item 3 of Lemma \ref{independence_regeneration} and Proposition \ref{regeneration_expectation}, we have
\[ \E \left[ \tau^2-\tau^1 \right]=\frac{\E [\tau^1 \mathbbm{1}_{\forall n \geq 0, \,  H_n \geq 0}]}{\P (\forall n \geq 0, \, H_n \geq 0)} <+\infty.\]
Moreover, $H_{\tau^2}-H_{\tau^1} \leq \tau^2-\tau^1$, so $\E \left[ H_{\tau^2}-H_{\tau^1} \right]<+\infty$ as well. By Lemma \ref{independence_regeneration} and the law of large numbers, we have
\[ \frac{\tau^j}{j} \xrightarrow[j \to +\infty]{a.s.} \E \left[ \tau^2-\tau^1 \right] \hspace{1cm} \mbox{ and } \hspace{1cm} \frac{H_{\tau^j}}{j} \xrightarrow[j \to +\infty]{a.s.} \E \left[ H_{\tau^2}-H_{\tau^1} \right].\]
For every $n>\tau^1$, let $j(n)$ be the index such that $\tau^{j(n)} \leq n < \tau^{j(n)+1}$. Then we have $\frac{j(n)}{n} \to \E [\tau^2-\tau^1]^{-1}$ a.s.. Moreover, we have $H_{\tau^{j(n)}} \leq H_{n} \leq H_{\tau^{j(n)+1}}$ by the definition of regeneration times, so $\frac{H_n}{j(n)} \to \E[H_{\tau^2}-H_{\tau^1}]$ a.s.. The result follows, with
\[v_{\mu}=\frac{\E \left[ H_{\tau^2}-H_{\tau^1} \right]}{\E \left[ \tau^2-\tau^1 \right]}>0.\]
\end{proof}

\begin{proof}[Proof of Proposition \ref{regeneration_expectation}]
We will actually show that $\tau^1$ has a subpolynomial tail, i.e. for every $\beta>0$, we have
\[\P \left( \tau^1 > n \right) = o ( n^{-\beta} ).\]
We first need to introduce a few notation. We define by induction stopping times $\tau_j$ and $\tau'_j$ for every $j \geq 1$:
\begin{itemize}
\item
$\tau_1=\inf \{n | H_n>0\}$,
\item
$\tau'_j=\inf \{n \geq \tau_j | H_n < H_{\tau_j} \}$ for every $j \geq 1$,
\item
$\tau_{j+1}=\inf \left\{ n>\tau'_j | H_n > \max \left( H_0, H_1, \dots, H_{\tau'_j} \right) \right\}$.
\end{itemize}
Let also $J$ be the largest index such that $\tau_J<+\infty$. We claim that $J$ is a geometric variable.
Indeed, on the one hand, we know that $H_n \to + \infty$ when $n \to +\infty$, so almost surely, if $\tau'_j<+\infty$, then $\tau_{j+1}<+\infty$. On the other hand, if $\tau_j<+\infty$, let $\mathcal{F}_{\tau_j}$ be the $\sigma$-algebra generated by $\bcau_{0, H_{\tau_j}}$ and $(X_0, X_1, \dots, X_{\tau_j})$.
Then the variable
\[\left( \bcau_{ H_{\tau_j}, \infty}, ( X_{\tau_j+i} )_{0 \leq i \leq \tau'_j-\tau_j} \right)\]
is independent of $\mathcal{F}_{\tau_j}$ and has the same distribution as $\left( \bcau_{ 1, \infty}, ( X_{\tau_1+i} )_{0 \leq i \leq \tau'_1-\tau_1} \right)$.
In particular, if $\tau_j<+\infty$, we have
\[\P \left( \tau'_j<+\infty | \mathcal{F}_{\tau_j} \right)=\P \left( \tau'_1<+\infty \right)=\P \left( T_{\partial}=+\infty \right),\]
so $\P \left( \tau_{j+1}<+\infty | \tau_j<+\infty \right)$ does not depend on $j$. This shows that $J$ is a.s.~finite and geometric. Note that $\tau^1=\tau_J$. For any $n>0$, we also denote by $J_n$ the largest index $j$ such that $\tau_j \leq n$. Finally, we recall that $D_n$ is the greatest "descent" of $X$ before time $n$.

In order to estimate the tail of $\tau^1$, we partition the event $\{\tau^1>n\}$ into several "bad" events. Let $\delta>0$ be small (we will actually only need $\delta<1/3$). We have
\begin{align} \label{decomposition_tail_tau}
\P \left( \tau^1>n \right) &= \P \left( \tau^1 \ne \tau_{J_n} \right) \nonumber \\
& \leq  \P \left( J_n \geq n^{\delta}\right) + \P \left( D_n \geq n^{\delta} \right) + \P \left( J_n<n^{\delta}, D_n<n^{\delta}, \tau^1 \ne \tau_{J_n} \right).
\end{align}
We now bound these terms one by one. First, we know that $J_n \leq J$, which is a geometric variable. Hence, the first term is at most $\exp(-cn^{\delta})$ for some constant $c$, so it is $o(n^{-\beta})$ for any $\beta>0$. Moreover, the second part of Proposition \ref{prop_quasi_positive} shows that the second term is $o(n^{-\beta})$ as well.

Finally, we study the third term of \eqref{decomposition_tail_tau}. We first show that if $D_n<n^{\delta}$ and $J_n<n^{\delta}$, then $H_{\tau_{J_n}}<n^{2\delta}$ (this is a deterministic statement). 
If $D_n<n^{\delta}$, let $1 \leq j < J_n$. We have $\tau_{j+1} \leq n$, so $\tau'_j \leq n$ and $H_{\tau'_j}=H_{\tau_j}-1$ by the definition of $\tau'_j$. By the definitions of $\tau_{j+1}$ and of $D_n$, we have
\[H_{\tau_{j+1}}-H_{\tau_j} = 1+\max_{[0, \tau'_j]} H - (H_{\tau'_j}+1) \leq D_n <n^{\delta}.\]
By summing over $j$ (and remembering $H_{\tau_1}=1$), we obtain
\[H_{\tau_{J_n}} \leq 1+n^{\delta}(J_n-1)<1+n^{\delta} (n^{\delta}-1)<n^{2 \delta}.\]
Therefore, if the event in the third term of \eqref{decomposition_tail_tau} occurs, we have $H_{\tau_{J_n}}<n^{2\delta}$ but $\tau'_{J_n}<+\infty$, so there is $k>n$ such that $H_k<n^{2\delta}$. On the other hand, if $\delta<1/3$, we have
\[ \P \left( \exists k > n, H_k \leq n^{2\delta} \right) \leq  \P \left( \exists k > n, H_k \leq k^{1-\delta} \right) \leq \sum_{k > n} \P \left( H_k \leq k^{1-\delta} \right)=\sum_{k>n} o \left( k^{-(\beta+1)} \right)\]
by the first point of Proposition \ref{prop_quasi_positive}. This proves that the third term of \eqref{decomposition_tail_tau} decays superpolynomially, which concludes the proof.
\end{proof}

\section{Counterexamples and open questions}\label{causal_sec_counter}

We finally discuss the necessity of the various assumptions made in the results of this paper, and we state a few conjectures. See Figure \ref{summary_conjectures} for a quick summary.

\paragraph{Liouville property.}
We first note that if we do not require the strips $(S_i)$ to be i.i.d., then Theorem \ref{thm_2_bis} fails. Indeed, we start from $\cau(\mathbf{T})$ and choose a ray $\gamma_0$ of $\mathbf{T}$. We then duplicate many times the horizontal edges to add a very strong lateral drift towards $\gamma_0$. If we also duplicate the edges of $\gamma_0$ enough times, we can make sure that the simple random walk eventually stays on the path $\gamma_0$. This yields a map of the form $\m \left( \mathbf{T}, (S_i) \right)$ which has the intersection property, so it is Liouville.

\paragraph{Poisson boundary.}
The description of the Poisson boundary given by Theorem \ref{thm_2_Poisson} cannot be true for any map of the form $\m \left( \mathbf{T}, (S_i) \right)$, even if the strips $(S_i)$ are i.i.d.. Indeed, it is possible to choose $S_i$ such that the walk $(X_n)$ has a positive probability to stay in $S_i$ forever, and such that $S_i$ itself has a non-trivial Poisson boundary. In this case, the Poisson boundary of $\m \left( \mathbf{T}, (S_i) \right)$ is larger than $\widehat{\partial} \mathbf{T}$, and nonatomicity in Theorem \ref{thm_2_Poisson} is false. On the other hand, we conjecture that if we furthermore assume that all the slices $S_i$ are recurrent graphs (and i.i.d.), then $\widehat{\partial} \mathbf{T}$ is a realization of the Poisson boundary. As explained in Remark \ref{rem_poisson_general_setting}, our arguments cannot handle this general setting.

\paragraph{Positive speed.}
The positive speed is also false in general maps of the form $\m \left( \mathbf{T}, (S_i) \right)$ if the strips $S_i$ are too large and do not add vertical drift. For example, if they are equal to the half-planar regular triangular lattice, then the random walk will spend long periods in the same strip, where it has speed zero. On the other hand, we conjecture that the assumption $\mu(0)=0$ is not necessary in Theorem \ref{thm_3_positive}.

As for Galton--Watson trees, another process of interest on the maps $\cau(T)$ is the $\lambda$-biased random walk $X^{\lambda}$. If a vertex $x$ has $c(x)$ children and $X_n^{\lambda}=x$, then $X_{n+1}^{\lambda}$ is equal to $y$ with probability $\frac{1}{c(x)+3\lambda}$ for every child $y$ of $x$, and to $z$ with probability $\frac{\lambda}{c(x)+3\lambda}$ if $z$ is the parent or one of the two neighbours of $x$.

If $\lambda>1$, we expect that, whether $\mu(0)=0$ or not, the process behaves in the same way as on trees \cite{LPP96}: the walk is recurrent for $\lambda>\lambda_c$ (as easily shown by the Nash--Williams criterion) and should have positive speed for $\lambda<\lambda_c$, where $\lambda_c= \sum i \mu(i)$. If $\lambda<1$ and $\mu(0)=0$, it is easy to see that the speed is positive on $\cau(T)$ since the drift at every vertex is positive. For $\lambda<1$ and $\mu(0)>0$, the $\lambda$-biased walk on $T$ has speed zero for $\lambda$ small enough ($\lambda \leq f'(q)$, where $q$ is the extinction probability of $T$ and $f$ the generating function of $\mu$). We believe that this regime disappears on causal maps, and that the $\lambda$-biased walk on $\cau(T)$ has positive speed for every $\lambda<1$.

\paragraph{Other properties of the simple random walk (for $\mu(0)=0$).}
As shown by Theorem \ref{thm_2_Poisson}, the harmonic measure of $\cau(T)$ on $\widehat{\partial} T$ is a.s.~nonatomic and has full support. It would be interesting to investigate finer properties of this measure,  as it has been done for Galton--Watson trees \cite{LPP95, Lin17}. We believe that as for Galton--Watson trees, the harmonic measure is not absolutely continuous with respect to the mass measure, and should satisfy a dimension drop.

Another quantity of interest related to the simple random walk is the heat kernel decay, i.e. the probability of returning to the root at time $n$. Perhaps surprisingly, the annealed and quenched heat kernels might have different behaviours: if $\mu(1)>0$, the possibility that $T$ does not branch during the first $n^{1/3}$ steps gives an annealed lower bound of order $e^{-n^{1/3}}$. On the other hand, the worst possible traps after the first branching points seem to be large portions of square lattice, which yield a quenched lower bound of order $e^{-n^{1/2}}$. Our argument for quasi-positive speed could be adapted to prove that the heat kernel decays quicker than any polynomial, which seems far from optimal. On the other hand, a natural first step to show that the lower bounds are tight would be to prove anchored expansion for $\cau(T)$. However, this property does not seem well suited to the study of causal maps since connected subsets of $\cau(T)$ can be quite nasty.

\paragraph{Other random processes.}
Finally, other random processes such as percolation on $\cau(T)$ might be investigated. We expect that we should have $p_c<p_u$, i.e. there is a regime where infinitely many infinite components coexist, as it is generally conjectured for graphs with a hyperbolic behaviour (like for example nonamenable transitive graphs \cite{BS96}). We note that oriented percolation is studied in a work in progress of David Marchand.

More generally, for unimodular, planar graphs, other notions of hyperbolicity (including $p_c<p_u$) have been studied in \cite{AHNR16} and proved to be equivalent to each other. It might be interesting to study the relation with our setting: if it is true that any hyperbolic (in the sense of \cite{AHNR16}) unimodular map contains a supercritical Galton--Watson tree, then our results of Section \ref{causal_sec_metric} apply. On the other hand, it is clear that every unimodular planar map containing a supercritical Galton--Watson tree is hyperbolic in the sense of \cite{AHNR16}.

\appendix

\section{The regeneration structure}\label{causal_appendix_renewal}

The goal of this appendix is to prove Lemma \ref{independence_regeneration}. We recall that $\partial \hcau$ is the set of vertices at height $-1$, and that $T_{\partial}$ is the first time at which $X$ hits $\partial \hcau$. We will first prove the following intermediate result.

\begin{lem}\label{before_after_tau1}
\begin{enumerate}
\item
We have $\tau^1<+\infty$ a.s..
\item
The path-decorated map $\left( \bcau_{H_{\tau^{1}}, \infty}, (X_{\tau^1+i})_{i \geq 0} \right)$ is independent of $\left( \bcau_{0, H_{\tau^1}}, (X_i)_{0 \leq i \leq \tau^1} \right)$ and has the same distribution as $\left( \hcau, (X_i)_{i \geq 0} \right)$ conditioned on the event $\{ T_{\partial}=+\infty \}$.
\end{enumerate}
\end{lem}

Note that the first point follows from Proposition \ref{regeneration_expectation}, so we only need to focus on the second point.

\begin{proof}[Proof of Lemma \ref{before_after_tau1}]
We first note that, by Proposition \ref{prop_quasi_positive}, we have $H_n \to +\infty$ a.s., so the conditioning on $\{ T_{\partial}=+\infty \}$ is non-degenerate.

For every $h \geq 0$, let $T_h=\min \{n \geq 0 | H_n=h\}$, and let $T'_h=\min \{n \geq T_h | H_n<h\}$. By Proposition \ref{prop_quasi_positive}, we have $T_h<+\infty$ a.s.. We also know that the rooted map $\left( \bcau_{h, \infty}, X_{T_h} \right)$ is independent of $\left( \bcau_{0,h}, (X_i)_{0 \leq i \leq T_h}\right)$ and has the same distribution as $\hcau$. Therefore, the path-decorated map $\left( \bcau_{h, \infty}, (X_{T_h+i})_{0 \leq i \leq T'_h-T_h} \right)$ is independent of $\left( \bcau_{0,h}, (X_i)_{0 \leq i \leq T_h}\right)$ and has the same distribution as $\left( \hcau, (X_i)_{0 \leq i \leq T_{\partial}} \right)$.

It follows that, for any two measurable sets $A$ and $B$ of path-decorated maps, we have
\begin{align*}
& \P \left( \left( \bcau_{0, H_{\tau^1}}, (X_i)_{0 \leq i \leq \tau^1} \right) \in A \mbox{ and } \left( \bcau_{H_{\tau^1}, \infty}, (X_{\tau^1+i})_{i \geq 0} \right) \in B \right)\\
&= \sum_{h \geq 0} \P \left( \left( \bcau_{0, h}, (X_i)_{0 \leq i \leq T_h} \right) \in A \mbox{ and } \left( \bcau_{h, \infty}, (X_{T_h+i})_{i \geq 0} \right) \in B \mbox{ and } \tau^1=h \right)\\
&= \sum_{h \geq 0} \P \, \bigl( \left( \bcau_{0, h}, (X_i)_{0 \leq i \leq T_h} \right) \in A \mbox{ and } \forall i<h, T'_i \leq T_h\\[-10pt]
& \hphantom{= \sum_{h \geq 0} \P \, \bigl( \left( \bcau_{0, h}, (X_i)_{0 \leq i \leq T_h} \right) \in A} \,\, {} \mbox{ and } ( \bcau_{h, \infty}, (X_{T_h+i})_{i \geq 0} ) \in B \mbox{ and } T'_h=+\infty \bigr),
\end{align*}
by noting that $H_{\tau^1}$ is the smallest height $i$ such that $T'_i=+\infty$. Note that the event $\{ \forall i<h, T'_i \leq T_h \}$ is a measurable function of $\left( \bcau_{0, h}, (X_i)_{0 \leq i \leq T_h} \right)$, and the event $\{ T'_h=+\infty \}$ is a measurable function of $\left( \bcau_{h, \infty}, (X_{T_h+i})_{i \geq 0} \right)$. Hence, by the independence and the distribution of $\left( \bcau_{h, \infty}, (X_{T_h+i})_{i \geq 0} \right)$ found above, we have
\begin{align*}
& \P \left( \left( \bcau_{0, H_{\tau^1}}, (X_i)_{0 \leq i \leq \tau^1} \right) \in A \mbox{ and } \left( \bcau_{H_{\tau^1}, \infty}, (X_{\tau^1+i})_{i \geq 0} \right) \in B \right)\\
&= \sum_{h \geq 0} \P \left( \left( \bcau_{0, h}, (X_i)_{0 \leq i \leq T_h} \right) \in A \mbox{ and } \forall i<h, T'_i \leq T_h \right) \P \left( \left( \bcau_{0, \infty}, (X_i)_{i \geq 0} \right) \in B \mbox{ and } T_{\partial}=+\infty \right)\\
&= \P \left( \left( \bcau_{0, \infty}, (X_i)_{i \geq 0} \right) \in B \big| T_{\partial}=+\infty \right) f(A),
\end{align*}
where $f(A)$ is a function of $A$. Therefore, the path-decorated maps $\left( \bcau_{0, H_{\tau^1}}, (X_i)_{0 \leq i \leq \tau^1} \right)$ and $\left( \bcau_{H_{\tau^1}, \infty}, (X_{\tau^1+i})_{i \geq 0} \right)$ are independent and, by taking $A=\Omega$, we obtain that the distribution of the second is a multiple of the distribution of $\left( \hcau, (X_i)_{i \geq 0} \right)$ conditioned on the event $\{ T_{\partial}=+\infty\}$. Since both are probability measures, they coincide.
\end{proof}

\begin{proof}[Proof of Lemma \ref{independence_regeneration}]
We define the shift operator $\theta$ as follows:
\[ \left( \hcau, (X_i)_{i \geq 0} \right) \circ \theta=\left( B_{H_{\tau^1}, \infty} (\hcau), (X_{\tau^1+i})_{i \geq 0} \right).\]

We first notice that Lemma \ref{before_after_tau1} remains true if we consider $(\hcau, X)$ under the measure $\P \left( \cdot | T_{\partial}=+\infty \right)$ instead of $\P$. Indeed, conditioning on an event of positive probability does not change the fact that $\tau^1<+\infty$ a.s.. Moreover, the event $\{ T_{\partial}=+\infty \}$ only depends on $\left( \bcau_{0, H_{\tau^1}}, (X_i)_{0 \leq i \leq \tau^1} \right)$ and not on $\left( \bcau_{H_{\tau^1}, \infty}, (X_{\tau^1+i})_{i \geq 0} \right)$, so conditioning on this event affects neither the independence of these two path-decorated maps, nor the distribution of the second.

But by Lemma \ref{before_after_tau1}, the map $\left( \hcau, (X_i)_{i \geq 0} \right) \circ \theta$ has the same distribution as $\left( \hcau, (X_i)_{i \geq 0} \right)$ under $\P \left( \cdot | T_{\partial}=+\infty \right)$, so Lemma \ref{before_after_tau1} applies after composition by $\theta$. In particular, we have $\tau^1 \circ \theta <+\infty$ a.s., i.e. $\tau^2<+\infty$ a.s.. Moreover, the two following path-decorated maps are independent:
\begin{itemize}
\item
$\left( \bcau_{0, H_{\tau^1}}, (X_i)_{0 \leq i \leq \tau^1} \right) \circ \theta = \left( \bcau_{H_{\tau^1}, H_{\tau^2}}, (X_{\tau^1+i})_{0 \leq i \leq \tau^2-\tau^1} \right)$,
\item
$\left( \bcau_{H_{\tau^1}, \infty}, (X_{\tau^1+i})_{i \geq 0} \right) \circ \theta = \left( \bcau_{H_{\tau^2}, \infty}, (X_{\tau^2+i})_{i \geq 0} \right)$,
\end{itemize}
and the second one has the same distribution as $\left( \hcau, (X_i)_{i \geq 0} \right)$ under $\P \left( \cdot | T_{\partial}=+\infty \right)$. From here, an easy induction on $j$ shows that for any $j \geq 1$, we have $\tau^j<+\infty$ and the path-decorated map $\left( \bcau_{H_{\tau^j}, H_{\tau^{j+1}}}, (X_{\tau^j+i})_{0 \leq i \leq \tau^{j+1}-\tau^j} \right)$ has indeed the right distribution and is independent of $\left( \bcau_{H_{\tau^{j+1}}, \infty}, (X_{\tau^{j+1}+i})_{i \geq 0} \right)$. This proves Lemma \ref{independence_regeneration} (the third item is a direct consequence of the first two).
\end{proof}

\bibliographystyle{abbrv}
\bibliography{bibli}

\end{document}